\documentclass[reqno,11pt]{amsart}  

\usepackage{amsmath}
\usepackage{amssymb} 
\usepackage{amsthm}
\usepackage{mathtools}   
\usepackage{graphicx}
\usepackage{graphics}  
\usepackage[pdftex]{color}
\usepackage{color}
\usepackage{verbatim} 
\usepackage[normalem]{ulem} 
\usepackage[mathscr]{eucal}
\usepackage{upgreek}
\usepackage{enumerate} 
\usepackage{color}
\usepackage{dsfont}
\usepackage{tikz}
\usetikzlibrary{decorations.pathreplacing,angles,quotes}
\usepackage{subcaption}
 \usepackage{scalerel}
 
\usepackage[titletoc]{appendix}

\DeclareRobustCommand{\bigO}{%
  \text{\usefont{OMS}{cmsy}{m}{n}O}%
}

\numberwithin{equation}{section}

\newtheorem{theorem}{Theorem}[section] 
\newtheorem{lemma}[theorem]{Lemma} 
\newtheorem{prop}[theorem] {Proposition} 
 
\newtheorem{remark}[theorem]  {Remark}

\theoremstyle{definition}

 \usepackage{float}
\floatstyle{boxed}
\restylefloat{figure}

\DeclareMathAlphabet{\mathpzc}{OT1}{pzc}{m}{it}

\DeclarePairedDelimiter{\abs}{\lvert}{\rvert}
\DeclarePairedDelimiter{\norm}{\lVert}{\rVert}

\newcommand{\blambda}{\boldsymbol{\lambda}}
\newcommand{\bmu}{\boldsymbol{\mu}}

\renewcommand{\L} {\Lambda} 

\newcommand{\brho}{\boldsymbol{\rho}}

\def\d{\delta} 
\newcommand{\e} {\varepsilon} 
\newcommand{\eps}{\varepsilon}

\newfam\Bbbfam 
\font\tenBbb=msbm10 
\font\sevenBbb=msbm7 
\font\fiveBbb=msbm5 
\textfont\Bbbfam=\tenBbb 
\scriptfont\Bbbfam=\sevenBbb 
\scriptscriptfont\Bbbfam=\fiveBbb

\newcommand{\R}     {\mathbb{R}} 
\newcommand{\Z}     {\mathbb{Z}} 
\newcommand{\N}     {\mathbb{N}} 
\renewcommand{\P}   {\mathbb{P}} 
 
\newcommand{\E}     {\mathbb{E}}

\def\1{{\mathchoice {1\mskip-4mu\mathrm l}      
{1\mskip-4mu\mathrm l} 
{1\mskip-4.5mu\mathrm l} {1\mskip-5mu\mathrm l}}} 
\newcommand{\ssup}[1] {{\scriptscriptstyle{({#1}})}}

\newcommand{\PMF} {\ssup{\scaleto{\scriptscriptstyle{PMF}}{2.55pt}}}

\newcommand{\CMF} {\ssup{\scaleto{CMF}{2.55pt}}}

\newcommand{\HYL} {\ssup{\scaleto{HYL}{2.55pt}}}

\newtheoremstyle{thm}{2ex}{2ex}{\itshape\rmfamily}{} 
{\bfseries\rmfamily}{}{1.7ex}{} 
 
\newtheoremstyle{rem}{1.3ex}{1.3ex}{\rmfamily}{} 
{\itshape\rmfamily}{}{1.5ex}{} 
 
\newenvironment{proofsect}[1] 
{\vskip0.1cm\noindent{\scshape #1.}\hskip0.5cm} 

 \def\e{{\rm e}}

\newcommand{\Ccal}   {{\mathcal C }}

\newcommand{\Hcal}   {{\mathcal H }}

\newcommand{\Ncal}   {{\mathcal N }}

\newcommand{\Ucal}   {{\mathcal U }}

 \newcommand{\ex}{{\rm e}} 
 
\renewcommand{\d}{{\rm d}} 
\newcommand{\per}{{\text{\rm per}}} 
 
\newcommand{\Li}{\operatorname{Li}\,}
\newcommand{\Leb}{{\rm Leb}} 
\newcommand{\Sym}{\mathfrak{S}}

\newcommand{\Dir}{{\operatorname {Dir}\,}}

\newcommand{\Exp}{\mathscr{E}\kern-0.2mm{\operatorname{xp}}}
\newcommand{\Log}{\mathscr{L}\kern-0.2mm{\operatorname{og}}}

\newcommand{\heap}[2]{\genfrac{}{}{0pt}{}{#1}{#2}} 
\newcommand{\bc}{{\operatorname{bc}}}
\renewcommand{\emptyset} {\varnothing}

\makeindex

\begin{document}

\title[\hfill \hfill]
{Large deviation analysis for classes of interacting  Bosonic cycle counts}


\author{Stefan Adams and Matthew Dickson}
\address{Mathematics Institute, University of Warwick, Coventry CV4 7AL, United Kingdom}
\email{S.Adams@warwick.ac.uk, M.Dickson.1@warwick.ac.uk}

\thanks{M.D. is supported by EPSRC as part of the MASDOC CDT at the University of Warwick, Grant No. EP/HO23364/1}
  
\date{}

\subjclass[2000]{Primary 60F10; 60J65; 82B10; 81S40}
 
\keywords{Marked Poisson point process, large deviations, empirical cycle counts, variational formula, pressure representation, Bose-Einstein condensation (BEC)}  


\begin{abstract}
This paper studies  probabilistic mean-field models for interacting bosons at a positive temperature in the thermodynamic limit with random particle density. In particular, we prove large deviation principles for empirical cycle counts in all our models, and as such generalise recent work in \cite{ACK} where upper and lower large deviation bounds do not match. Namely, on the one hand we generalise to the so-called \textit{grand canonical ensemble}, and on the other hand, we consider classes of interaction potentials depending on the cycle counts  which are not restricted to be positive. 
Our  large deviation results provide representation formulae for the thermodynamic limit of the pressure which in turn leads to proving \textit{Bose-Einstein-condensation} (BEC)  in all our models.
 A primary focus and novelty is the  pressure representation  via  extended large deviation analysis  for the so-called hard-sphere, or HYL-model (Huang-Yang-Luttinger) studied in \cite{Lew86}. This model has negative counter terms in the Hamiltonian and shows BEC depending on the coupling constants.

\end{abstract} 

\maketitle

\section{Introduction and main results}\label{Intro}

\noindent In this paper, we study  probabilistic mean-field models for interacting bosons at positive temperature in the thermodynamic limit with random particle density. See Section~\ref{sec-BEC} for the physical background. Our aim is to prove large deviation principles for empirical cycle counts and to prove the so-called Bose-Einstein condensation (BEC) phase transition in terms of the behaviour of the minimisers (zeroes) of our rate functions. This study is a direct continuation of recent work in \cite{ACK}, namely, on the one hand we generalise to the so-called \textit{grand canonical ensemble}, and on the other hand we consider classes of interaction potentials depending solely on the cycle counts which are not restricted to be positive. This approach allows us to prove BEC for our models as well as improving the analysis in \cite{ACK} in that we obtain matching upper and lower large deviation bounds. The main focus and novelty is to obtain configurational cycle weights and to obtain a detailed and extended large deviation analysis  for the so-called hard-sphere or HYL-model (Huang-Yang-Luttinger) studied in \cite{Lew86}. In Section~\ref{sec-BEC} we provide discussion on how our results relate to the ones in \cite{BLP} and \cite{BCMP05}.

\subsection{The model}\label{sec-model} The main object is 
the following symmetrised sum of Brownian bridge expectations,
\begin{equation}\label{defpartition}
Z_\L^{\ssup{\rm bc}}(\beta,\mu)=\sum_{N=0}^\infty \frac{\ex^{\beta\mu N}}{N!}\sum_{\sigma\in\Sym_N}\int_\L\d x_1\cdots\int_\L\d x_N\bigotimes_{i=1}^N\bmu_{x_i,x_{\sigma(i)}}^{\ssup{\rm bc,\beta}}\Big[\exp\Big\{-\sum_{1\le i<j\le N}\int_0^\beta v(|B_s^{\ssup{i}}-B_s^{\ssup{j}}|)\,\d s\Big\}\Big].
\end{equation} 
Here $ \bmu_{x,y}^{\ssup{\rm bc,\beta}} $ is the canonical Brownian bridge measure with boundary condition ${\rm bc }\in\{ \emptyset, {\rm per},{\rm Dir}\} $, time horizon $ \beta>0 $, initial point $ x\in\L $ and terminal point $ y\in\L $. The sum is on permutations $ \sigma\in\Sym_N $ of $ 1,\dots,N$ and $\mu\in\R $ is the \textit{chemical potential}. The {\it interaction potential} $ v\colon\R\to[0,\infty] $ is measurable, decays sufficiently fast at infinity and is possibly infinite close to the origin.  We assume that $\L$ is a measurable subset of $\R^d$ with finite volume.   The boundary condition ${\rm bc }=\emptyset $ refers to the standard Brownian bridge, whereas for $ {\rm bc }={\rm Dir } $, the expectation is on those Brownian bridge paths which stay in $ \L $ over the time horizon $[0,\beta]$. In the case of periodic boundary condition, $ {\rm bc }={\rm per }$, we consider Brownian bridges on the torus $\L= (\R/L\Z)^d $ with side length $L$.
 
Our main motivation to study the quantity $Z_\L^{\ssup{\rm bc}}(\beta,\mu)$ is the fact that it is related to all $N$-body Hamilton operators
\begin{equation}\label{Hdef}
\Hcal_{N,\L}^{\ssup{\rm bc}} =-\sum_{i=1}^N \Delta_{i}^{\ssup{\rm bc}} +\sum_{1\leq i<j\leq N} v(|x_i-x_j|),\qquad x_1,\dots,x_N\in\L, \qquad \rm {bc}\in\{{\rm Dir }, {\rm per}\}, N\in\N,
\end{equation}
and $N$-particle Hilbert spaces $ L_2(\L)^{\otimes N} $, 
where  $\Delta_{i}^{\ssup{\rm bc}}$ stands for the Laplacian with bc boundary condition. 
More precisely, $Z_\L^{\ssup{\rm bc}}(\beta,\mu)$ is equal to the trace of the projection of the direct sum of the operators $\exp{\{-\beta \Hcal_{N,\L}^{\ssup{\rm bc}}\}}$ to the set of all symmetric (i.e., permutation invariant) functions $(\R^d)^N\to\R, N\in\N$. This statement is proven via the Feynman-Kac formula, see \cite{G70} or \cite{BR97}. Hence, we call $Z_\L^{\ssup{\rm bc}}(\beta,\mu)$ a partition function. The cycle expansion of the partition function in Section~\ref{sec-notation} for vanishing interaction provides  the cycle weights for our reference Poisson point process.  

It is the main purpose of this paper to derive variational expressions for the {\it limiting pressure} 
\begin{equation}\label{limpressure}
p^{\ssup {\rm bc}}(\beta,\mu)=\frac{1}{\beta}\lim_{\L\uparrow\R^d}\frac{1}{|\L|}\log Z_{\L}^{\ssup {\rm bc}}(\beta,\mu),
\end{equation}
where $\beta\in(0,\infty), \mu\in\R $, $d\in\N$ and  ${\rm bc }\in\{\emptyset, {\rm per},{\rm Dir }\}$ for classes of interacting models. The existence of the thermodynamic limit in \eqref{limpressure} with $ {\rm bc}\in\{ {\rm per},{\rm Dir }\} $ under suitable assumptions on the interaction potential $v$ can be shown by standard methods, see for example \cite[Th.~3.58]{Rue69} and \cite{BR97}. We achieve this via large deviation principles for empirical cycle counts under mean-field cycle weights in Section~\ref{sec-cycle}. Our large deviation results generalise previous results in  \cite{BLP} for the HYL model in the way that our empirical cycle counts provide deeper information (higher level of large deviation analysis) and that we can study all possible interaction parameter choices. The analysis of the minimiser (zeroes) of our rate functions in Section~\ref{sec-minimiser}  enables us to prove BEC in our various models using the particle density of the minimiser and the derivative of the pressure.

Our approach and the remainder of Section~\ref{Intro} can be summarized as follows. Since any permutation decomposes into cycles, and using the Markov property, the family of the $N$ bridges in \eqref{defpartition} decomposes into cycles of various lengths, i.e. into bridges that start and end at the same site, which is uniformly distributed over $\L$. We conceive these initial-terminal sites as the points of a standard Poisson point process on $\R^d$ and the cycles as marks attached to these points; see Section~\ref{sec-notation} for the relevant notation. In Proposition~\ref{lem-rewrite} below,  we rewrite $Z_{\L}^{\ssup{\rm bc}}(\beta,\mu)$ in terms of an expectation over a reference process: the marked Poisson point process $\omega_{\rm P}$. 

In Section~\ref{sec-cycle}, we present our results on the large deviation principles for empirical cycle counts. Along the way, we also obtain representations of the limiting pressure in \eqref{limpressure} for our models. The proofs for our large deviation principles are in Section~\ref{sec-LDPproofs}. In Section~\ref{sec-minimiser} we analyse all our rate functions and obtain expressions for the zeroes of the rate functions. This analysis leads to pressure representation formulae and to our results on BEC in Section~\ref{sec-pressure}.  All proofs for this analysis are in Section~\ref{sec-zeroLDP} and Section~\ref{sec-pressureBEC}, respectively.   The physical interpretation, motivation and relevance are discussed in Section~\ref{sec-BEC}.  Definitions and details for the different boundary conditions are given in Appendix~\ref{app-A}, and in Appendix~\ref{app-Bose} we summarise definition and main properties of the Bose functions, and in Appendix~\ref{app-Lambert} we give definitions and main properties of the Lambert function.

\subsection{Representation of the partition function}\label{sec-notation}

\noindent In this section, we introduce our representation of the partition function $Z_{\L}^{\ssup{\rm bc}}(\beta,\mu)$ for each boundary condition $ {\rm bc }\in\{\emptyset, {\rm per},{\rm Dir }\} $ in terms of an expectation over a marked Poisson point process. The main result of this section is Proposition~\ref{lem-rewrite}. We have to introduce some notation. 

We begin with the mark space. The space of marks is defined as
\begin{equation}
E^{\ssup{\rm bc}}=\bigcup_{k\in\N} \Ccal^{\ssup{\rm bc}}_{k,\L},  \qquad {\rm bc} \in \{\emptyset,{\rm per},  {\rm Dir}\},
\end{equation}
where, for $k\in\N$, we denote by $\Ccal_k=\Ccal^{\ssup \emptyset}_{k,\L}$  the set of continuous functions $f\colon [0,k\beta] \to \mathbb{R}^d$ satisfying $f(0)=f(k\beta)$,  equipped with the topology of uniform convergence. Moreover,
$  \Ccal^{\ssup{\rm Dir}}_{k,\L} $, (respectively $\Ccal_{k,\L}^{\ssup{\rm per}}$)  is the space of continuous functions in $ \L $ (respectively on the torus $\L=(\R \slash L \Z)^d$) with time horizon $ [0,k\beta] $.  We sometimes call the marks {\it cycles}. By  $\ell\colon E^{\ssup{\rm bc}}\to\N $  we denote the canonical map defined by $\ell(f)=k$ if $f\in \Ccal^{\ssup {\rm bc}}_{k, \L}$. We  call $\ell(f)$ the {\it length} of $f\in E$. When dealing with the empty boundary condition, we sometimes drop the superscript $\emptyset$.

We consider spatial configurations that consist of a locally finite set $\xi\subset\R^d$ of particles, and to each particle $x\in\xi$ we attach a mark $f_x\in E^{\ssup {\rm bc}}$satisfying $f_x(0)=x$. Hence, a configuration is described by the counting measure
$$
\omega=\sum_{x\in\xi}\delta_{(x,f_x)}
$$ 
on $\R^d\times E$ for the empty boundary condition (respectively on $\L \times E^{\ssup {\rm bc}}$ for ${\rm bc} \in \{{\rm per},  {\rm Dir}\}$).

We now introduce three marked Poisson point processes for the three boundary conditions. The one for the empty boundary condition will later serve as a reference process and is introduced separately first. The remaining two processes are defined in Appendix~\ref{app-A}. 

\noindent\textbf{Reference process.}

\noindent Consider on $\Ccal=\Ccal_1$ the canonical Brownian bridge measure
\begin{equation}\label{nnBBM}
\bmu^{\ssup{\emptyset, \beta}}_{x,y}(A)=\bmu^{\ssup \beta}_{x,y}(A)=\frac{\P_x(B\in A;B_\beta\in\d y)}{\d y},\qquad A\subset\Ccal\mbox{ measurable}.
\end{equation}
Here $B=(B_t)_{t\in[0,\beta]}$ is a Brownian motion in $\R^d$ with generator $\Delta$, starting from $x$ under $\P_x$. Then $ \mu^{\ssup \beta}_{x,y}$ is a regular Borel measure on $\Ccal$ with total mass equal to the Gaussian density,
\begin{equation}\label{Gaussian}
\bmu_{x,y}^{\ssup \beta}(\Ccal)=g_{\beta}(x,y)=\frac {\P_x(B_\beta\in\d y)}{\d y}=(4\pi\beta)^{-d/2}{\rm e}^{-\frac 1{4\beta}|x-y|^2}.
\end{equation}
We write $ \P_{x,y}^{\ssup \beta}=\bmu_{x,y}^{\ssup \beta}/g_{\beta}(x,y)$ for the normalized Brownian bridge measure on $\Ccal$.
 Let
$$
\omega_{\rm P} = \sum_{x \in \xi_{ \rm P}} \delta_{(x,B_x)},
$$
be a Poisson point process on $\R^d\times E$ with intensity measure equal to $ \nu $, whose projection onto $ \R^d\times\Ccal_k $ is equal to 
\begin{equation}\label{nudef}
\nu_k(\d x,\d f)=\frac{1}{k}\Leb(\d x)\otimes\ex^{\beta\mu k}\mu_{x,x}^{\ssup{k\beta}}(\d f),\qquad k\in\N.
\end{equation}
Alternatively, we can conceive $\omega_{\rm P}$ as a marked Poisson point process  on $\R^d$, based on some Poisson point process $\xi_{\rm P}$ on $\R^d$, and a family $(B_x)_{x\in \xi_{\rm P}}$ of  i.i.d.~marks, given $\xi_{\rm P}$. The intensity of $\xi_{\rm P}$ is 
\begin{equation}\label{q*def}
\overline{q}^{\ssup{\mu}}=\sum_{k\in\N}q_k^{\ssup{\mu}},\qquad\mbox{ with }\quad q_k^{\ssup{\mu}}=\frac{\ex^{\beta\mu k}}{(4\pi\beta)^{d/2}k^{1+d/2}},\qquad k\in\N.
\end{equation}
Conditionally given $\xi_{\rm P}$, the length $\ell(B_x)$ is an $\N$-valued random variable with distribution $(q_k^{\ssup{\mu}}/\overline{q}^{\ssup{\mu}})_{k\in\N}$, and, given $\ell(B_x)=k$,  $B_x$ is in distribution equal to a Brownian bridge with time horizon  $[0,k\beta]$, starting and ending at $x$. Let $\tt Q$ denote the distribution of $\omega_{\rm P}$ and denote by $\tt E$ the corresponding expectation (respectively, write $\tt Q^{\ssup{\rm bc}} $ for $ {\rm bc}\in\{{\rm per},{\rm Dir}\} $). Hence, $\tt Q$ is a probability measure on the set $\Omega$ of all locally finite counting measures on $ \R^d\times E $. Note that our reference process is a countable superposition of Poisson point processes, and, as long as $ \overline{q}^{\ssup{\mu}} <\infty $ is finite, this reference process is a Poisson point process as well. In accordance with the existence of the limiting pressure for the reference process, we consider $ \mu\le 0 $ in the following. Our other models allow positive values of the chemical potential $ \mu\in\R $.

We now formulate our first main result, a presentation of the partition function defined in \eqref{defpartition}  in $\L\subset\R^d$ with $ |\L|<\infty $ and boundary condition $ {\rm bc}\in\{\emptyset,{\rm per},{\rm Dir}\}$. We provide a  general form of this representation for future reference. We introduce a functional on $\Omega$ that expresses the interaction between particles in $ \L\subset\R^d$, or more precisely, between their marks.  Define the {\em Hamiltonian} $ H_\L\colon\Omega\to [0,\infty] $ by
\begin{equation}\label{Hamiltonian}
H_\L(\omega)=\sum_{x,y\in\xi\cap\L}T_{x,y}(\omega),\qquad\mbox{where }\omega=\sum_{x\in\xi}\delta_{(x,f_x)}\in\Omega,
\end{equation}
where we abbreviate
\begin{equation}\label{Tdef}
T_{x,y}(\omega)=\frac{1}{2}\sum_{i=0}^{\ell(f_x)-1}\sum_{j=0}^{\ell(f_y)-1}\1_{\{(x,i)\not=(y,j)\}}\int_0^\beta v(|f_x(i\beta+s)-f_y(j\beta+s)|)\,\d s, \quad \omega\in\Omega,x,y\in\xi. 
\end{equation}

The function $H_\L(\omega)$ summarises the interaction between different marks of the point process and between different legs of the same mark; here we call the restriction of a mark $f_x$ to the interval $[i\beta,(i+1)\beta)]$ with $i\in\{0,\dots,\ell(f_x)-1\}$ a leg of the mark. 
When $ {\rm bc} \in \{{\rm per},  {\rm Dir}\}$, we replace $\overline{q}^{\ssup{\mu}} $ and all $ q_k^{\ssup{\mu}} $ by $ \overline{q}^{\ssup{{\rm bc},\mu}} $ and $ q_k^{\ssup{{\rm bc},\mu}} $ respectively, defined in \eqref{qdefbc} in Appendix~\ref{app-A}. 
\begin{prop}\label{lem-rewrite}
Fix $\beta\in(0,\infty)$. Let $ v\colon[0,\infty)\to (-\infty,\infty] $ be measurable and bounded from below and let $ \L\subset\R^d$ be measurable with finite volume (assumed to be a torus for periodic boundary condition). Then, for any $ {\rm bc}\in\{\emptyset,{\rm per},{\rm Dir}\} $,
\begin{equation}\label{rewrite}
\begin{aligned}
Z_\L^{\ssup{\rm bc}}(\beta,\mu)&=\e^{\abs{\L}\overline{q}^{\ssup{{\rm bc},\mu}}}{\tt E}^{\ssup{\rm bc}}\big[{\rm e}^{-H_\L(\omega_{\rm P})}\}\big].
\end{aligned}
\end{equation}
\end{prop}
Formula \eqref{rewrite} expresses the partition function as an expectation of the Boltzmann factor. The proof of Proposition~\ref{lem-rewrite} is a direct adaptation of the proof of \cite[Proposition~1.1]{ACK} and is omitted here. This formula has also been recently obtained for Poisson loop processes on graphs, see \cite{AV17}.

\subsection{Large deviations principles for empirical cycle counts}\label{sec-cycle}
We let $ \L_N=[-N,N]^d, N\in\N $, be a sequence of boxes $ \L_N\uparrow\R^d $ as $ N\to\infty $. We shall be interested in the countable set of random variables
\begin{equation}
	\mathcal{N}_{k}(\omega)=\#\big\{x\in\xi\cap\Lambda_N\colon  \ell(f_x)=k\big\}, \qquad k\in\N,
\end{equation}
being the number of points in $ \L_N $ (whose marks do not have to be contained in $\L_N$) with mark length equal to $k$.
These are independent Poisson random variables with respective means $\abs{\L_N}q^{\ssup{\mu}}_k$. The \textit{empirical cycle count} is the empirical density of numbers of cycles of all lengths,
\begin{equation}
	\blambda_{N}(\omega) = \big(\blambda^{\ssup{k}}_N(\omega)\big)_{k\in\N}= \left(\frac{1}{\abs{\Lambda_N}}\mathcal{N}_{k}\left(\omega\right)\right)_{k\in\N},
\end{equation}
taking values in $\ell_1\left(\R_+\right)$.
Suppose
\begin{equation}
	\left(\lambda_k\right)_{k\in\N}\in \mathcal{M}_{\Lambda_N}:=\Big\{\lambda\in\R^\N_+ : \sum_{k\in\N}k\lambda_k<\infty; \left|\Lambda_N\right|\lambda_k\in\N_0\Big\}\subset\ell_1(\R_+).
\end{equation}
Then, using the independence and the Poisson nature of our reference process, we compute the probability
\begin{equation}
	{\tt Q}\big(\blambda_{N} =(\lambda_k)_{k\in\N}\big) = \prod_{k\in\N}\frac{\e^{-\abs{\Lambda_N}q^{\ssup{\mu}}_k}\left(\abs{\Lambda_N}q^{\ssup{{\rm bc},\mu}}_k\right)^{\abs{\Lambda_N}\lambda_k}}{\left(\abs{\Lambda_N}\lambda_k\right)!}.
\end{equation}
The distribution on $ \ell_1(\R_+) $ is denoted 
\begin{equation*}
	\nu_{N,\mu}^{\ssup{\bc}}:= {\tt Q}^{\ssup{\rm bc}}\circ\blambda^{-1}_{N},
\end{equation*}
which  in fact has support $\mathcal{M}_{\Lambda_N}$.  We shall derive large deviation principles for the empirical cycle counts under various measures. The result for $ (\nu_{N,\mu})_{N\in\N} $ is the case for the ideal Bose gas, that is, with vanishing interaction  $ v\equiv 0 $.

\begin{prop}[\textbf{Ideal Bose gas, grand canonical ensemble}]
\label{THM-Ideal}
	For $d\in\N,  \beta>0  $, $\mu\leq0$ and $ \bc\in\{\emptyset,\per, \Dir\} $, the sequence $(\nu_{N,\mu}^{\ssup{\bc}})_{N\in\N}$ satisfies a large deviation principle (LDP) on $\ell_1(\R)$ with rate $\abs{\Lambda_N} $ and rate function
	\begin{equation*}
		I_\mu\left(x\right) = \begin{cases}
		\sum^\infty_{k=1}x_k\left(\log\frac{x_k}{q^{\ssup{\mu}}_k} - 1\right) + \bar{q}^{\ssup{\mu}}, & \text{for }x\in\ell_1\left(\R_+\right),\\
		+\infty, & \text{otherwise}.
		\end{cases}
	\end{equation*}
\end{prop}

We shall now define our various interaction models. We restrict ourselves to interaction potentials which are functions of the empirical cycle counts. The first class are generalisations of mean-field models studied in the physics literature. In these models one adds an energy term proportional to the square of the particle/point number (density). The number of points  for the reference process equals just the number of cycles and is just $ N_{\L_N}(\omega)=\sum_{k=1}^\infty\Ncal_{k}(\omega) $. On the other hand, the number of ``physical particles'' is the total length of the marks of the particles in $ \L_N$,
\begin{equation}\label{particles}
N_{\L_N}^{\ssup{\ell}}(\omega)=\sum_{x\in\xi\cap\L_N}\ell(f_x)=\sum_{k=1}^\infty k\Ncal_{k}(\omega),
\end{equation}
because a cycle with time horizon $ k\beta $ represents exactly $k$ Boson particles. The functional $ N_{\L_N} $ is continuous with respect to the reference process and continuous in the empirical cycle counts. The number of ``physical particles'' however, is only lower semicontinuous and not upper semicontinuous (see \cite{ACK}). For $ a\ge 0 $, define the cycle-mean-field model (CMF),
\begin{equation}\label{CMF}
\begin{aligned}
H^{\CMF}(x)& =\frac{a}{2}\Big(\sum_{k=1}^\infty x_k\Big)^2,\qquad x\in\ell_1(\R), \\
\nu_{N,\mu}^{\CMF} (\d x) &= \frac{\exp(-\abs{\L_N}\beta H^{\CMF}(x))}{Z_N^{\CMF}(\beta,\mu)}\nu_{N,\mu}^{\ssup{\bc}}(\d x) ,\quad \mu\le 0,
\end{aligned}
\end{equation}
with partition function 
$$  
Z_N^{\CMF}(\beta,\mu)=\E_{\nu_{N,\mu}^{\ssup{\bc}}}\big[\e^{-\abs{\L_N}\beta H^{\CMF}(x)}\big]={\tt E}^{\ssup{\bc}}\Big[\exp\big\{-\frac{\beta a}{2\abs{\L_N}}\big(N_{\L_N})^2\big\}\Big],
$$
and the particle-mean-field model (PMF), 
\begin{equation}\label{PMF}
\begin{aligned}
H^{\PMF}_\mu(x)& =  -\mu\sum_{k=1}^\infty k  x_k +    \frac{a}{2}\Big(\sum_{k=1}^\infty kx_k\Big)^2,\qquad x\in\ell_1(\R), \\
\nu_{N,\mu,\alpha}^{\ssup{\scaleto{PMF}{2.55pt}}} (\d x) &= \frac{\exp(-\abs{\L_N}\beta H^{\PMF}_\mu(x))}{Z_N^{\PMF}(\beta,\mu,\alpha)}\nu_{N,\alpha}^{\ssup{\bc}}(\d x),\quad \alpha\le 0, \mu\in\R, 
\end{aligned}
\end{equation}
with partition function 
$$
\begin{aligned}
Z_N^{\PMF}(\beta,\mu,\alpha)=\E_{\nu_{N,\alpha}^{\ssup{\bc}}}\big[\e^{-\abs{\L_N}\beta H^{\PMF}_\mu(x)}\big]={\tt E}^{\ssup{\bc}}\Big[\exp\big\{\beta\mu N_{\L}^{\ell}-\frac{\beta a}{2\abs{\L_N}}\big( N_{\L_N}^{\ell}\big)^2\big\}\Big].
\end{aligned}
$$
The following large deviation principles hold.

\begin{theorem}[\textbf{Large deviations principle  for CMF  models}]\label{THM-CMF}
For any $ d\in\N,  a>  0 $ and $ \mu\leq0$ and $ \bc\in\{\emptyset,\Dir,\per\} $ the following holds. The sequence $ \big(\nu_{N,\mu}^{\CMF}\big)_{N\ge 1}   $ satisfies an LDP on $ \ell_1(\R) $ with rate $ \abs{\L_N} $ and rate function 
\begin{equation}
I^{\CMF}_\mu(x) = \beta H^{\CMF}(x)+ I_\mu(x) - \inf_{y\in \ell_1\left(\R\right)}\{\beta H^{\CMF}(y)+I_\mu(y)\}.
\end{equation}
\end{theorem}

In the following, we denote  $ H^\PMF_{\mu,l.s.c.} $ the    lower semicontinuous regularisation of $ H^\PMF_{\mu} $ and  $D(x)=\sum_{k=1}^\infty k x_k $ the \textit{density} for any sequence $ x\in\ell_1(\R_+) $, and we write $ I\equiv I_0 $. 

\begin{theorem}[\textbf{Large deviation principle for PMF models}]\label{THM-PMF}
For any $ d\in\N,  a>  0, \alpha\le 0 $, $ \mu\in\R$, and $ \bc\in\{\emptyset,\Dir,\per\} $    the following holds. The sequence $ \big(\nu_{N,\mu,\alpha}^{\PMF}\big)_{N\ge 1}   $ satisfies an LDP on $ \ell_1(\R) $ with rate $ \abs{\L_N} $ and rate function 
\begin{equation}
I^\PMF_{\mu,\alpha}(x) = \beta H^\PMF_{\mu+\alpha,l.s.c.}(x)+ I(x) - \inf_{y\in \ell_1\left(\R\right)}\{\beta H^{\PMF}_{\mu+\alpha,l.s.c.}(y)+I(y)\},
\end{equation}
with
\begin{equation}\label{lscreg}
H^\PMF_{\mu,l.s.c.}(x)=H^\PMF_\mu(x)-\frac{1}{2a}\big(\mu-aD(x)\big)_+^2=\begin{cases} -\mu D(x)+\frac{a}{2}D(x)^2 &,D(x)\ge \frac{\mu}{a},\\-\frac{\mu^2}{2a}& ,D(x)<\frac{\mu}{a}.\end{cases}
\end{equation}

\begin{remark}
For $ \mu+\alpha\le 0 $,  the rate function in Theorem~\ref{THM-PMF} reads
$$
I^{\PMF}_{\mu,\alpha}(x) = \beta H^{\PMF}_{\mu+\alpha}(x)+ I(x) - \inf_{y\in \ell_1\left(\R\right)}\{\beta H^{\PMF}_{\mu+\alpha}(y)+I(y)\}.
$$
\hfill $ \diamond $
\end{remark}

\end{theorem}

The major novelty of our large deviation analysis concerns an substantial extension of the so-called HYL-model (Hunag-Yang-Luttinger-model) studied in \cite{BLP}. On one hand we replace the cycle weights originating from the energy representation in \cite{BLP} by our cycle weights stemming from spatial representation of the partition function, and on the other hand we obtain higher level large deviation principles allowing a detailed insight in the structure of the minimiser and possible phases and phase transitions. 
For any $ a\ge b> 0 $ and any $ \alpha\le 0, \mu\in\R $, define the HYL-model by
\begin{equation}
\begin{aligned}
H^\HYL_\mu(x)&=-\mu\sum_{k=1}^\infty k x_k +\frac{a}{2}\Big(\sum_{k=1}^\infty k x_k\Big)^2-\frac{b}{2}\sum_{k=1}^\infty k^2 x_k^2,\qquad x\in\ell_1(\R_+),\\
\nu^\HYL_{N,\mu,\alpha}(\d x)&=\frac{\exp(-\beta\abs{\L_N} H^\HYL_\mu(x))}{Z^\HYL_N(\beta,\mu,\alpha)}\,\nu_{N,\alpha}^{\ssup{\bc}}(\d x),
\end{aligned}
\end{equation}
with partition function
$$
\begin{aligned}
Z_N^{\HYL}(\beta,\mu,\alpha)=\E_{\nu_{N,\alpha}^{\ssup{\bc}}}\big[\e^{-\abs{\L_N}\beta H^{\HYL}_\mu(x)}\big]={\tt E}^{\ssup{\bc}}\big[\exp\big\{\beta \mu N_{\L}^{\ell}-\frac{\beta a}{2\abs{\L_N}}\big( N_{\L}^{\ell}\big)^2 +\frac{\beta b}{2\abs{\L_N}}\sum_{k=1}^\infty k^2\Ncal_{k}^2\big\}\big].
\end{aligned}
$$

\medskip

In this paper we will focus our attention on the $a>b$ case. For a discussion of $a=b$, see \cite{AD2018b}.

\begin{theorem}[\textbf{Large deviations principle for HYL-models}]\label{Thm:HYL}

For any $ d\in\N,  a>b\geq 0 , \alpha\le 0 $, $ \mu\in\R$,   and $ \bc\in\{\emptyset,\Dir,\per\} $    the following holds. The sequence $ \big(\nu_{N,\mu,\alpha}^{\HYL}\big)_{N\ge 1}   $ satisfies an LDP on $ \ell_1(\R) $ with rate $ \abs{\L_N} $ and rate function 
\begin{equation}
I^{\HYL}_{\mu,\alpha}(x) = \beta H^{\HYL}_{\mu+\alpha,l.s.c.}(x)+ I(x) - \inf_{y\in \ell_1\left(\R\right)}\{\beta H^{\HYL}_{\mu+\alpha,l.s.c.}(y)+I(y)\} ,
\end{equation}
with
\begin{equation}
\begin{aligned}
H^{\HYL}_{\mu,l.s.c.}(x)&=H^\HYL(x)-\frac{1}{2(a-b)}\big(\mu-aD(x)\big)_+^2\\ &= -\frac{b}{2}\sum_{k=1}^\infty k^2 x_k^2 + \begin{cases} -\mu D(x)+\frac{a}{2}D(x)^2 &,D(x)\ge \frac{\mu}{a},\\ - \frac{b}{a-b}\left(-\mu D(x)+\frac{a}{2}D(x)^2\right) -\frac{\mu^2}{2(a-b)} &, D(x)<\frac{\mu}{a}.\end{cases}
\end{aligned}
\end{equation}

\end{theorem}

\begin{remark}
	In the PMF and HYL models we have made a distinction between the chemical potential arising in the reference measure, $\alpha$, and the chemical potential arising in the relevant Hamiltonian, $\mu$. Because the ideal Bose gas partition function diverges for $\alpha>0$, we require $\alpha\leq0$, but no such bound is needed for $\mu$. Note though, that a model with $ \alpha=\alpha_1\le 0 $ and $ \mu=\mu_1\in \R $ coincides with same model but with $ \alpha=0 $ and $ \mu=\mu_1+\alpha_1 $. In many of our proofs we will - without loss of generality - use the $ \alpha=0 $ and $ \mu=\mu_1+\alpha_1 $ arrangement.\hfill $ \diamond $
\end{remark}

Following \cite{BLP} and \cite{Lew86}  one might be interested in large deviation principles for the empirical  particle density
\begin{equation}
\brho_{N}:=\frac{1}{\abs{\L_N}}N_{\L_N}^{\ssup{\ell}}=\frac{1}{\abs{\L_N}}\sum_{k=1}^\infty k\Ncal_k.
\end{equation}
As the particle number $ N_{\L_N}^{\ssup{\ell}} $ is only lower semicontinuous and not upper semicontinuous, a proof via the contraction principle is only feasible if one considers cut-off versions of the empirical particle density $ \brho_N^{\ssup{K}}=\frac{1}{\abs{\L_N}}\sum^K_{k=1}k\Ncal_k $ followed  by analysing the limit $ K\to\infty $ for the corresponding rate functions. We do not follow this approach here and briefly outline a direct approach as follows.

\begin{prop}\label{logmomentdensity}
Let $ \alpha\le 0 $, then for all $ t\in\R $,
\begin{equation}
\begin{aligned}
\L(t)&:=\lim_{N\to\infty} \frac{1}{\abs{\L_N}}\log\E_{\nu_{N,\alpha}^{\ssup{\bc}}}\Big[\ex^{t\sum_{k=1}^\infty k\blambda_N^{\ssup{k}}}\Big]=  \lim_{N\to\infty} \frac{1}{\abs{\L_N}}\log{\tt E}^{\ssup{\bc}}\Big[ \ex^{t\sum_{k=1}^\infty k\Ncal_k}\Big]\\
& = \sum_{k=1}^\infty q_k^{\ssup{\alpha}}\big(\ex^{tk}-1\big)=\begin{cases} +\infty &, \mbox{ if } t>\abs{\alpha}\beta,\\
\in \R & ,\mbox{ if } \alpha\beta+t\le 0.
\end{cases}
\end{aligned}
\end{equation}
\end{prop}
The following large deviation results uses the critical density for the ideal Bose gas, the thermodynamic limit of the pressure and the free energy  defined, respectively,  in \eqref{critical} in Section~\ref{idealgas}, \eqref{p} in Section~\ref{pressure} and Proposition~\ref{IdealFreeEnergy} in Section~\ref{idealgas}. Denote $ Q_{N,\alpha}={\tt Q}^{\ssup{\bc}}\circ \brho_N^{-1} $ the distribution of $ (\brho_N)_{N\ge 1} $ with chemical potential $ \alpha\le 0 $ and define the distribution $ Q_{N,\mu,\alpha}^\PMF $ via its Radon-Nikodym density
\begin{equation}
\frac{\d Q_{N,\mu,\alpha}^\PMF}{\d Q_{N,\alpha}}(x)=\frac{\exp\big(-\beta\big(-\mu x+\frac{a}{2}x^2\big)\big)}{Z_N^\PMF(\beta,\mu,\alpha)}.
\end{equation}

\begin{theorem}\label{THM-densityLDP}
Let $ d\in\N, \beta>0 $, and $ \bc\in\{\emptyset, \per,\Dir\} $.
\begin{enumerate}
\item For any $ \alpha<0 $, the sequence $ \big(Q_{N,\alpha}\big)_{N\ge 1} $  satisfies an LDP on $ \R$ with rate $ \abs{\L_N} $ and rate function
\begin{equation}
J_\alpha(x)=\begin{cases} \beta\big(p(\beta,\alpha)+f(\beta,x)-\alpha x\big) &, \mbox{ if } x\in [0,\varrho_{\rm c}] \;\mbox{ for } d\ge 3 \wedge\;  x\in [0,\infty) \;\mbox{ for } d=1,2,\\
+\infty &, \mbox{ if } x\notin[0,\varrho_{\rm c}].\end{cases}
\end{equation}

\item For any $ \alpha<0 $ and $ \mu\in\R $, the sequence $ \big(Q_{N,\mu,\alpha}^\PMF\big)_{N\ge 1} $ satisfies an LDP on $ \R $ with rate $ \abs{\L_N} $ and rate function
\begin{equation}
J_{\mu,\alpha}^\PMF(x)=\begin{cases} \beta(-\mu x+\frac{a}{2}x^2)+J_\alpha(x) -\Ncal &, \mbox{ if } x\in [0,\varrho_{\rm c}] \;\mbox{ for } d\ge 3 \wedge\; x\in [0,\infty) \;\mbox{ for } d=1,2,\\
+\infty & ,\mbox{ if } x\notin[0,\varrho_{\rm c}],\end{cases}
\end{equation}
where $ \Ncal= \inf_{y\in\R}\{\beta(-\mu y+\frac{a}{2} y^2)+J_\alpha(y)\} $.
\end{enumerate} 
\end{theorem}

\begin{remark}\label{RemdensityLDP}
The results in Theorem~\ref{THM-densityLDP} make the heuristic derivations in \cite{Lew86} rigorous and extend them to all $ \alpha<0 $ and $ \mu\in\R $. 
The free energy of the PMF model is $ f^\PMF(\beta,\varrho)=f(\beta,\varrho)+\frac{a}{2}\varrho^2 $, whereas the pressure
\begin{equation}\label{PMFpressure}
\begin{aligned}
p^\PMF(\beta,\mu,\alpha)&=p(\beta,\alpha)+\frac{1}{\beta}\sup_{x\in\R}\big\{-\beta h(x)-J_\alpha(x)\big\}=\sup_{x\in\R}\big\{(\mu+\alpha)x-\frac{a}{2}x^2-f(\beta,x)\big\}\\
&=\sup_{x\in\R}\big\{ (\mu+\alpha)x-f^\PMF(\beta,x)\big\},
\end{aligned}
\end{equation}
with $ h(x):=-\mu x+\frac{a}{2}x^2 $.
The CMF and the HYL model both require higher level  empirical functionals as the energy cannot be expressed as a functional of the empirical particle density. 
\hfill $ \diamond $
\end{remark}

\section{Variational analysis, pressure representations, and BEC}
Our large deviation analysis in Section~\ref{sec-cycle} is complemented by a complete analysis for the zeroes of the rate functions in Section~\ref{sec-minimiser} followed by pressure representations in Section~\ref{sec-pressure}. In Section~\ref{sec-BEC},  we finally study the onset of Bose-Einstein condensation (BEC) and discuss the relevance of our results.

\subsection{Zeroes of the Rate Functions}\label{sec-minimiser}
We have shown that the ideal Bose gas model, the cycle mean-field (CMF) model, and the particle mean-field (PMF) model and the hard-sphere (HYL) model all satisfy a LDP  for empirical cycle counts with the rate functions $ I_\mu$, $ I^{\CMF}_\mu$, $ I^{\PMF}_{\mu,\alpha}$, and $ I^\HYL_{\mu,\alpha} $ respectively. 

%

We summarise our results on the zeroes in the following statements.

\begin{prop}\label{zeroIdeal}
The rate function for the ideal Bose gas model, $ I_\mu $, has a unique zero $\xi\in\ell_1\left(\R\right)$ given by
\begin{equation}
\xi_k=q^{\ssup{\mu}}_k,\quad k\in\N.
\end{equation}
\end{prop}

\begin{prop}\label{zeroCMF}
	The rate function $ I^{\CMF}_\mu$ has a unique zero at $\xi^\CMF\in\ell_1\left(\R\right)$ given by
	\begin{equation}
	\xi^\CMF_k = \frac{W_0\left(K\right)}{K}q^{\ssup{\mu}}_k, \quad k\in\N,
	\end{equation}
	where $W_0$ is the real branch of the Lambert W function for non-negative arguments, and $K = K\left(a,\beta,\mu\right)$ is a dimensionless quantity defined by
	\begin{equation}
	K := a\beta \bar{q}^{\ssup{\mu}} = \frac{a\beta}{\left(4\pi\beta\right)^{\frac{d}{2}}}g\left(1+\frac{d}{2},-\beta \mu\right).
	\end{equation}
\end{prop}

\medskip

\begin{remark}
Definition and properties of the Lambert function are given in Appendix~\ref{app-Lambert}.
\hfill $ \diamond$
\end{remark}

\medskip

\begin{prop}\label{zeroPMF}
The rate function $ I^\PMF_{\mu,\alpha}$ has a unique zero at $ \xi^\PMF\in\ell_1(\R_+) $ where
\begin{equation*}
\xi^\PMF_k=q_k^{\ssup{0}} \exp\big(\beta k\big(\mu + \alpha-a\delta^*\big)_-\big),\quad k\in\N,
\end{equation*}
and $ \delta^*=\delta^*(\beta,\mu+\alpha,a) $ is given implicitly as the unique solution to the equation
\begin{equation}
\delta^*=\sum_{k=1}^\infty k q_k^{\ssup{\mu}}\exp\big(\beta k\big(\mu + \alpha-a\delta^*\big)_-\big).
\end{equation}

\end{prop}

\begin{remark}
	\label{Ideal-PMF}
	Note that Proposition~\ref{zeroPMF} tells us that the zero of $I^\PMF_{\mu,\alpha}$ is equal to the zero of $I_{\eta}$, where $\eta = \left(\mu + \alpha-a\delta^*\right)_-$.\hfill $ \diamond $
\end{remark}

The zeroes for the HYL model are more complex and involved. We summarise this in the following two propositions.

\begin{prop}\label{zeroHYL} 
The zeroes $ \xi\subset\ell_1(\R_+) $ of rate the function $ I^\HYL_{\mu,\alpha} $ satisfy the following expression,
	\begin{equation*}
	\xi_k = -\frac{1}{b\beta k^2}W_{\chi^*_k}\left(-b\beta k^2 q_k^{\ssup{0}} \exp\left[\beta k \left(\mu +\alpha - a \delta^*\right)
	\begin{Bmatrix}
	1&:a\delta^* \geq \mu + \alpha\\
	-\frac{b}{a-b}&:a\delta^* \leq \mu + \alpha
	\end{Bmatrix}
	\right]\right), \qquad k\in\N,
	\end{equation*}
	where $\left(\delta^*,\chi^*\right) \in \R_+\times \left\{0,-1\right\}^\N$ is a solution to
	\begin{equation}
	\label{eqn:HYLconsistency}
	\delta = g^\chi\left(\delta\right) := -\frac{1}{b\beta} \sum^{\infty}_{j=1} \frac{1}{j}W_{\chi_j}\left(-b\beta j^2 q^{\ssup{0}}_j \exp\left[\beta j \left(\mu + \alpha - a \delta\right)
	\begin{Bmatrix}
	1&:a\delta \geq \mu + \alpha\\
	-\frac{b}{a-b}&:a\delta \leq \mu + \alpha
	\end{Bmatrix}
	\right]\right).
	\end{equation}
\end{prop}

The next proposition shows that there exists a regime for the parameters $ \beta, \alpha,\mu,a,b,d $ such that the rate function has a unique zero.
\begin{prop}
	\label{HYLboringcase}
	There exists $\tilde{\mu} = \tilde{\mu}\left(d,\beta,a,b\right)\in\R$ such that for $\mu + \alpha<\tilde{\mu}$ the rate function $I^\HYL_{\mu,\alpha}$ has a unique zero at $\xi^\HYL\in\ell_1\left(\R_+\right)$ where
	\begin{equation*}
	\xi^\HYL_k=-\frac{1}{b\beta k^2}W_0\left(-b\beta k^2 q_k^\ssup{0}\exp\left[\beta k \left(\mu + \alpha - a \delta^*\right)\right]\right),\qquad k\in\N,
	\end{equation*}
	and $ \delta^*=\delta^*(\beta,\mu+\alpha,a,b) $ is given implicitly as the unique solution to the equation $\delta^*=g^0\left(\delta^*\right)$,
	where $ g^0=g^{\chi\equiv 0} $. See Figure~\ref{fig:gHYL}.
\end{prop}

\begin{figure}
	\centering
	\begin{tikzpicture}[scale = 2]
	\draw[->] (0,0) node[below left]{$0$} -- (0,3) node[left]{$g^0\left(\delta\right)$};
	\draw[->] (0,0) -- (4,0) node[below right]{$\delta$};
	\draw[dashed] (1,0) node[below]{$\frac{\mu+\alpha}{a}$} -- (1,3);
	\draw[thick] (1,2.5) to [out=270,in=175] (4,0.1);
	\draw[thick] (1,2.5) to [out=270,in=35] (0,0.5);
	\fill (1,2.5) circle (0.05);
	\end{tikzpicture}
	\caption{Sketch of $g^\chi\left(\delta\right)$ for $\chi=0$, $d=3,4$, and $\beta\leq \left(\frac{b^2\ex^2}{\left(4\pi\right)^d}\right)^\frac{1}{d-2}$. This shows $\mu+\alpha>0$, but the sketch translates with $\mu$.}\label{fig:gHYL}
\end{figure}
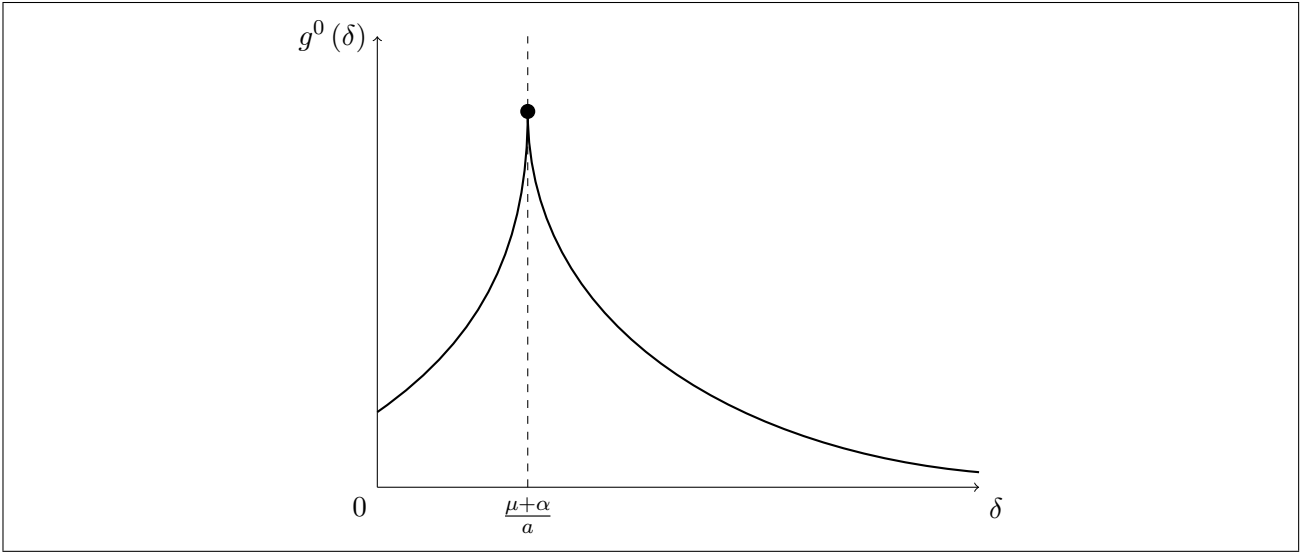

\subsection{Pressure and Bose-Einstein condensation}\label{sec-pressure}
We derive representations for the pressure  and use their formulae to prove the onset of a  phase transition called Bose-Einstein condensation in all our models.
Using our large deviation principles in Section~\ref{sec-cycle} and the zeroes of the rate functions in conjunction with our representation of the partition functions in Proposition~\ref{lem-rewrite} we obtain the thermodynamic limit of the pressure in our various models, defined as 
\begin{equation}\label{pressure}
p^{\ssup{\rm bc}}(\beta,\mu):=\lim_{N\to\infty}\frac{1}{\beta\abs{\L_N}} \log Z_{\L_N}^{\ssup{\rm bc}}(\beta,\mu).
\end{equation} 
We summarise our findings in the following theorem.

\begin{theorem}[\textbf{Pressure representations}]\label{THM-pressure}
Let $ \beta>0 $ and $ {\rm bc}\in\{\emptyset,\per, \Dir\} $, then
\begin{align}
	p(\beta,\mu) &=\lim_{N\to\infty}\frac{1}{\beta\abs{\L_N}}\log \ex^{\abs{\L_N}\overline{q}^{\ssup{\rm bc},\mu}}=    \frac{1}{\beta(4\pi\beta)^{\frac{d}{2}}}\sum_{k=1}^\infty\frac{\ex^{\beta\mu k}}{k^{1+d/2}},  \qquad\text{with } \mu \leq 0,\label{p}\\[1.5ex]
	p^\CMF(\beta,\mu) &=\lim_{N\to\infty}\frac{1}{\beta\abs{\L_N}}\log Z_N^\CMF(\beta,\mu)=\frac{1}{a\beta^2}W_0\left(a\beta\overline{q}^{\ssup{\mu}}\right)\Big(1+\frac{1}{2}W_0\big(a\beta\overline{q}^{\ssup{\mu}}\big)\Big),\nonumber\\[1.5ex]
	&\quad\text{ with }\mu\leq 0, \text{and } a\ge 0,\label{p-cmf}\\[1.5ex]
	p^\PMF(\beta,\mu,\alpha) &=\lim_{N\to\infty}\frac{1}{\beta\abs{\L_N}} \log Z_N^\PMF(\beta,\mu,\alpha) \nonumber \\[1.5ex] 
	&= \begin{cases}
	\frac{a}{2}{\delta^*}^2 + \frac{1}{\beta}\sum^\infty_{k=1}q^\ssup{0}_k\exp\left( \beta k\left(\mu + \alpha - a\delta^*\right)\right)&, \mu + \alpha\leq a\varrho_\mathrm{c},\\[1.5ex]
	\frac{\left(\mu+\alpha\right)^2}{2a}+ \frac{1}{\beta}\sum^\infty_{k=1}q^\ssup{0}_k &, \mu + \alpha\geq a\varrho_\mathrm{c},
	\end{cases}\nonumber\\[1.5ex]
	&\quad\text{ with } \mu\in\R,\alpha\leq 0,\text{and }a\ge 0,\label{p-pmf}\\[1.5ex]
	p^\HYL(\beta,\mu,\alpha)&= \lim_{N\to\infty}\frac{1}{\beta\abs{\L_N}}\log Z_N^\HYL(\beta,\mu,\alpha) = p\left(\beta,0\right)-\inf_{x\in\ell_1(\R)}\left\{I(x) + \beta H^\HYL_{\mu+\alpha,l.s.c.}(x)\right\}, \nonumber\\
	&\quad\text{ with }\mu\in\R,\alpha\leq 0,\text{and } a > b \geq 0.\label{p-hyl}	
\end{align}
\end{theorem}

\begin{proofsect}{Proof of Theorem~\ref{THM-pressure}}
The thermodynamic limit of the average finite-volume pressure exists in all models, see \cite{BR97}. The independence of the thermodynamic limit on the choice of boundary conditions follows with \cite{ACK}, or using \cite{R71,AN73}, see Appendix~\ref{app-A}.
The pressure for the ideal Bose gas in  \eqref{p} follows easily using \eqref{rewrite} and the independence of the boundary conditions. 

For \eqref{p-cmf}, \eqref{p-pmf}, and \eqref{p-hyl} our large deviation principles give us
\begin{align*}
\lim_{N\to\infty}\frac{1}{\left|\Lambda_N\right|}\log Z^\CMF_{N}(\beta,\mu)  &= p\left(\beta,\mu\right) -\inf_{x\in\ell_1(\R)}\left\{I_\mu(x) + \beta H^\CMF(x)\right\},\\
\lim_{N\to\infty}\frac{1}{\left|\Lambda_N\right|}\log Z^\PMF_{N}(\beta,\mu,\alpha) &= p\left(\beta,0\right)-\inf_{x\in\ell_1(\R)}\left\{I(x) + \beta H^\PMF_{\mu+\alpha,l.s.c.}(x)\right\},\\
\lim_{N\to\infty}\frac{1}{\left|\Lambda_N\right|}\log Z^\HYL_{N}(\beta,\mu,\alpha)&= p\left(\beta,0\right)-\inf_{x\in\ell_1(\R)}\left\{I(x) + \beta H^\HYL_{\mu+\alpha,l.s.c.}(x)\right\}.
\end{align*}
Then for \eqref{p-cmf} and \eqref{p-pmf} we substitute in the zeroes found in Proposition~\ref{zeroCMF} and Proposition~\ref{zeroPMF}. The ideal Bose gas pressure $ p(\beta,\mu) $ respectively $ p(\beta,\alpha) $  appears in all pressure formulae due to the term in \eqref{rewrite} stemming from the reference Poisson process.  \qed
\end{proofsect}

In the following subsections we study our different models separately in terms of their thermodynamic behaviour and the onset of Bose-Einstein condensation. Crucial observation is that the partial derivative of the pressure with respect to the chemical potential $ \mu $ at inverse temperature $ \beta $ equals the expected physical particle density at equilibrium in the grand canonical ensemble at inverse temperature $ \beta $ and chemical potential $ \mu $.  In the following, we will distinguish between different regimes depending on whether the expected particle density equals the density of the zeroes of the rate functions or not. In the latter case, the excess density equals the density of the condensate in the BEC state.

\subsubsection{Thermodynamics and BEC of the ideal Bose gas}\label{idealgas}
We collect well-known properties of  the ideal gas pressure  for convenience of the reader and comparison purposes to our other models, for more details see \cite{BLP} and \cite{BCMP05}. 

\begin{prop}\label{P:ideal1}
\begin{enumerate}
\item For $ \beta> 0,\mu>0 $, we define $ p(\beta,\mu)=+\infty $. Then the pressure $ p(\beta,\cdot) $ is a closed convex function on $ \R $.

\item For $\beta>0,\mu<0$, the ideal gas pressure $p(\beta,\mu)$ is smooth with respect to $\mu$. In particular, 
\begin{equation*}
\frac{\d p}{\d \mu} = D\left(q^{\ssup{\mu}}\right).
\end{equation*}
\end{enumerate}
\end{prop}

\smallskip

We denote 
$$ p^{\ssup{\rm bc}}_{\L_N} =\frac{1}{\beta\abs{\L_N}} Z_{\L_N}^{\ssup{\rm bc}}(\beta,\mu) 
$$ the average finite-volume pressure with $ \bc\in\{\emptyset,\per,\Dir\} $.

\begin{prop}\label{densityIdeal}
\begin{enumerate}
\item For $ \beta > 0, \mu< 0 $, and any $ N\in\N $,
\begin{equation}
\frac{1}{\abs{\L_N}}{\tt E}^{\ssup{\rm bc}}\big[N_{\L_N}^{\ssup{\ell}}\big]=\sum_{k\in\N} kq_k^{\ssup{{\rm bc},\mu}}=\frac{\d}{\d\mu} p_{\L_N}^{\ssup{\rm bc}}(\beta,\mu).
\end{equation}
The function $ \mu\mapsto \frac{\d}{\d\mu} p_{\L_N}^{\ssup{\rm bc}}(\beta,\mu) $ is increasing on $ (-\infty,0) $. It follows that we can give $ \frac{1}{\abs{\L_N}}{\tt E}^{\ssup{\rm bc}}[N_{\L_N}^{\ssup{\ell}}]$ any pre-assigned value $ \varrho\in(0,\infty) $ by choosing $ \mu_N(\varrho)\in(-\infty,0) $.

\item In the thermodynamic limit $ N\to\infty $,
\begin{equation}\label{critical}
\varrho_{\rm c}:=\lim_{\mu\uparrow 0}\big(\frac{\d}{\d\mu} p^{\ssup{\rm bc}}_{\L_N}(\beta,\mu)\big)=\begin{cases}  +\infty &, d=1,2,\\   \frac{1}{\left(4\pi\beta\right)^\frac{d}{2}}\zeta\left(\frac{d}{2}\right) &, d\ge 3. \end{cases}
\end{equation}
Let $ \mu_N(\varrho) $ denote the unique root of 
\begin{equation}
\big(\frac{\d}{\d\mu} p^{\ssup{\rm bc}}_{\L_N}(\beta,\mu)\big)=\varrho
\end{equation}
then $ \mu(\varrho)=\lim_{N\to\infty} \mu_N(\varrho) $ exists and is equal to the unique root of
\begin{equation}
\big(\frac{\d}{\d\mu} p^{\ssup{\rm bc}}(\beta,\mu)\big)=\varrho \quad\mbox{ if } \varrho<\varrho_{\rm c},
\end{equation}
and it is equal to zero otherwise.
\end{enumerate}

\end{prop}

\begin{prop}
	\label{IdealFreeEnergy}
	For $\varrho>0$, we define the ideal Bose gas free energy as the Legendre-Fenchel transform of the pressure,
	\begin{align}
		f\left(\beta,\varrho\right):=\sup_{\mu\in\R}\left\{\mu\varrho-p\left(\beta,\mu\right)\right\}
		= \begin{cases}
		\frac{-1}{\beta\left(4\pi\beta\right)^{\frac{d}{2}}}g(1+\frac{d}{2},-\beta\alpha) + \varrho\alpha &,  \varrho\leq\varrho_{\rm c}\; ,\\
		\frac{-1}{\beta\left(4\pi\beta\right)^\frac{d}{2}}\zeta\left(1+\frac{d}{2}\right) &,\varrho\geq\varrho_{\rm c}\;, 
		\end{cases}
	\end{align}
	where $\alpha\le 0 $ is a solution to
	\begin{equation*}
	\frac{1}{\left(4\pi\beta\right)^\frac{d}{2}}g\big(\frac{d}{2},-\beta\alpha\big) = \varrho,
	\end{equation*}
	which exists and is unique for $\varrho\leq\varrho_{\rm c}$.
\end{prop}

It is easy to see that $ \varrho\mapsto f(\beta,\varrho) $ is a decreasing convex function; it is given by
$$
f(\beta,\varrho)=\mu(\varrho)\varrho-p(\beta,\mu(\varrho)) \quad\mbox{ for } \varrho<\varrho_{\rm c}.
$$
The linear segment in the graph of $ f$ where $ f $ is constant and equal to $ -p(\beta,0) $ for $ \varrho\ge \varrho_{\rm c} $, signals a first-order phase-transition at $ \mu=0 $. This phase-transition is called Bose-Einstein condensation (BEC).

In \cite{Lew86}, Lewis suggests an order parameter for the Bose-Einstein condensation phase transition. Having in mind the density of particles with zero single particle energy in the thermodynamic limit, Lewis first takes the finite volume expected density of particles with energy below some cut-off. He then takes the thermodynamic limit before taking the cut-off to zero. In contrast to Lewis' model, we do not keep track of the particles' energy. Instead, we partition our gas by loop type, and expect the condensate to occupy loops of diverging length. Therefore we want to evaluate the `condensate density' given by
	\begin{equation*}
	\Delta\left(\beta,\mu\right) := \lim_{K\to\infty}\lim_{N\to\infty} \E_{\nu_{N,\mu}}\left[D-D_K\right],\qquad D_K\left(x\right):=\sum^K_{j=1}jx_j.
	\end{equation*}
	
	\begin{theorem}\label{thm:IDEALcondensate}
		For $\beta>0$, $\mu<0$, we have $\Delta\left(\beta,\mu\right) = 0$.
		
		For $\beta>0$, $\mu=0$,
		\begin{equation*}
		\Delta\left(\beta,0\right) = \begin{cases}
		+\infty &, d=1,2, \text{ or } d\geq3, \bc = {\rm per}\\
		0 & ,d\geq 3, \bc\in\left\{\emptyset,{\rm Dir}\right\}.
		\end{cases}
		\end{equation*}
	\end{theorem}

\noindent \textbf{Conclusion.}     The ideal gas shows some critical behaviour, namely, the density and chemical potential relation breaks down in dimensions $ d\ge 3 $. This can be seen as a signal of a phase transition manifesting itself in the in-equivalence of the canonical and grand canonical ensemble.  The excess density above the critical value of the density is identified as the BEC condensate density. The free energy is constant for all densities above the critical one, showing that the condensate density does not contribute to the free energy. In \cite{BCMP05}, using the energy (Fourier) representation, it is shown that the excess density (for dimensions $ d\ge 3 $) equals the expected density of particles in the zero-energy mode. Mathematically, the critical behaviour is seen in the thermodynamic limit of distribution function of the empirical particle density (called Kac distribution), the limit of the Kac distribution  exists only for densities $ \varrho<\varrho_{\rm c} $ and is the degenerate distribution 
$$
\mathbb{K}(x)=\begin{cases} 0 &, x <\varrho, \\ 1  & , x\ge \varrho.\end{cases}
$$
This critical behaviour is also seen in the rate function for the density large deviation principle in Theorem~\ref{THM-densityLDP}. 	We believe that the unexpected results in Theorem~\ref{thm:IDEALcondensate} are due to this degeneracy.

\subsubsection{Thermodynamics of the CMF model}\label{sec-thermCMF}
We summarise our findings for the CMF model below.

\medskip

\begin{prop}\label{P:densityCMF}
	\begin{enumerate}
		\item For $ \beta> 0 $, we define $ p^\CMF(\beta,\mu)=+\infty $ for $ \mu > 0 $. Then $ p^\CMF(\beta,\cdot) $ is a closed convex function on $ \R $.
		
		\item For $\beta>0$, $\mu<0$, the CMF gas pressure $p^\CMF(\beta,\mu)$ is smooth with respect to $\mu$. In particular,
			\begin{equation*}
			\frac{\d p^\CMF}{\d \mu} = D\left(\xi^\CMF\right) = \frac{W_0\left(a\beta\bar{q}^\ssup{\mu}\right)}{a\beta\bar{q}^\ssup{\mu}}D\left(q^\ssup{\mu}\right).
			\end{equation*}

\item In the thermodynamic limit $ N\to\infty $,
\begin{equation}\label{criticalCMF}
\varrho^{\CMF}_{\rm c}:=\lim_{\mu\uparrow 0}\big(\frac{\d}{\d\mu}p^{\CMF}_{\L_N}(\beta,\mu)\big)=\begin{cases} +\infty &, d=1,2,\\[1.5ex] \frac{W_0(a\beta\overline{q}^{\ssup{0}})}{a\beta\overline{q}^{\ssup{0}}(4\pi\beta)^{\frac{d}{2}}}\zeta\big(\frac{d}{2}\big) &, d\ge 3.
\end{cases}
\end{equation}
	\end{enumerate}
\end{prop}

\medskip

\begin{prop}\label{P:LFT-CMF}
	For $\varrho>0$, the CMF Bose gas free energy is defined as the Legendre-Fenchel transform of the pressure, 
	\begin{equation}
	f^\CMF\left(\beta,\varrho\right) := \sup_{\alpha\in\R}\left\{\alpha\varrho - p^\CMF\left(\beta,\alpha\right)\right\}
	= \begin{cases}
	\varrho\mu - p^\CMF\left(\beta,\mu\right) &, \varrho\leq \varrho^\CMF_{\rm c},\\
	-p^\CMF\left(\beta,0\right) &, \varrho\geq \varrho^\CMF_{\rm c},
	\end{cases}
	\end{equation}
	where $\alpha=\mu$ is a solution to
	\begin{equation*}
	\frac{1}{a\beta\bar{q}^{\ssup{\alpha}}}W_0\left(a\beta \bar{q}^{\ssup{\alpha}}\right)\sum^\infty_{k=1}kq_k^{\ssup{\alpha}} = \varrho,
	\end{equation*}
	which exists and is unique for $\varrho\leq \varrho^\CMF_{\rm c}:=\frac{W_0(K)}{K}\varrho_{\rm c}$.
\end{prop}

\bigskip

	\begin{theorem}\label{thm:CMFcondensate}
	For all $\beta>0$, $\mu<0$,
	\begin{equation*}
	\Delta^\CMF\left(\beta,\mu\right) =0.
	\end{equation*}
\end{theorem}

\bigskip

\noindent \textbf{Conclusion.}   The CMF model shows the same BEC phase transition as the ideal Bose gas as the free energy is constant in the density beyond its specific  critical density. The critical density of the CMF model is different from the ideal Bose gas one. Using properties of the Lambert function, see Appendix~\ref{app-Lambert}, we know that 
$$
\begin{aligned}
\lim_{c\downarrow 0}\frac{W_0(cx)}{cx}&=1,\\
\lim_{c\to\infty}\frac{W_0(cx)}{cx}&=0.
\end{aligned}
$$
Hence, as the coupling parameter $a \to 0 $ vanishes, we obtain the critical ideal Bose gas density, and as $ a\to\infty $ the critical density decreases indicating BEC for much lower particle densities. When the coupling parameter $a$ increases  the number of finite cycles is suppressed in the probability measure,  and therefore the system undergoes a transition to a regime where the particle density is realised in so-called infinite cycles. The CMF model has not been studied in the literature so far, it shows similar behaviour as the ideal Bose gas because the Hamiltonian adds only weight on large numbers of cycles present. The following models involving the physical particle density have a more complex phase transition behaviour. As in the ideal Bose gas, the condensate density in Theorem~\ref{thm:CMFcondensate} can be computed only in the sub-critical regime. This is again, due to the degenerate behaviour of the distribution.

\subsubsection{Thermodynamics of the PMF Bose Gas}\label{sec-thermPMF}
We collect our findings for the PMF model. Our results are using rigorous large deviation analysis and all possible ranges of the chemical potential. We obtain all results in \cite{BCMP05} with a completely different method for all values of the chemical potential, in addition, we compute the condensate density in Theorem~\ref{thm:PMFcondensate} below.  The density square term in the Hamiltonian stabilises the distribution such that the limiting distribution is no longer degenerate, allowing us the compute the partial derivatives for all values of the chemical potential. We identify regimes where the expected particle density equals the density of the rate function zero or not.

\medskip

\begin{prop}\label{P:densPMF}
	For $\beta>0$, the pressure $p^{\PMF}\left(\beta,\cdot\right)\in C^{1}\left(\R\right)$ and is convex. In particular,
	\begin{align*}
	\mu +\alpha < a\varrho_{\mathrm{c}} &\implies \frac{\d p^\PMF}{\d \mu} = D\left(\xi^\PMF\right) > \frac{\mu+\alpha}{a}\\
	\mu +\alpha= a\varrho_{\mathrm{c}} &\implies \frac{\d p^\PMF}{\d \mu} = D\left(\xi^\PMF\right) = \frac{\mu+\alpha}{a}\\
	\mu +\alpha > a\varrho_{\mathrm{c}} &\implies \frac{\d p^\PMF}{\d \mu} = \frac{\mu+\alpha}{a} > D\left(\xi^\PMF\right).
	\end{align*}
\end{prop}

\medskip

\begin{prop}\label{P:fPMF}
	Let $ \alpha\le 0 $. For $\varrho>0$, the PMF Bose gas free energy is defined as the Legendre-Fenchel transform of the pressure,
	\begin{equation}
	f^\PMF\left(\beta,\varrho\right):=\sup_{\mu\in\R}\left\{(\mu+\alpha)\varrho - p^\PMF\left(\beta,\mu,\alpha\right)\right\}= f\left(\beta,\varrho\right)+ \frac{a}{2}\varrho^2.
	\end{equation}
\end{prop}

\bigskip

\begin{theorem}\label{thm:PMFcondensate}
	For all $\beta>0$, $\mu\in\R$ and $ \alpha\leq 0 $,
	\begin{equation*}
	\begin{aligned}
		\Delta^\PMF\left(\beta,\mu,\alpha\right) &= \Big(\frac{\d p^\PMF}{\d \mu}\left(\beta,\mu,\alpha\right) - \varrho_\mathrm{c}\left(\beta\right)\Big)_+ = \Big(\frac{\mu+\alpha}{a} - \varrho_\mathrm{c}\left(\beta\right)\Big)_+ \\ & = \Big(\frac{\partial}{\partial \mu}\left(H^\PMF - H^\PMF_{\mu+\alpha,l.s.c.}\Big)\right)\left(\xi^\PMF\right),
		\end{aligned}
	\end{equation*}
	where $\xi^\PMF$ is the unique minimiser (zero) of the rate function $I^\PMF_{\mu,\alpha}$.
\end{theorem}

\bigskip

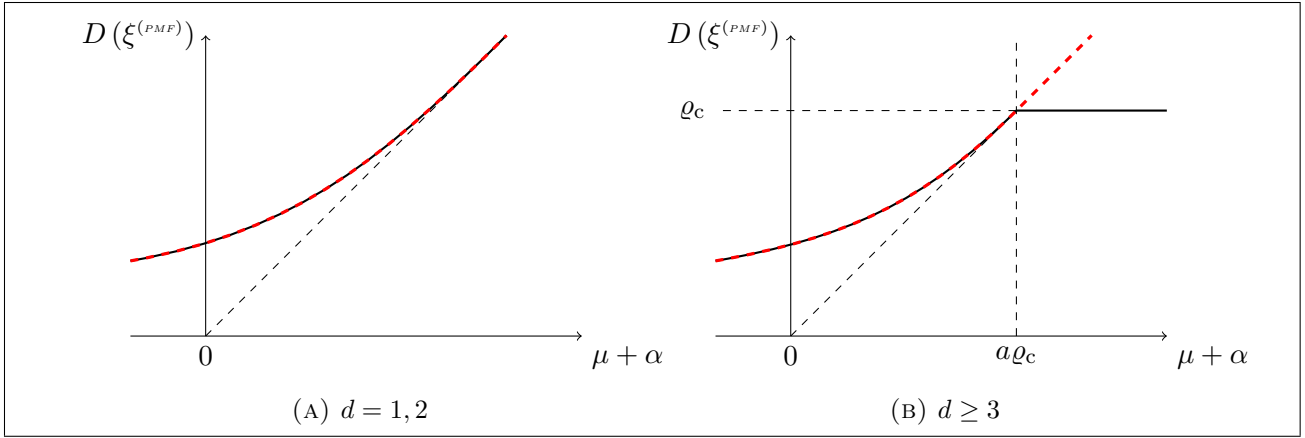
\begin{figure}
	\centering
	\begin{subfigure}[b]{0.45\textwidth}
		\begin{tikzpicture}[scale = 1]
		\draw[->] (-1,0) -- (5,0) node[below right]{$\mu+\alpha$};
		\draw[->] (0,0) node[below]{$0$} -- (0,4) node[left]{$D\left(\xi^\PMF\right)$};
		\draw[dashed] (0,0) -- (4,4);
		\draw[thick] (-1,1) to [out=10,in=225] (4,4);
		\draw[very thick, red, dashed] (-1,1) to [out=10,in=225] (4,4);
		\end{tikzpicture}
		\caption{$d=1,2$}
		\label{fig:PMFdensity12}
	\end{subfigure}
	\begin{subfigure}[b]{0.45\textwidth}
		\begin{tikzpicture}[scale = 1]
		\draw[->] (-1,0) -- (5,0) node[below right]{$\mu+\alpha$};
		\draw[->] (0,0) node[below]{$0$} -- (0,4) node[left]{$D\left(\xi^\PMF\right)$};
		\draw[dashed] (3,0) node[below]{$a\varrho_\mathrm{c}$}--(3,4);
		\draw[dashed] (0,0) -- (3,3);
		\draw[dashed] (3,3) -- (-1,3) node[left]{$\varrho_\mathrm{c}$};
		\draw[thick] (-1,1) to [out=10,in=225] (3,3) -- (5,3);
		\draw[very thick, red, dashed] (-1,1) to [out=10,in=225] (3,3) -- (4,4);
		\end{tikzpicture}
		\caption{$d\geq3$}
		\label{fig:PMFdensity>=3}
	\end{subfigure}
	\caption{Total particle density of the zero of $I^\PMF_{\mu,\alpha}$. The limiting expected particle density (including the condensate) only differs for $\mu+\alpha>a\varrho_{\mathrm{c}}$, where it follows the dashed plot.}
	\label{fig:PMFdensity}
\end{figure}

\bigskip

\noindent \textbf{Conclusion.}  The BEC phase transition is established in various equivalent ways, in Theorem~\ref{thm:PMFcondensate} it is shown that the excess particle density is carried by so-called loops of unbounded length. Alternatively, Proposition~\ref{P:densPMF} and Proposition~\ref{P:fPMF} establish the phase transition via the change of the pressure density relation. The advantage of our LDP approach is that the rate function has unique zero and not an approximating sequence of minimiser. This is due to the fact that we are using the lower semicontinuous regularisation of the energy proving the large deviation principle. A close inspection of Figure~\ref{fig:PMFdensity}
reveals this. For $ d\ge 3 $, we know that $ a\varrho_{\rm c} <\infty $, and thus the density of the zero  of the rate function is constant for all $ \mu+\alpha \ge a\varrho_{\rm c} $. In this region, the total particle density is the dashed line intersecting the point $ (a\varrho_{\rm c},\varrho_{\rm c}) $. The so-called condensate density is then $ (\frac{\mu+\alpha}{a}-\varrho_{\rm c})_+ =\Delta^\PMF(\beta,\mu,\alpha) $.

\subsubsection{Thermodynamics of the HYL model}\label{sec-thermHYL}

We collect our findings for the HYL model. Note that our results hold for all parameter $ a>b $ wheres the ones in \cite{BLP} apply only to $ a=2b$, and we can dispense a technical assumption necessary in \cite{BLP}.
We first identify  the sub-critical regime where the expected physical particle density equals the density of the possible zeroes of the HYL rate function.

\begin{prop}\label{P:HYLsub}
	There exists $\tilde{\mu} = \tilde{\mu}\left(d,\beta,a,b\right)\in\R$ such that for $\mu+\alpha<\tilde{\mu}$, $p^\HYL\left(\beta,\mu,\alpha\right)$ is smooth and convex in $\mu$. In particular,
	\begin{align*}
		\frac{\d p^\HYL}{\d \mu} = D\left(\xi^\HYL\right)
	\end{align*}
	for this range of $\mu$. $\xi^\HYL$ is given in Proposition~\ref{HYLboringcase}.
\end{prop}

\bigskip

The following proposition shows that, for dimension $ d=3, 4$ and large $ \beta $ depending on the value for the counter energy term with pre-factor $ b$, the pressure is no longer smooth and thus the density pressure relation is void, signalling some critical behaviour. The analysis is complex, and in future work \cite{AD2018b} we hope to identify this regime more explicitly.

\bigskip

\begin{prop}\label{HYLpressure}
	For $d=3,4$, $a>b>0$ and $\beta \geq \beta^* = \left(\frac{b^2\ex^2}{\left(4\pi\right)^d}\right)^\frac{1}{d-2}$, the pressure $p^\HYL\left(\beta,\cdot,\alpha\right)\notin C^1\left(\R\right)$.
\end{prop}

\bigskip

	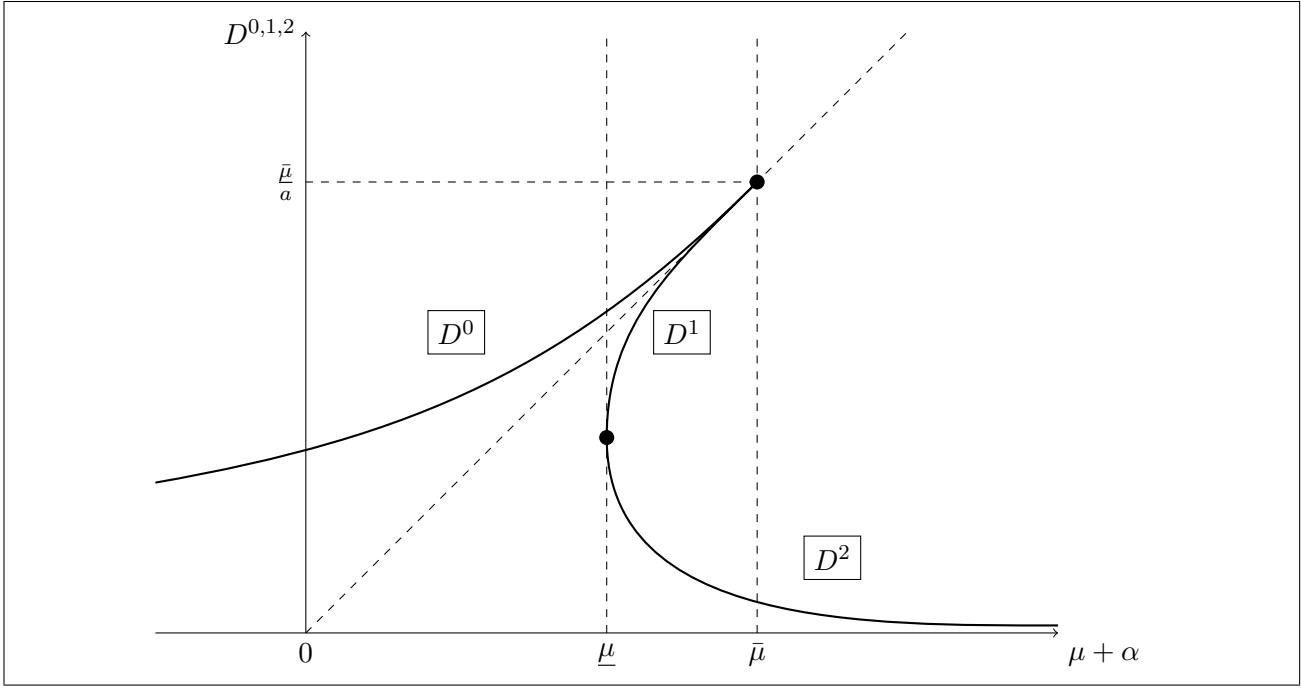
\begin{figure}
		\centering
		\begin{tikzpicture}[scale = 2]
		\draw[->] (-1,0) -- (5,0) node[below right]{$\mu+\alpha$};
		\draw[->] (0,0) node[below]{$0$} -- (0,4) node[left]{$D^{0,1,2}$};
		\draw[dashed] (3,0) node[below]{$\bar{\mu}$}--(3,4);
		\draw[dashed] (2,0) node[below]{$\underline{\mu}$}--(2,4);
		\draw[dashed] (0,0) -- (4,4);
		\draw[thick] (-1,1) to [out=10,in=225] (3,3) to [out=225,in=90] (2,1.3) to [out=270,in=180] (5,0.05);
		\fill (2,1.3) circle (0.05);
		\fill (3,3) circle (0.05);
		\draw[dashed] (3,3) -- (0,3) node[left]{$\frac{\bar{\mu}}{a}$};
		\node[draw] at (1,2) {$D^0$};
		\node[draw] at (2.5,2) {$D^1$};
		\node[draw] at (3.5,0.5) {$D^2$};
		\end{tikzpicture}
		\caption{Sketch of the total particle density of the three $\chi=0$ solutions for $d=3,4$, $\beta\geq \beta^*$.}
		\label{fig:HYL3sols}
	\end{figure}

	\bigskip
	
	The following result shows that the condensate density has a limit for certain regimes of thermodynamic parameter $ \beta $ and $ \mu $ and energy parameter $ a $ and $ b$. Explicit expressions for the minimiser and condensate density in the critical regime are not available. We think this can be achieved with a combined density and cycle count large deviation principle in \cite{AD2018b}.
	
	\medskip

	\begin{theorem}\label{THM:HYL-Cond}
	For $\beta>0$, $\mu\in\R$, $\alpha\leq0$, where the derivative is defined,
	\begin{equation}\label{HYLcondenstate1}
	\begin{aligned}
	\Delta^\HYL\left(\beta,\mu,\alpha\right)& := \lim_{K\to\infty}\lim_{N\to\infty} \E_{\nu_{N,\mu,\alpha}^\HYL}\left[D-D_K\right] \\ &= -\lim_{K\to\infty}\left.\frac{\d}{\d s}\left(\inf_{\ell_1}\left\{\frac{1}{\beta}I^\HYL_{\mu+s,\alpha} + s D_K\right\} - p^\HYL\left(\beta,\mu+s,\alpha\right)\right)\right|_{s=0}.
	\end{aligned}
	\end{equation}
	In particular, if the infimum is achieved by $\xi\left(s\right)\in C^1\left(\left(-\epsilon,\epsilon\right):\ell_1\right)$ for some $\epsilon>0$, then
	\begin{equation}\label{HYLcondenstate2}
	\Delta^\HYL\left(\beta,\mu,\alpha\right) = \frac{a}{a-b}\left(\frac{\mu+\alpha}{a} - D\left(\xi\left(0\right)\right)\right)_+ = \left(\frac{\partial}{\partial \mu}\left(H^\HYL_{\mu+\alpha} - H^\HYL_{\mu+\alpha,l.s.c.}\right)\right)\left(\xi\left(0\right)\right).
	\end{equation}
\end{theorem}

\bigskip

\noindent \textbf{Conclusion.}  The BEC phase transition for the HYL model  is established as follows. Proposition~\ref{P:HYLsub} establishes a subcritical regime showing that for all $ \mu+\alpha\le \tilde{\mu} $ the pressure is smooth and its derivative gives the particle density with no condensation.    
In Proposition~\ref{HYLpressure} we identify a regime for the inverse temperature where the pressure-density relation is broken. Depending on the density $ D(\xi(0)) $, for large enough $ \mu $ the particle density in loops of unbounded length is not vanishing. In Figure~\ref{fig:HYL3sols} we can identify the regime $ \mu+\alpha\ge \bar{\mu} $ when the density of the zero of the rate function is decreasing with $ \mu $ such that the excess density is carried by loops of unbounded length. If we choose $ a=2b $ in \eqref{HYLcondenstate2}, we can recover the results in \cite{Lew86} and \cite{BLP}. 
It shows in fact, that for increasing values of the coupling parameter $ a $, the condensate density decreases. On the other hand, if the parameter for the counter energy term, $ b $, is approaching $ a $, the condensate density increases. This is due to the fact that with large counter terms the system distributes the physical particles in as few as possible different cycles lengths. To accommodate the particle density, the only way is to put them in infinitely long cycles.  Our analysis actually shows that the BEC phase transition for HYL is more complex and requires further detailed study, see \cite{AD2018b}.

\subsection{Relevance and discussion}\label{sec-BEC}

\noindent One of the most prominent open problems in mathematical physics is the understanding of  {\em Bose-Einstein condensation (BEC)}, a phase transition in an interacting many-particle system when a  sufficiently low temperature is reached. That is, a macroscopic part of the system condenses to a state which is highly correlated and coherent.  The first experimental realisation of BEC was only in 1995, and it has been awarded with a Nobel prize. In spite of an enormous research activity, this phase transition has withstood a mathematical proof yet. Only partial successes have been achieved, like the description of the free energy of the ideal, i.e., non-interacting, system (already contained  in Bose's and Einstein's seminal paper in 1925) or the analysis of mean-field models (e.g. \cite{Toth90, DMP05,AD06}) or  the analysis of dilute systems at vanishing temperature \cite{LSSY05} or the proof of BEC in lattice systems with half-filling \cite{LSSY05}. However, the original problem for non-vanishing expected particle density and temperature is still waiting for a promising attack.  The main purpose of the present paper is to provide representation formulae for the thermodynamic limit of the pressure in the so-called grand canonical ensemble and to show onset of BEC using these representations in various interaction models.  We study first the well-known  ideal Bose gas and its novel cousin  the CMF model. The CMF model shows onset of BEC as the ideal Bose gas only for different critical values. We then analyse the well-known PMF model but we allow for strictly positive chemical potential. To prove the large deviation principle for the PMF model it is necessary to consider the lower semicontinuous regularisation of the mean-field energy term. This regularisation has not been done in the literature yet, see \cite{BCMP05}, and it allows to find a minimum for the pressure representation for any set of parameters (inverse temperature, chemical potential and coupling constant), and thus establishes a rigorous and streamlined approach improving all previous studies of particle mean-field models which have to consider sequences of approximating minimisers. 
Our interaction potentials take also negative values for the HYL model. Our analysis for the HYL model shows that the BEC is more complex with the negative counter term and that it depends on the interaction parameter $ a $ and $ b $ as well as the thermodynamic parameter.

The mathematical description of bosons is in terms of the symmetrised trace of the negative exponential of the corresponding Hamiltonian times the inverse temperature. The symmetrisation creates long range correlations of the interacting particles making the analysis an extremely challenging endeavour. The Feynman-Kac formula gives, in a natural way, a representation in terms of an expansion with respect to the cycles of random paths. It is conjectured by Feynman \cite{F53} that BEC is signalled by the decisive appearance of  a macroscopic amount of \lq infinite\rq\ cycles, i.e. cycles whose lengths diverge with the number of particles. This phenomenon is also signalled by a loss of probability mass in the distribution of the \lq finite\rq\ cycles. See  \cite{S02}  for proofs of this coincidence in the ideal Bose gas and some mean-field models. A different line of research is studying the effect of the symmetrisation in random permutation and random partition models, see \cite{Ver96}, \cite{BCMP05}, \cite{AD06,AK08,A08}, or in spatial random permutation models going back to \cite{F91}.

In the present paper, we prove large deviation principles for all our models, and as such generalise recent work in \cite{ACK} where upper and lower large deviation bounds do not match.  In current work \cite{AD2018b}, we analyse the details of the HYL model and its BEC transition in more depth extending \cite{BLP} in a significant way. In future work, our aim is to establish level-3 large deviation principles for Lennard-Jones type potentials in the grand canonical ensemble for pressure representation formulae.

The methods used in the present paper are mainly probabilistic and are extensions and adaptations of recent work \cite{ACK}. Our starting point is the well-known Feynman-Kac formula, which translates the partition function in terms of an expectation over a large symmetrised system of interacting Brownian bridge paths. In a second step, which is also well-known, we reduce the combinatorial complexity by concatenating the bridges, using the symmetrisation. The novelty of the approach in \cite{ACK} is a reformulation of this system in terms of an expectation with respect to a {\em marked Poisson point process}, which serves as a reference process. This is a Poisson process in the space $\R^d$ to whose particles we attach cycles called marks, starting and ending at that particle. The symmetrisation is reflected by an {\it a priori} distribution of cycle lengths. 

Approaches to Bose gases using point processes have occasionally been used in the past (see \cite{F91} and the references therein) and also in \cite{R09}, but systems with interactions have not yet been considered using this technique, to the best of our knowledge. 

The greatest advantage of this approach is that it is amenable to a large deviations analysis. The central objects here are  the {\it empirical cycle count} and the {\it empirical particle density} of the marked point process. For some class of interacting systems, this direction of research was explored in \cite{GZ93,G94}. In the present paper, we apply these ideas to the more difficult case of the interacting Bose gas.

\section{Proof of the Large deviations principles}\label{sec-LDPproofs}
This section contains the proofs for all large deviation principles. In Section~\ref{sec-Ideal} we establish the  LDP for the ideal Bose gas model. We use this large deviation principle   as a stepping stone towards arriving at LDPs for our interaction models. Section~\ref{sec-CMF}  gives the proof for the cycle mean-field (CMF) LDP, Section~\ref{sec-PMF} for particle mean-field (PMF) LDP, and Section~\ref{sec-HYL} the hard-core (HYL) LDP.

\subsection{Proof of Proposition~\ref{THM-Ideal} - Ideal Bose Gas LDP}\label{sec-Ideal}

\begin{remark}
	The condition that $\mu\leq 0$ arises from the $\overline{q}^{\ssup{\mu}}$ term. Our reference marked Poisson point process is superposition of independent marked Poisson point processes on $ \R^d\times \Ccal_k $ with intensity measure given in \eqref{nudef}. The superposition is itself a marked Poisson point process if and only if $ \overline{q}^{\ssup{\mu}} <\infty $.  Clearly, $\overline{q}^{\ssup{\mu}}$ is finite if and only if $\mu\leq0$. \hfill $ \diamond $
\end{remark}

Our derivation of this LDP will be based on applying Baldi's Theorem. We recall that theorem for the convenience of the reader in the following lemma.
\begin{lemma}[\textbf{Baldi's Theorem}]\label{Baldi}
	Suppose $(\nu_N)_{N\in\N} $ is an exponentially tight sequence of measures on $\ell_1(\R)$. Let $\L\colon\ell_\infty(\R)\to [0,\infty] $ be the limiting cumulant generating function, and suppose that it exists and is finite for $t\in\ell_\infty(\R)$. If $\L$ is G{\^a}teaux differentiable, and lower semicontinuous on $\ell_\infty(\R)$, then $(\nu_N)_{N\in\N} $ satisfies an LDP with rate function
	\begin{equation}
		\L^*(x) = \sup_{t\in\ell_\infty(\R)}\big\{\langle t,x\rangle - \L(t)\big\},\qquad x\in\ell_1(\R).
	\end{equation}
\end{lemma}

We shall now set about establishing that the hypotheses of Baldi's Theorem are satisfied. We adapt a beautiful proof in a recent study of Bosonic loop measures on graphs given in \cite{Dan15}.

\begin{lemma}\label{exptight}
For every  $ \mu\le 0 $, 	$\big(\nu_{N,\mu}\big)_{N\in\N} $ is an exponentially tight sequence of measures.
\end{lemma}

\begin{proofsect}{Proof}
	Suppose there exists an $x=x\left(\alpha\right)\in\ell_1\left(\R\right)$ such that for all $k\geq1$,
	\begin{equation*}
		\limsup_{N\rightarrow\infty}\frac{1}{\left|\Lambda_N\right|}\log\nu_{N,\mu}\left(\blambda^{\ssup{k}}_{N}\geq x_k\right) <-2^{-k}\alpha,
	\end{equation*}
	where $\blambda_N=\big(\blambda^{\ssup{k}}_{N}\big)_{k\in\N} $ is an $\ell_1\left(\R\right)$-valued random variable with law $\nu_{N,\mu}$. Also, define the set
	\begin{equation*}
		K = \left\{y\in\ell_1\left(\R\right)\colon \left|y_k\right| \leq \left|x_k\right| \forall k \geq 1\right\},\qquad x\in\ell_1(\R).
	\end{equation*}
To show compactness of $K$ it is first  easy to see that $K$ is bounded and closed. Boundedness follows from $ \norm{y}_{\ell_1(\R)}\le \norm{x}_{\ell_1(\R)} $ for all $ y\in K $. Suppose that $K$ is not closed, that is, there exists a sequence $ y^{\ssup{n}}\in K $ with limit $ y^{\ssup{n}}\to y\notin K $ as $ n\to\infty $. Suppose that $ \abs{y_k}>\abs{x_k} $ for $ y\notin K $ and some $k\in\N $. Choose $ \eps=\frac{1}{2}(\abs{y_k}-\abs{x_k}) $, then, for $ n $ sufficiently large, 
$$
\abs{y_k^{\ssup{n}}-y_k}\le \sum_{j\in\N}\abs{y^{\ssup{n}}_j-y_j}< \frac{1}{2}(\abs{y_k}-\abs{x_k}),
$$
which implies that 
$$
\abs{y^{\ssup{n}}_k}>\abs{y_k}-\eps=\frac{1}{2}(\abs{y_k}+\abs{x_k})>\abs{x_k},
$$ contradicting $ y^{\ssup{n}}\in K $. Hence, $ K$ is closed. It remains to show that $K$ is totally bounded.  From that, we shall find a finite cover of $\eps$-open balls for $K$. Pick $ \eps> 0 $, and choose $ N\in\N$ such that $ \sum_{k>N} \abs{x_k}<\eps/2 $, and define the so-called cut-off sequences $ \widetilde{K}=\{y\in K\colon y_k=0, k> N\} $. Clearly, $ \widetilde{K} $ is isomorphic to the totally bounded set 
$$
[-\abs{x_1},\abs{x_1}]\times\cdots[-\abs{x_N},\abs{x_N}] \subset \R^N,
$$
and thus it is itself totally bounded. There exist $ w^{\ssup{1}},\ldots,w^{\ssup{M}}\in \widetilde{K} $ such that 
$$ 
\widetilde{K}\subset \bigcup_{i=1}^M B(w^{\ssup{i}},\frac{\eps}{2}).
$$
For any $y\in $ denote $ \widetilde{y}\in\widetilde{K} $ the sequences which agrees with $y$ on the first $ N$ terms, and choose $ w^{\ssup{i}} $ such that $ \widetilde{y}\in B(w^{\ssup{i}},\frac{\eps}{2}) $. Then,
$$
\norm{y -w^{\ssup{i}}}_{\ell_1(\R)}=\sum_{k=1}^N\abs{\widetilde{y}_k-w^{\ssup{i}}_k}+\sum_{k>N}\abs{y_k}<\frac{\eps}{2}+\frac{\eps}{2}.
$$
Thus, $ K\subset\bigcup_{i=1}^M B(w^{\ssup{i}},\frac{\eps}{2}) $, and we conclude with the compactness of $K$.

	Now, since the $\blambda^{\ssup{k}}_N$ are independent, we have
	\begin{equation*}
		\limsup_{N\rightarrow\infty}\frac{1}{\left|\Lambda_N\right|}\log\nu_{N,\mu}\left(K^{\rm c}\right) = \limsup_{N\rightarrow\infty}\frac{1}{\left|\Lambda_N\right|}\sum_{k\in\N}\log\nu_{N,\mu}\left(\blambda^{\ssup{k}}_N>x_k\right)<-\alpha,
	\end{equation*}
	and conclude with the statement in the lemma.
	All that remains now is to find such a sequence $x$. We consider each $x_k$ in turn. For all constants $c\geq0$, and $\tau>0$, we have the Chernoff bound
	\begin{align*}
		\nu_{N,\mu}\left(\blambda^{\ssup{k}}_N> c\right) &= \nu_{N,\mu}\left(\ex^{\frac{\tau}{\abs{\L_N}} \Ncal_k} > \ex^{\tau c}\right)\\
		&\leq \ex^{-\tau c}{\tt E}^{\ssup{k}}\left[\ex^{\frac{\tau}{\abs{\L_N}} \Ncal_k}\right]\\
		&= \ex^{-\tau c}\exp\left(\left|\Lambda_N\right|q^{\ssup{\mu}}_k\left(\ex^{\frac{\tau}{\left|\Lambda_N\right|}}-1\right)\right).
	\end{align*}
	The inequality is an application of Markov's inequality, and the expectation is nothing other than the moment generating function of $\Ncal_k$ and $ {\tt E}^{\ssup{k}} $ is the expectation with respect to the Poisson point process on $ \R^d $ with intensity $ q^{\ssup{\mu}}_k $. Since $\Ncal_k$ is a Poisson random variable with mean $\abs{\Lambda_N}q^{\ssup{\mu}}_k$, this can be calculated.
	
	Differentiating this bound with respect to $\tau$ gives us that the minimum occurs at $\tau^* = \abs{\Lambda_N}\log\frac{c}{q^{\ssup{\mu}}_k}$. If $c>0$, then $\tau^*>0$ for sufficiently large $N$. This means that we can optimise this form of bound as
	\begin{equation*}
		\nu_{N,\mu}(\blambda^{\ssup{k}}_N>c) \leq \Big(\frac{c}{q^{\ssup{\mu}}_k}\Big)^{-\left|\Lambda_N\right|c}\exp\left(\left|\Lambda_N\right|\left(c-q^{\ssup{\mu}}_k\right)\right).
	\end{equation*}
	Taking $N\rightarrow\infty$ then gives us
	\begin{equation*}
			\limsup_{N\rightarrow\infty}\frac{1}{\abs{\Lambda_N}}\log\nu_{N,\mu}(\blambda^{\ssup{k}}_N>c) \leq c - q^{\ssup{\mu}}_k -c\log\frac{c}{q^{\ssup{\mu}}_k}.
	\end{equation*}
	
	Now note that on $c>0$, the maps
	\begin{equation*}
		c\mapsto c - q^{\ssup{\mu}}_k- c\log\frac{c}{q^{\ssup{\mu}}_k} + 2^{-k}\alpha, \qquad k\in\N,
	\end{equation*}
	are differentiable, strictly decreasing, and have at most a unique zero $c^*_k$. If there does not exist such a zero, then the map is strictly negative, and it will suffice in what follows to set $c^*_k=0$. Since our maps are strictly negative for $c>c^*_k$, we only need to find a sequence $x$ such that $x_k>c^*_k$ for all $k$. Now we only need to find such an $x\in\ell_1\left(\R\right)$.
	
	Consider $x_k = c^*_k + 2^{-k}$. Therefore $x\in\ell_1\left(\R\right)$ if and only if $c^*\in\ell_1\left(\R\right)$. If we defined $c^*_k$ as a zero, then $c^*_k$ solves
	\begin{equation*}
		c^*_k\Big(1-\log\Big(\frac{c^*_k}{q^{\ssup{\mu}}_k}\Big)\Big) = q^{\ssup{\mu}}_k + 2^{-k}\alpha.
	\end{equation*}
	Otherwise, $c^*_k=0$ and
	\begin{equation*}
		c^*_k\Big(1-\log\Big(\frac{c^*_k}{q^{\ssup{\mu}}_k}\Big)\Big) = 0.
	\end{equation*}	
	So noting that $q^{\ssup{\mu}}_k,\alpha>0$ and that the sum of $q^{\ssup{\mu}}_k$ converges give us
	\begin{equation*}
		\sum_{k\in\N}c^*_k\Big(1-\log\Big(\frac{c^*_k}{q^{\ssup{\mu}}_k}\Big)\Big) \leq \alpha + \sum_{k\in\N}q^{\ssup{\mu}}_k < \infty.
	\end{equation*}
	
	Suppose, for contradiction, that $\sum_{k\in\N}c^*_k=\infty$. Then, in order for the left hand side of the above inequality to converge, we require $1-\log\Big(\frac{c^*_k}{q^{\ssup{\mu}}_k}\Big)\rightarrow0$ as $k\rightarrow\infty$. Consequently, there exists a $K\geq1$ such that for $k\geq K$, $\frac{c^*_k}{q^\star_k}<3$, and hence
	\begin{equation*}
		\sum_{k\geq K}c^*_k \leq 3\sum_{k\geq K}q^{\ssup{\mu}}_k\leq 3\sum_{k\in\N}q^{\ssup{\mu}}_k<\infty.
	\end{equation*}
	We have a contradiction, and $\sum_{k\in\N}c^*_k<\infty$ as required.\qed
\end{proofsect}

Recall that $\blambda_N$ is a $\ell_1\left(\R\right)$-valued random variable with law $\nu_{N,\mu}$, and let $t\in\ell_\infty\left(\R\right)$, the dual space. Then,  the \emph{limit cumulant generating function} is given by
\begin{equation}
\L(t) = \lim_{N\rightarrow\infty}\frac{1}{\abs{\Lambda_N}}\log\mathbb{E}_{\nu_{N,\mu}}\Big[\exp\Big(\abs{\L_N}\langle t,\blambda_N\rangle\Big)\Big], \qquad t\in\ell_\infty(\R).
\end{equation}

\begin{lemma}\label{cumulant}
	The limit cumulant generating function exists and is given by
	\begin{equation*}
		\L\left(t\right) = \sum_{k\in\N}q^{\ssup{\mu}}_k\left(\e^{t_k}-1\right)<\infty,\qquad t\in\ell_\infty\left(\R\right).
	\end{equation*}
	Moreover, $\L$ is G{\^a}teaux differentiable, lower semicontinuous, and strictly convex.
\end{lemma}

\begin{proof}
	First, let us evaluate the logarithmic moment generating function. Recall, that our reference process is  a independent superposition of countably many independent marked Poisson point processes.  Denote the marginal law of $\blambda^{\ssup{k}}_N$ by $\nu^{(k)}_N$, then we have,
$$
\begin{aligned}
	\L\left(t\right) &=   \lim_{N\rightarrow\infty}\frac{1}{\abs{\Lambda_N}}\log\mathbb{E}_{\nu_{N,\mu}}\Big[\exp\Big(\abs{\L_N}\langle t,\blambda_N\rangle\Big)\Big]= \lim_{N\rightarrow\infty}\frac{1}{\left|\Lambda_N\right|}\sum_{k\in\N}\log \mathbb{E}_{\nu^{(k)}_N}\left[\exp\left(\abs{\Lambda_N}t_k\blambda^{\ssup{k}}_N\right)\right]\\
	&= \lim_{N\rightarrow\infty}\frac{1}{\left|\Lambda_N\right|}\sum_{k\in\N}\left|\Lambda_N\right|q^{\ssup{\mu}}_k\left(\ex^{t_k}-1\right)= \sum_{k\in\N}q^{\ssup{\mu}}_k\left(\ex^{t_k}-1\right).
	\end{aligned}
$$
	Here, we were able to evaluate the expectation with respect to $\nu^{(k)}_N$ using the logarithmic moment generating function for a Poisson distribution and recalling that $\blambda^{\ssup{k}}_N$ has mean $q^{\ssup{\mu}}_k$.
	
	To see that $\L\left(t\right)$ is finite, note that $t\in\ell_\infty\left(\R\right)$ implies that $T:=\sup_{j\in\N}t_j$ is finite. Hence 
	\begin{equation*}
		\L\left(t\right)\leq\left(\ex^T-1\right)\overline{q}^{\ssup{\mu}}<\infty.
	\end{equation*}
	
	To confirm G{\^a}teaux differentiability, let $t,s\in\ell_\infty\left(\R\right)$ and consider
	\begin{equation*}
		\frac{\d}{\d \epsilon}\L\left(t+\epsilon s\right) = \sum_{k\in\N}q^{\ssup{\mu}}_k s_k \ex^{t_k+\epsilon s_k}.
	\end{equation*}
	This sum is finite because $t$ and $s$ are bounded above and $q^{\ssup{\mu}}\in\ell_1\left(\R\right)$ for $\mu <0$. In particular, the derivative is defined at $\epsilon=0$, and hence $\L$ is G{\^a}teaux differentiable.
	
	Lower semicontinuity is an immediate consequence of Fatou's Lemma. For any sequence $t^{(n)}\rightarrow t$ in $\ell_\infty\left(\R\right)$,
	\begin{equation*}
		\liminf_{n\rightarrow\infty}\L\left(t^{(n)}\right) = \liminf_{n\rightarrow\infty}\sum_{k\in\N}q^{\ssup{\mu}}_k\left(\ex^{t^{(n)}_k}-1\right) \geq \sum_{k\in\N}q^{\ssup{\mu}}_k\left(\ex^{t_k}-1\right) = \L\left(t\right).
	\end{equation*}
	
	To show strict convexity, consider distinct $t,s\in\ell_\infty\left(\R\right)$ and $\lambda\in\left[0,1\right]$. Then
	\begin{align*}
		\L\left(\lambda s + \left(1-\lambda\right)t\right) &= \sum_{k\in\N}q^{\ssup{\mu}}_k\left(\ex^{\lambda s_k + \left(1-\lambda\right)t_k}-1\right)\\
		&< \lambda\sum_{k\in\N}q^{\ssup{\mu}}_k \ex^{s_k} + \left(1-\lambda\right)\sum_{k\in\N}q^{\ssup{\mu}}_k\ex^{t_k} - \overline{q}^{\ssup{\mu}}\\
		&= \lambda \L\left(s\right) + \left(1-\lambda\right)\L\left(t\right),
	\end{align*}
	where the inequality follows from the strict convexity of the exponential function.
\end{proof}

\begin{remark}
	If we do not have $\mu\leq0$, then we do not have $\L\left(t\right)<\infty$ for all $t\in\ell_\infty\left(\R\right)$. To see this, let $t$ be a constant sequence $t_k=C>0$. Then $\L\left(t\right) = C\overline{q}^{\ssup{\mu}} = \infty$ unless $\mu\leq0$.\hfill $ \diamond $
\end{remark}

\begin{lemma}\label{legendre}
	For all $x\in\ell_1\left(\R\right)$, we have
	\begin{equation*}
		\L^*\left(x\right) := \sup_{t\in\ell_\infty\left(\R\right)}\left\{\left\langle t,x\right\rangle - \L\left(t\right)\right\} =  I(x).
	\end{equation*}
\end{lemma}

\begin{proof}
	Let $g_x\left(t\right)$ denote the functional we wish to maximise in the definition of $\L^*$, so
	\begin{equation*}
	g_x\left(t\right) = \sum^\infty_{k=1}\left[x_kt_k + q^{\ssup{\mu}}_k\left(1-\ex^{t_k}\right)\right].
	\end{equation*}
	
	First let us consider $x\in\ell_1\left(\R\right)\setminus \ell_1\left(\R_+\right)$. Hence there exists an index $k^\prime$ such that $x_{k^\prime}<0$. Now let $t^{(T)}=-T\delta_{k^\prime}\in\ell_\infty\left(\R\right)$. Therefore
	\begin{equation*}
	\L^*\left(x\right)\geq g_x\left(t^{(T)}\right) = -Tx_{k^\prime} + q^{\ssup{\mu}}_{k^\prime}\left(1-\ex^{-T}\right)
	\xrightarrow{T\rightarrow\infty}+\infty.
	\end{equation*}
	This means $\L^*\left(x\right) = +\infty =I\left(x\right)$ for all $x\in\ell_1\left(\R\right)\setminus \ell_1\left(\R_+\right)$.
	
	To show the required inequality on $\ell_1\left(\R\right)$, let us now search for critical points of $g_x$. We know $g_x$ is G{\^a}teaux differentiable because it is the sum of a linear term (with coefficients $x\in\ell_1\left(\R\right)$) and a G{\^a}teaux differentiable term. Taking the G{\^a}teaux derivative of $g_x$ gives us
	\begin{equation*}
	\d g_x\left(t;s\right) = \sum^\infty_{k=1}s_k\left(x_k - q^{\ssup{\mu}}_k\ex^{t_k}\right), \qquad\forall t,s\in\ell_\infty\left(\R\right).
	\end{equation*}
	Now $t$ is a critical point if and only if $\d g_x\left(t;s\right)=0$ $\forall s \in\ell_\infty\left(\R\right)$. This means that we want to investigate the sequence $\tilde{t}_k = \log \frac{x_k}{q^\mu_k}$. If $\tilde{t}\in\ell_\infty\left(\R\right)$, then this gives us the supremum, and a simple substitution tells us that $\L^*\left(x\right) = I\left(x\right)$ for such $x$. Unfortunately, this is not necessarily the case.
	
	Nevertheless, these critical points will give us the supremum over all sequences in $\left(\R\cup\left\{-\infty\right\}\right)^\N$. Since $\ell_\infty\left(\R\right)\subset\left(\R\cup\left\{-\infty\right\}\right)^\N$, we have
	\begin{equation}
		\L^*\left(x\right) = \sup_{t\in\ell_\infty\left(\R\right)}g_x\left(t\right) \leq \sup_{t\in\left(\R\cup\left\{-\infty\right\}\right)^\N}g_x\left(t\right) = I\left(x\right).
	\end{equation}
	
	To find the reverse inequality, let us consider
	\begin{equation*}
	t^{(K)}_k =
	\begin{cases}
	\mathds{1}\left\{k\leq K\right\} \log\frac{x_k}{q^{\ssup{\mu}}_k} &,  x_k\ne0,\\
	-K\mathds{1}\left\{k\leq K\right\} &, x_k=0.
	\end{cases}
	\end{equation*}
	Since $t^{(K)}$ truncates, it is clearly in $\ell_\infty\left(\R\right)$ for all $K$. Now let us substitute it into $g_x$.
	\begin{align*}
	g_x\left(t^{(K)}\right) &= \sum_{k\leq K:x_k\ne0}\left(x_k\log\frac{x_k}{q^{\ssup{\mu}}_k}-x_k+q^{\ssup{\mu}}_k\right) +  \sum_{k\leq K:x_k=0} q^{\ssup{\mu}}_k\left(1-\ex^{-K}\right)\\
	&= \sum^K_{k=1}\left[x_k\left(\log\frac{x_k}{q^{\ssup{\mu}}_k}-1\right)+q^{\ssup{\mu}}_k\right] - \ex^{-K}\sum_{k\leq K\colon x_k=0} q^{\ssup{\mu}}_k\\
	&\xrightarrow{K\rightarrow\infty} I \left(x\right).
	\end{align*}
	In the second equality, we have used the convention that $0\log0=0$. The limit holds because the sum of the $q^{\ssup{\mu}}_k$ converges, and the sum defining $ I \left(x\right)$ converges.
	
	This sequence $\left(t^{(K)}\right)_{K\in\N}$ shows that for $x\in\ell_1\left(\R_+\right)$,
	\begin{equation*}
	\L^*\left(x\right) = \sup_{t\in\ell_\infty\left(\R\right)}g_x\left(t\right)\geq I\left(x\right),
	\end{equation*}
	as required.
\end{proof}
Using  Baldi's Theorem in Lemma~\ref{Baldi} in conjunction with Lemma~\ref{exptight}, ~\ref{cumulant}, ~\ref{legendre}, we conclude with the statement in Proposition~\ref{THM-Ideal}.\qed
\subsection{Proof of Theorem~\ref{THM-CMF} - Cycle Mean Field LDP}\label{sec-CMF}
Recall that $\mathcal{N}_{k}$ is a Poisson random variable describing the number of cycles in our loop soup with length $\beta k$, and that the total cycle number is the random variable $N_{\L_N}=\sum_{k\in\N} \Ncal_{k} $. 
We are going to apply Varadahn's Lemma in \cite[Theorem~4.3.1]{DZ09}.  To show continuity of $H^{\CMF}$, let us show sequential continuity. This implies continuity since $\ell_1(\R)$ is a metric space. Let $x^{(n)}\rightarrow x$ be a convergent sequence in $\ell_1(\R)$, so $\lim_{n\rightarrow\infty}\sum_{j\in\N}\abs{x^{(n)}_k-x_k}=0$.
 Let $S(x):=\sum_{k\in\N}x_k$. Then
	\begin{equation*}
	\lim_{n\rightarrow\infty}\abs{S(x^{(n)})-S(x)}= \lim_{n\rightarrow\infty}\big|\sum_{k\in\N}x^{(n)}_k - \sum_{k\in\N}x_k\big| \leq \lim_{n\rightarrow\infty}\sum_{k\in\N}\abs{x^{(n)}_k-x_k} =0.
	\end{equation*}
	Hence $S$ is continuous. We can then write the Hamiltonian as the composition of continuous functions $H^{\CMF} = T\circ S$, where $T\colon \R\rightarrow\R$, $x\rightarrow \frac{a}{2}x^2$.
We can now simply apply Varadhan's Lemma. The lower bound
\begin{equation}\label{CMFlower}
\liminf_{N\to\infty}\frac{1}{\abs{\L_N}}\log \E_{\nu_{N,\alpha}}\Big[\ex^{-\beta\abs{\L_N} H^\CMF}\Big]\ge \sup_{x\in\ell_1(\R)}\big\{-\beta H^\CMF(x)-I_\alpha(x)\big\}
\end{equation}
follows easily with \cite[Lemma~4.3.4]{DZ09}.  For the corresponding upper bound we simply note that the tail-condition in \cite[Theorem~4.3.1]{DZ09} holds due to $ H^\CMF(x)\ge 0 $ for all $ x\in\ell_1(\R_+) $. Therefore, with \cite[Lemma~4.3.6]{DZ09} we obtain the corresponding upper bound
\begin{equation}\label{CMFupper}
\limsup_{N\to\infty}\frac{1}{\abs{\L_N}}\log \E_{\nu_{N,\alpha}}\Big[\ex^{-\beta\abs{\L_N} H^\CMF}\Big]\le \sup_{x\in\ell_1(\R)}\big\{-\beta H^\CMF(x)-I_\alpha(x)\big\}.
\end{equation}
We conclude with the statement in Theorem~\ref{THM-CMF} by combining the lower bound \eqref{CMFlower} and the upper bound \eqref{CMFupper}.\qed

\subsection{Proof of Theorem~\ref{THM-PMF} - Particle Mean Field LDP}\label{sec-PMF}
This subsection gives the proof for Theorem~\ref{THM-PMF}.  To prove the large deviation principle for $ \nu_{N,\mu,\alpha}^\PMF $ one would simply use Varadhan's Lemma. However, the first term of the Hamiltonian $ H^\PMF_\mu $ for $ \mu>0 $ is not lower semicontinuous whereas the second term is only lower semicontinuous.  Using \cite{GZ93}, one would arrive at lower and upper bounds for $ \E_{\nu_{N,\alpha}}[\ex^{-\abs{\L_N}\beta H^\PMF_\mu}] $ using the upper and the lower semicontinuous regularisation of $ H^\PMF_\mu $, respectively.    Unfortunately, the upper semicontinuous regularisation of the Hamiltonian equals infinity, and thus it does not provide a lower bound for the large deviation principle. Our strategy is therefore twofold. For the large deviation upper bound we use the lower semicontinuous regularisation in conjunction with the corresponding bound in Varadhan's Lemma.  We obtain the corresponding large deviation lower bound by conditioning that the empirical cycle count is supported on a finite-dimensional subspace. On this event we can replace our measure by the corresponding measure with finite dimensional mark space. On this subspace the Hamiltonian is in fact continuous and thus application of Varadhan's Lemma provides a lower bound. To remove the cutoff parameter we will construct finite-dimensional sequences approximating the infimum of the corresponding lower bound.
We start with a couple of observations.

\begin{lemma}
For all $ \mu>0  $, the lower semicontinuous regularisation of $ H^\PMF_\mu $ is given as
\begin{equation}\label{lsc}
H^{\PMF}_{\mu,l.s.c.}(x)=H^\PMF_\mu(x)-\frac{1}{2a}\big(\mu-aD(x)\big)_+^2=\begin{cases} -\mu D(x)+\frac{a}{2}D(x)^2 &,D(x)\ge \frac{\mu}{a},\\-\frac{\mu^2}{2a}&, D(x)<\frac{\mu}{a}, \end{cases}\quad x\in\ell_1(\R_+), 
\end{equation}
whereas for all $ \mu\le 0 $, $ H^\PMF_{\mu,l.s.c.}\equiv H^\PMF_\mu $.
\end{lemma}

\begin{figure}
	\centering
	\begin{tikzpicture}[scale=2]
	\draw[->] (0,0) -- (2.5,0) node[below]{$D$};
	\draw[->] (0,-1.5) -- (0,1) node[left]{$H^{\PMF}_{\mu,l.s.c.}$};
	\draw[very thick] (1,-1) parabola (2.4,1);
	\draw[very thick] (0,-1) node[left]{$-\frac{\mu^2}{2a}$} -- (1,-1);
	\draw[dashed] (1,-1) parabola (0,0);
	\draw[dashed] (1,-1) -- (1,0) node[above]{$\frac{\mu}{a}$};
	\end{tikzpicture}
	\caption{Sketch of $H^\PMF_{\mu,l.s.c.}$ as a function of the total particle density $D$.}
\end{figure}
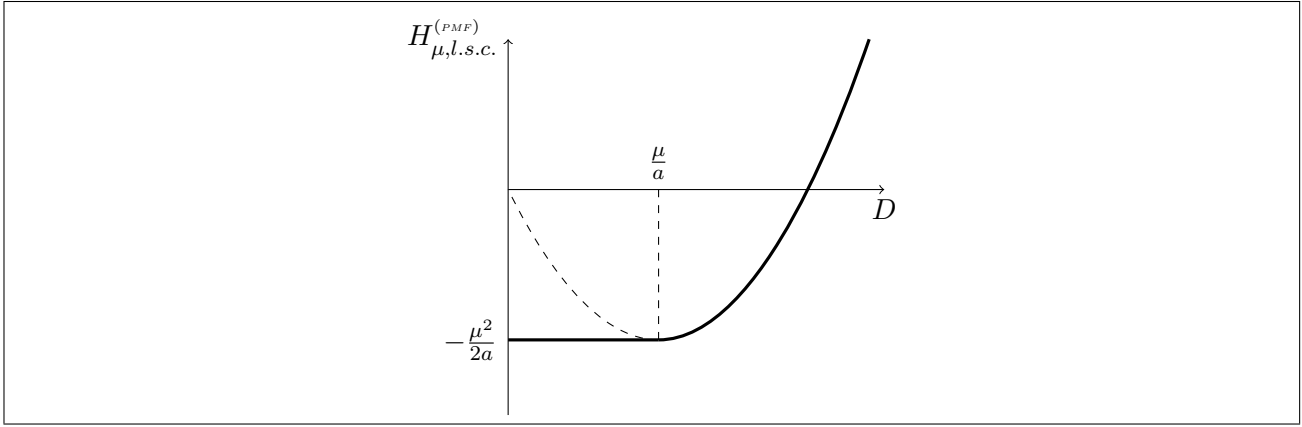

\begin{proofsect}{Proof}
Clearly, $ H^\PMF_{\mu,l.s.c.}(x)\le H^\PMF_\mu(x) $ for all $ x\in\ell_1(\R_+) $.  Suppose that $ \mu\le 0 $. Then the Hamiltonian $ H^\PMF_\mu $ is lower semicontinuous and $ H^\PMF_{\mu,l.s.c.}\equiv H^\PMF_\mu $.
Suppose now that $ \mu>0 $.  Denote $ h(x) $ the right hand side of \eqref{lsc}. For all $ x\in\ell_1(\R_+) $ with $ D(x) <\frac{\mu}{a} $ and any sequence $ (x_n)_{n\in\N} $ with $ x_n\to\ x $ in $ \ell_1(\R) $ we see that   $ h(x_n)=-\frac{\mu^2}{a} =h(x) $ for sufficiently large $ n $, and thus $ \liminf_{n\to\infty} h(x_n)=h(x) $. Fix $ x\in\ell_1(\R_+) $ with $ D(x)>\frac{\mu}{a} $. The function $ f\colon [\frac{\mu}{a},\infty)\to\R\cup\{\pm\infty\}, z\mapsto f(z) =-\mu z+\frac{a}{2} z^2 $ is continuous and increasing, and $ h=f\circ D $. The particle density, $ D\colon\ell_1(\R_+)\to[0,\infty], x\mapsto \sum_{k=1}^\infty k x_k  $, is lower semicontinuous. For all $ x\in\ell_1(\R_+) $ with $ D(x)=\infty $, we get $  h(x)=\infty $ and for any $ (x_n)_{n\in\N} $ with $ x_n\to x $ as $ n\to\infty $ we see  that $ h(x_n)\to\infty $ as  $ n\to \infty $. Suppose now that $ x_0\in\ell_1(\R_+) $ with $ D(x_0)\in (\frac{\mu}{a},\infty) $. For any $ \eps >0 $ there exists $ \widetilde{\eps}> 0  $ such that  $\abs{f(D(x_0)-\widetilde{\eps})-f(D(x_0))}<\eps $ using the continuity of $f$. As $ D$ is lower semicontinuous, there exists a neighbourhood $ \Ucal(x_x) $ such that
$$
D(x)\ge D(x_0)-\widetilde{\eps} \quad\mbox{ for all } x\in \Ucal(x_0).
$$
Using the fact that $ f $ is increasing, we obtain
$$
\begin{aligned}
f(D(x))\ge f(D(x_0)-\widetilde{\eps})\ge f(D(x_0))-\widetilde{\eps}\quad \mbox{ for all } \in\Ucal(x_0),
\end{aligned}
$$
and thus proving that $h$ is lower semicontinuous. We are left to show that $ h$ is the largest lower semicontinuous function smaller or equal to $ H^\PMF_\mu $. It suffices to consider only $ x^0\in\ell_1(\R_+) $ with $ D(x^0)<\frac{\mu}{a} $ as in the other cases $ h$ coincides with $ H^\PMF_\mu $. Define $ x^{\ssup{n}}\in\ell_1(\R_+) $ by
$$
x^{\ssup{n}}_k=x^0_k+\frac{1}{n}\big(\frac{\mu}{a}-D(x^0)\big)\1\{n=k\}, \quad k\in\N.
$$
For all $ n\in\N $, $ D(x^{\ssup{n}})=\frac{\mu}{a} $, and thus
$$
-\frac{\mu^2}{2a}=H^\PMF_{\mu,l.s.c.}(x^0)\le \liminf_{n\to\infty} H^\PMF_{\mu,l.s.c.}(x^{\ssup{n}})\le \liminf_{n\to\infty} H^\PMF_\mu(x^{\ssup{n}})=-\frac{\mu^2}{2a}.
$$
The greatest lower semicontinuous function that satisfies this is the right hand side in \eqref{lsc}.
 \qed
\end{proofsect}

\begin{prop}[\textbf{Upper  bound PMF-model}]
For all $ \mu\in\R, \alpha\le 0 $, and $ a>0 $,
\begin{equation}\label{ubound}
\limsup_{N\to\infty}\frac{1}{\abs{\L_N}}\log \E_{\nu_{N,\alpha}}\Big[\ex^{-\beta\abs{\L_N} H^\PMF_\mu}\Big]\le -\inf_{x\in\ell_1(\R_+)}\Big\{ I(x)+\beta H^\PMF_{\mu+\alpha,l.s.c.}(x)\Big\}.
\end{equation}
\end{prop}

\begin{proofsect}{Proof}
The statement follows easily with the upper bound estimate in Varadhan's Lemma given in \cite[Lemma~4.3.6]{DZ09} using  the inequality $ H^\PMF_\mu(x)\ge H^\PMF_{\mu,l.s.c.}(x)  $ for all $ x\in\ell_1(\R_+) $, the lower semicontinuity of $ H^\PMF_{\mu,l.s.c.} $,  and the fact that $ H^\PMF_{\mu,l.s.c.}(x)\ge -\frac{\mu^2}{2a} $. The later estimate provides the tail-condition necessary to apply \cite[Lemma~4.3.6]{DZ09}.
\qed
\end{proofsect}

\begin{prop}[\textbf{Lower  bound PMF-model}]\label{PMFlower}
For all $ \mu\in\R, \alpha\le 0 $, and $ a>0 $,
\begin{equation}\label{lbound}
\liminf_{N\to\infty}\frac{1}{\abs{\L_N}}\log \E_{\nu_{N,\alpha}}\Big[\ex^{-\beta\abs{\L_N} H^\PMF_\mu}\Big]\ge -\inf_{x\in\ell_1(\R_+)}\Big\{ I(x)+\beta H^\PMF_{\mu+\alpha,l.s.c.}(x)\Big\}.
\end{equation}
\end{prop}

\begin{proofsect}{Proof}
The strategy for proving the lower bound is to first introduce a cut-off parameter as done in \cite{ACK}, that is, we change the measure to obtain a finite-dimensional problem which gives continuity of the Hamiltonian and thus a large deviation lower bound for the finite-dimensional space. The final step is then to remove the cut-off parameter. As our Hamiltonian is not positive,  removing of the cut-off is not as straightforward as in \cite{ACK}.  We thus need to construct a sequence for the finite dimensional spaces which allows for energy estimates and at the same time gives convergence towards the lower semicontinuous regularisation. The last step is crucial as  a  lower bound can be  obtained via the upper semicontinuous regularisation which in our case is identical to infinity.

\medskip

\textbf{Step 1: Restriction of the mark space.}  We will approximate the mark space $E$ by the cut-off version
$$
E^{\ssup{K}}=\bigcup_{k=1}^K \Ccal_{k},  \quad K\in\N.
$$
The Poisson reference process on $ \R^d\times E^{\ssup{K}} $ is denoted by $ {\tt Q}^{\ssup{K}} $, and the corresponding measure on $ \R^K $ which is isomorphic to $ \pi_K(\ell_1(\R)) $ with $ \pi_K\colon\ell_1(\R)\to\R^k,x\mapsto (x_1,\ldots,x_K) $, is denoted $ \nu_{N,\alpha}^{\ssup{K}}={\tt Q}^{\ssup{K}}\circ \blambda_{\L_N}^{-1}=\nu_{N,\alpha}\circ\pi_K^{-1} $.   We obtain a large deviation principle for the cut-off version in the following.

\begin{lemma}
	\label{Thm:LDPidealRJ}
	For given $K\in\N $ and $\alpha\leq0$, the sequence $(\nu_{N,\alpha}^{\ssup{K}})_{N\in\N}$ satisfies an LDP on $\R^K$ with rate $\abs{\Lambda_N}$ and rate function
	\begin{equation}\label{RFK}
	I^{\ssup{K}}_\alpha(x) = 
	\begin{cases}
	\sum^K_{k=1}\big(x_k\log\frac{x_k}{q^{\ssup{\alpha}}_k} - x_k + q^{\ssup{\alpha}}_k\big), &x\in\R^K_+, \\
	+\infty, &\text{otherwise.}
	\end{cases}
	\end{equation}
\end{lemma}

\begin{proofsect}{Proof}
	Since the projection $\pi_K$ is continuous, we can apply the contraction principle to obtain a variational form of the rate function
	\begin{equation*}
	I^{\ssup{K}}_\alpha(x) = \inf_{\widetilde{x}\in\ell_1(\R)\colon\pi_K(\widetilde{x})=x}  I_\alpha(\widetilde{x}),
	\end{equation*}
	where $ I_\alpha $ is the rate function for $\nu_{N,\alpha}$ in Proposition~\ref{THM-Ideal}.
	
	If $x\in\R^K\setminus \R^K_+$, then any element $\widetilde{x}\in\ell_1(\R)$ with $\pi_K(\widetilde{x})=x$ is such that $\widetilde{x}\in\ell_1(\R)\setminus\ell_1(\R_+)$, and hence $ I(\widetilde{x})= +\infty$. 
	
	Now suppose $x\in\R^K_+$. Then if a corresponding $\widetilde{x}\in\ell_1(\R)$ was in fact in $\widetilde{x}\in\ell_1(\R)\setminus\ell_1(\R_+)$, then $I(\widetilde{x})= +\infty$. However, for all $\widetilde{x}\in\ell_1(\R_+)$ we have $ I(\widetilde{x}) < +\infty$. Hence we only need consider $\widetilde{x}\in\ell_1(\R)$ with positive entries. For $\widetilde{x}\in\ell_1(\R_+)$ with $\pi_K(\widetilde{x})=x\in\R^K_+$,
	\begin{equation*}
	I_\alpha(\widetilde{x}) = \sum^K_{k=1} x_k\Big(\log\frac{x_k}{q^{\ssup{\alpha}}_k}-1\Big) + \sum^\infty_{k=K+1}\widetilde{x}_k\Big(\log\frac{\widetilde{x}_k}{q^{\ssup{\alpha}}_k}-1\Big) + \overline{q}^{\ssup{\alpha}}.
	\end{equation*}
	Thus is suffices to minimise the second term, which can be done term-wise. The infimum is given by $\widetilde{x}=y$, where
	\begin{equation*}
	y_k =
	\begin{cases}
	x_j, &k=1,\ldots,K\\
	q^{\ssup{\alpha}}_k, &k>K.
	\end{cases}
	\end{equation*}
	
	This gives us 
	\begin{equation*}
	I^{\ssup{K}}_\alpha(x) = I_\alpha(\widetilde{x}) = \sum^K_{k=1}x_k\Big(\log\frac{x_k}{q^{\ssup{\alpha}}_k}-1\Big) - \sum^\infty_{j=K+1}q^{\ssup{\alpha}}_k + \overline{q}^{\ssup{\alpha}},
	\end{equation*}
	exactly as required.
\qed
\end{proofsect}

\medskip

\textbf{Step 2: Lower bound.}   
 We obtain inserting an indicator the lower bound,
 $$
 \E_{\nu_{N,\alpha}}\Big[\ex^{-\beta\abs{\L_N} H^\PMF_\mu}\Big]\ge \E_{\nu_{N,\alpha}}\Big[\ex^{-\beta\abs{\L_N} H^\PMF_\mu}\1\{\nu_{N,\alpha}(\R^K)=1\}\Big], 
 $$
where we identified $ \R^K $ with the corresponding subspace in $ \ell_1(\R) $. On that event we can replace $ H^\PMF_\mu $ by $ H_{\mu,K}^\PMF $, where
$$
 H_{\mu,K}^\PMF(x)=-\mu\sum_{k=1}^K k x_k +\frac{a}{2}\Big(\sum_{k=1}^K k x_k\Big)^2,
 $$ 
and $ \E_{\nu_{N,\alpha}} $ by $ \E_{\nu_{N,\alpha}^{\ssup{K}}} $. The finite-dimensional approximation $ H^\PMF_{\mu,K} $ is in fact continuous, and thus we obtain a large deviation lower bound  using Lemma~\ref{Thm:LDPidealRJ} and Varadhan's Lemma, see \cite[Lemma~4.3.4]{DZ09},
\begin{equation}
\begin{aligned}
\liminf_{N\to\infty}\frac{1}{\abs{\L_N}}\log \E_{\nu_{N,\alpha}}\Big[\ex^{-\beta\abs{\L_N} H^\PMF_\mu}\Big] \ge -\inf_{x\in\R^K}  \big\{ I^{\ssup{K}}(x)+\beta H^\PMF_{\mu+\alpha,K}(x) \big\}\end{aligned}
\end{equation}

\medskip

\textbf{Step 3: Removing the cut-off parameter.} We are left to remove the cut-off by taking $ K\to\infty $ and to prove that the $K\to\infty $ limit of $ H^\PMF_{\mu,K} $ is replaced by the lower semicontinuous regularisation of $ H^\PMF_\mu $.
\begin{lemma}
\begin{equation}\label{removing}
\limsup_{K\to\infty} \inf_{x\in\R^K}\big\{ I^{\ssup{K}}(x)+\beta H^\PMF_{\mu+\alpha,K}(x)\big\} \le \inf_{x\in\ell_1(\R)}\big\{ I(x)+\beta H^\PMF_{\mu+\alpha,l.s.c.}(x)\big\}.
\end{equation}
\end{lemma}
\begin{proofsect}{Proof}
Fix $ \widetilde{x}\in \ell_1(\R_+) $ satisfying $ I(\widetilde{x})+\beta H^\PMF_{\mu+\alpha}(\widetilde{x}) <\infty $. For $ K\in\N$, consider $ x^K=\pi_K(\widetilde{x}) $. By \eqref{RFK}, we have $ I^{\ssup{K}}(x^K)\le I(\widetilde{x}) $. We shall replace  $ x^K $ by $ \widehat{x}^K $, defined as  
$$
\widehat{x}^K_k=\begin{cases}  x^K_k, &, k=1,\ldots, K-1,\\
x^K_K +\frac{1}{K}\big(\frac{\mu+\alpha}{a}-D(\widetilde{x})\big)_+, &, k=K.\end{cases}
$$
Clearly, $ \norm{x^K-\widehat{x}^K}_{\ell_1(\R)} \to 0 $ as $ K\to\infty $, and $ x^K\to\widetilde{x} $ as $K\to\infty $. Furthermore,
$$
I^{\ssup{K}}(\widehat{x}^K)\le I^{\ssup{K}}(x^K)+\abs{I^{\ssup{K}}(\widehat{x}^K)-I^{\ssup{K}}(x^K)}\le I(\widetilde{x}) + \bigO(\frac{1}{K}\log K).
$$
We turn to the energy term which needs extra care as the Hamiltonian is not positive. Observe that
$$
D(\widehat{x}^K)=\sum_{k=1}^K k\widehat{x}^K_k =\begin{cases}  \frac{\mu+\alpha}{a} +D(x^K)-D(\widetilde{x}) &, D(\widetilde{x}) <\frac{\mu+\alpha}{a},\\ D(x^K) &, D(\widetilde{x})\ge \frac{\mu+\alpha}{a}. \end{cases} 
$$
Assume that $ D(\widetilde{x}) \ge \frac{\mu+\alpha}{a} $.   Then,
$$
H^\PMF_{\mu+\alpha,K}(\widehat{x}^K)=-\left(\mu+\alpha\right)\sum_{k=1}^K k x^K_k +\frac{a}{2}\Big(\sum_{k=1}^Kk x^K_k\Big)^2\le  H^\PMF_{\mu+\alpha,l.s.c.}(\widetilde{x}).
$$
In the other case, $ D(\widetilde{x})<\frac{\mu+\alpha}{a} $, for every  $ \eps>0 $ choose $ K$ sufficiently large such that $ \abs{D(x^K)-D(\widetilde{x})}<\eps $, and estimate
$$
\begin{aligned}
H^\PMF_{\mu+\alpha,K}(\widehat{x}^K)&=-\mu\big(\frac{\mu+\alpha}{a}+D(x^K)-D(\widetilde{x})\big)+\frac{a}{2}\big(\frac{\mu+\alpha}{a}+\big(D(x^K)-D(\widetilde{x})\big)\big)^2\\
& \le H^\PMF_{\mu+\alpha,l.s.c.}(\widetilde{x}) +2(\mu+\alpha)\eps+\eps^2\frac{a}{2}
\end{aligned}
$$
to finish the proof for \eqref{removing}.
\qed
\end{proofsect}

We finally combine Proposition~\ref{ubound} and Proposition~\ref{PMFlower}  to finish the proof for Theorem~\ref{THM-PMF}.

\qed
\end{proofsect}

\subsection{Proof of Theorem~\ref{Thm:HYL}} \label{sec-HYL}
This section proves Theorem~\ref{Thm:HYL} using  techniques which are similar to the ones in the  proof in Section~\ref{sec-PMF}. However, there are significant differences  to address due to the fact that the Hamiltonian $ H^\HYL_\mu $ has positive and negative contributions. We rewrite the Hamiltonian in two equivalent ways for any $ a\ge b>0 $ and $ \mu\in\R $,
\begin{align}
H^\HYL_\mu(x)&=-\mu\sum_{k=1}^\infty k x_k +\frac{(a-b)}{2}\Big(\sum_{k=1}^\infty k x_k\Big)^2+\frac{b}{2}\sum_{\heap{j,k=1}{j\not= k}}^\infty jkx_jx_k\label{version1}\\
&=-\mu\sum_{k=1}^\infty k x_k +\frac{a}{2}\Big(\sum_{k=1}^\infty k x_k\Big)^2-\frac{b}{2}\sum_{k=1}^\infty k^2x_k^2.\label{version2}
\end{align}
Note that that the right hand side in \eqref{version1} is the sum of a PMF Hamiltonian with interaction strength $ (a-b) $ and a lower semicontinuous and non-negative term. On the other hand, \eqref{version2}  expresses $ H^\HYL $ as the sum of a PMF Hamiltonian, and an upper semicontinuous and non-positive term. Let us introduce the following notations
\begin{align*}
H_{\mu,a}^\PMF(x)=-\mu\sum_{k=1}^\infty k x_k +\frac{a}{2}\Big(\sum_{k=1}^\infty k x_k\Big)^2, \quad a>0,\\
H_+(x)=\frac{b}{2}\sum_{\heap{j,k=1}{j\not= k}}^\infty jkx_jx_k,\qquad
H_-(x)= -\frac{b}{2}\sum_{k=1}^\infty k^2x_k^2.
\end{align*}

Thus $ H^\HYL_\mu = H_{\mu,a}^\PMF + H_-=H^\PMF_{\mu,a-b}+H_+$.
\begin{lemma}
For $ b>0 $, $H_- $ is upper semicontinuous and $H_+ $ is lower semicontinuous on $ \ell_1(\R_+) $. 
\end{lemma} 
\begin{proofsect}{Proof}
We shall show that $\sum_{\heap{j,k=1}{j\not= k}}^\infty jkx_jx_k $ and $  \sum_{k=1}^\infty k^2x_k^2 $ are both lower semicontinuous on $ \ell_1(\R_+) $. Suppose $ x^{\ssup{n}}\to x $ in $ \ell_1(\R_+) $. Clearly,
$$
\abs{x^{\ssup{n}}_k-x_k}\le \norm{x^{\ssup{n}}-x}_{\ell_1}\quad\mbox{ for all }k\in\N.
$$
Furthermore, due to the $ \ell_1 $-convergence and $ \ell_1\subset\ell_\infty $,  the term $ \abs{x^{\ssup{n}}_k-x_k} $ is bounded in both $k$ and $ n$. Hence
$$
\abs{(x^{\ssup{n}}_k)^2-x_k^2}=\abs{x^{\ssup{n}}_k-x_k}\abs{x^{\ssup{n}}_k+x_k}\to 0\; \mbox{ as } n\to\infty.
$$
Applying Fatou's Lemma here proves that $ H_- $ is upper semicontinuous. Similarly, for all $ (j,k)\in\N^2 $,
$$
\abs{x_j^{\ssup{n}}x_k^{\ssup{k}}-x_jx_k}\le \abs{x^{\ssup{n}}_j}\abs{x^{\ssup{n}}_k-x_k}+\abs{x_k}\abs{x^{\ssup{n}}_j-x_j}\to 0 \; \mbox{ as } n\to\infty.
$$
This convergence in conjunction with Fatou's Lemma shows that $ H_+$ is lower semicontinuous.
\qed
\end{proofsect}

\begin{lemma}\label{L:regHYL}
For $ a>b $, the $ \ell_1(\R_+) $ lower semicontinuous regularisation of $ H^\HYL_\mu $ is given by
\begin{equation}\label{lscHYL}
H^\HYL_{\mu,l.s.c.}(x)=H^\HYL_\mu(x) -\frac{\big(\mu-aD(x)\big)_+^2}{2(a-b)},\quad x\in\ell_1(\R_+).
\end{equation}
\end{lemma}

\begin{proofsect}{Proof}
Denote the right hand side of \eqref{lscHYL} by $ h$. Clearly, $ h(x)\le H^\HYL_\mu(x) $ and $ H^\PMF_{\mu,(a-b),l.s.c.} (x) \le h(x)=H^\PMF_{\mu,(a-b),l.s.c.}(x)+H_+(x)\le  H^\HYL_\mu(x) $ for all $ x\in\ell_1(\R_+) $. We need to show that $ h $ is the greatest lower semicontinuous function less or equal to $ H^\HYL_\mu $.

Suppose that $ x\in\ell_1(\R_+) $ with $ D(x)=\infty $. Then since $ H^\PMF_{\mu,(a-b),l.s.c.}(x) =\infty $, we have $ h(x)=\infty $. Suppose now that $ x\in\ell_1(\R_+) $ with $ D(x) <\infty $. For any sequence $ (x_n)_{n\in\N} $ with $ x_n\to x $ as  $ n\to\infty $ there exists $ (\eps^{\ssup{n}})_{n\in\N}\subset\ell_1(\R) $ such that $ x_n=x+\eps^{\ssup{n}} $, $ \eps^{\ssup{n}}\to 0 $ as $ n\to\infty $,  and $ x+\eps^{\ssup{n}}\in\ell_1(\R_+) $. Furthermore, $ \liminf_{n\to\infty}\big(D(x^{\ssup{n}})-D(x)\big)\ge 0 $ and thus $ \liminf_{n\to\infty}D(\eps^{\ssup{n}})\ge 0$. We show that $ h $ is lower semicontinuous by proving that 
\begin{equation}\label{hlsc}
\begin{aligned}
\liminf_{n\to\infty} &\big(h(x^{\ssup{n}})-h(x)\big)=\liminf_{n\to\infty}\Big\{-\mu\big(D(x^{\ssup{n}})-D(x)\big)+\frac{a}{2}\big(D(x^{\ssup{n}})^2-D(x)^2\big)\\
&-\frac{b}{2}\Big(\sum_{k=1}^\infty k^2\big((x_k^{\ssup{n}})^2-x_k^2\big)\Big)+\frac{1}{2(a-b)}\Big(\big(\mu-aD(x)\big)_+^2 -\big(\mu-aD(x^{\ssup{n}})\big)_+^2\Big)\Big\}\ge 0 .
\end{aligned}
\end{equation}     

We write $ \eps_k^{\ssup{n}}=\eps^{+\ssup{n}}_k -\eps^{-\ssup{n}}_k $ with $ \eps^{+\ssup{n}}_k,\eps^{-\ssup{n}}_k\ge 0 $. Clearly, $ \eps^{-\ssup{n}}_k\le x_k $ for all $ k\in\N $.
We shall show that
\begin{equation}\label{limsupzero}
\limsup_{n\to\infty}\sum_{k=1}^\infty k^2 x_k\eps^{\ssup{n}}_k = 0,
\end{equation}
and that 
\begin{equation}\label{upperbmixed}
\limsup_{n\to\infty} \sum_{k=1}^\infty k^2(x_k+\eps^{\ssup{n}}_k)^2 \le \sum_{k=1}^\infty k^2x_k^2 +\limsup_{n\to\infty} \big(D(\eps^{\ssup{n}})\big)^2.
\end{equation}
As $ D(x) <\infty $, there exists $ C>0 $ such that  $ \sup_{k\in\N} k^2 x_k < C$, and thus \eqref{limsupzero} follows as
$$
\limsup_{n\to\infty} \sum_{k=1}^\infty k^2 x_k\eps^{\ssup{n}}_k\le C\limsup_{n\to\infty}\norm{\eps^{\ssup{n}}}_{\ell_1(\R)} =0.
$$
To obtain \eqref{upperbmixed} we just expand
$$
\begin{aligned}
\sum_{k=1}^\infty k^2(x_k+\eps^{\ssup{n}}_k)^2\le \sum_{k=1}^\infty k^2 x_k^2+2\sum_{k=1}^\infty k^2(x_k-\eps^{-\ssup{n}}_k)\eps^{+\ssup{n}}_k + \big(D(\eps^{+\ssup{n}})\big)^2.
\end{aligned}
$$
The middle term vanishes due to \eqref{limsupzero}. To show that 
$$ \limsup_{n\to\infty}\big(D(\eps^{+\ssup{n}})\big)^2\le \limsup_{n\to\infty}\big(D(\eps^{\ssup{n}})\big)^2 
$$ note that $ D(x)<\infty $ implies that
$$
D(\eps^{\ssup{n}})-D(\eps^{+\ssup{n}})=D(\eps^{-\ssup{n}})\le D(x)<\infty.
$$
Hence, for any $ \delta $ there exists $ K\in\N $ such that 
$$
\sum_{k=K+1}^\infty k\eps^{-\ssup{n}}_k\le\sum_{k=K+1}^\infty k x_k <\frac{\delta}{2}.
$$
On  the other hand, there exists for this $ \delta $ a $ n(K)\in\N $ such that
$$
\sum_{k=1}^K k\eps^{-\ssup{n}}_k <\frac{\delta}{2},\quad \mbox{ for all }  n\ge n(K),
$$
thus showing \eqref{upperbmixed}. We continue with
\begin{multline}\label{hlsc2}
\mbox{ r.h.s. of \eqref{hlsc} }  \ge \liminf_{n\to\infty}\Big\{  \frac{1}{2(a-b)}\Big(\big(\mu-aD(x)\big)_+^2 -\big(\mu-aD(x^{\ssup{n}})\big)_+^2\Big) -\big(\mu-aD(x)\big)D(\eps^{\ssup{n}})\big)\\ \qquad +\frac{a-b}{2}\big(D(\eps^{\ssup{n}})\big)^2\Big\}.
\end{multline}
Recall that $ \liminf_{n\to\infty}D(\eps^{\ssup{n}})\ge 0 $, and thus we know that $ (\mu-aD(x))< 0 $ implies that eventually $ (\mu-aD(x)-aD(\eps^{\ssup{n}}))<0 $ and $ (\mu-aD(x)-aD(\eps^{\ssup{n}}) +bD(\eps^{\ssup{n}}))< 0 $.
Suppose that $ \mu/a<D(x) $. Then
$$
\mbox{r.h.s. of \eqref{hlsc2} }=\liminf_{n\to\infty} -\big(\mu-aD(x)\big)D(\eps^{\ssup{n}})+\frac{a-b}{2}D(\eps^{\ssup{n}})^2\ge 0.
$$
Suppose $ \mu/a\ge D(x) $ and $ \mu-aD(x)-aD(\eps^{\ssup{n}}) \le 0 $. Then 
$$
\mbox{ r.h.s. of \eqref{hlsc2}  }\ge \frac{1}{2(a-b)}\liminf_{n\to\infty}\Big(\mu-aD(x)-aD(\eps^{\ssup{n}})+bD(\eps^{\ssup{n}})\Big)^2\ge 0,
$$ and likewise for $  \mu/a\ge D(x)   $ and $   \mu-aD(x)-aD(\eps^{\ssup{n}}) > 0  $,
$$
\begin{aligned}
\mbox{r.h.s. of \eqref{hlsc2} } & \ge \frac{1}{2(a-b)} \liminf_{n\to\infty}\Big\{\Big(\mu-aD(x)-aD(\eps^{\ssup{n}})+bD(\eps^{\ssup{n}})\Big)^2 \\ 
& \quad  - \Big(\mu-aD(x)-aD(\eps^{\ssup{n}})\Big)^2  \Big\}\ge 0.
\end{aligned}
$$

We have established \eqref{hlsc} and thus the lower semicontinuity of $ h $. We finally show that $ h$ is the largest lower semicontinuous function less or equal to $ H^\HYL_\mu $.  Using the lower semicontinuity of $ h$ and $ h\le H^\HYL_\mu $, we know that
$$
\liminf_{n\to\infty} H^\HYL_\mu(x^{\ssup{n}})\ge h(x)
$$ for any sequence $ x^{\ssup{n}} $ with $ x^{\ssup{n}}\to x $ as $ n\to\infty $. We pick now a particular sequence $ x^{\ssup{n}}=x+ \eps^{\ssup{n}} $ with
\begin{equation}\label{seqeps}
\eps^{\ssup{n}}_k=\1\{k=n\}\frac{\big(\mu-aD(x)\big)_+}{n(a-b)}
\end{equation}
to find that
$$
h(x)\le H^\HYL(x)-\frac{\big(\mu-aD(x)\big)_+}{2(a-b)}.
$$

\qed
\end{proofsect}

\begin{prop}[\textbf{Upper  bound HYL-model}]\label{prop-uboundHYL}
For all $ \mu\in\R, \alpha\le 0 $, and $ a > b\geq0 $,
\begin{equation}\label{uboundHYL}
\limsup_{N\to\infty}\frac{1}{\abs{\L_N}}\log \E_{\nu_{N,\alpha}}\Big[\ex^{-\beta\abs{\L_N} H^\HYL_\mu}\Big]\le -\inf_{x\in\ell_1(\R_+)}\Big\{ I(x)+\beta H^\HYL_{\mu+\alpha,l.s.c.}(x)\Big\}.
\end{equation}
\end{prop}

\begin{proofsect}{Proof}
The statement follows easily with the upper bound estimate in Varadhan's Lemma given in \cite[Lemma~4.3.6]{DZ09} using  the inequality $ H^\HYL_{\mu+\alpha}(x)\ge H^\HYL _{\mu+\alpha,l.s.c.}(x) \ge H^\PMF_{\mu+\alpha,(a-b),l.s.c.}(x)   $ for all $ x\in\ell_1(\R_+) $, the lower semicontinuity of $ H^\HYL_{\mu+\alpha,l.s.c.} $,  and the fact that $ H^\HYL_{\mu+\alpha,l.s.c.}(x)\ge -\frac{\left(\mu+\alpha\right)^2}{2(a-b)} $. The later estimate provides the tail-condition necessary to apply \cite[Lemma~4.3.6]{DZ09}.
\qed
\end{proofsect}

For the lower bound we are using the lower bound \eqref{lbound} for the PMF model and $ H^\HYL_\mu=H^\PMF_{\mu,a} + H_- $ with $ H_- $ being upper semicontinuous.

\begin{prop}[\textbf{Lower  bound HYL-model}]\label{prop-lboundHYL}
For all $ \mu\in\R, \alpha\le 0 $, and $ a > b\geq0 $,
\begin{equation}\label{lboundHYL}
\liminf_{N\to\infty}\frac{1}{\abs{\L_N}}\log \E_{\nu_{N,\alpha}}\Big[\ex^{-\beta\abs{\L_N} H^\HYL_\mu}\Big]\ge -\inf_{x\in\ell_1(\R_+)}\Big\{ I(x)+\beta H^\HYL_{\mu+\alpha,l.s.c.}(x)\Big\}.
\end{equation}
\end{prop}

\begin{proofsect}{Proof}
Using
$$
 \E_{\nu_{N,\alpha}}\Big[\ex^{-\beta\abs{\L_N} H^\HYL_\mu}\Big]=\E_{\nu_{N,\mu,\alpha}^\PMF}\Big[\ex^{-\beta\abs{\L_N} H_-}\Big] Z^\PMF(\beta,\mu,\alpha)
$$ 
in conjunction with the LDP in Theorem~\ref{THM-PMF} and in particular the lower bound  \eqref{lbound}  we arrive at
$$
\begin{aligned}
\liminf_{N\to\infty}&\frac{1}{\abs{\L_N}}\log\Big( \E_{\nu_{N,\mu,\alpha}^\PMF}\Big[\ex^{-\beta \abs{\L_N}H_-}\Big] Z^\PMF(\beta,\mu,\alpha)\Big)\ge -\inf_{x\in\ell_1(\R)}\big\{I^\PMF_{\mu,\alpha}(x)+\beta H_-(x)\big\} -\inf_{x\in\ell_1(\R)}\big\{I^\PMF_{\mu,\alpha}(x)\big\} \\
& =-\inf_{x\in\ell_1(\R)}\big\{ I(x)+\beta H^\PMF_{\mu+\alpha,a,l.s.c.}(x)+\beta H_-(x)\big\}\\
&= -\inf_{x\in\ell_1(\R)}\big\{ I(x)+\beta H^\HYL_{\mu+\alpha,l.s.c.}(x)\big\},
\end{aligned}
$$
where the last equality follows from Lemma~\ref{regularisation} below.

\qed
\end{proofsect}

\begin{lemma}\label{regularisation}
$$
\inf_{x\in\ell_1(\R)}\big\{ I(x)+\beta H^\PMF_{\mu,a,l.s.c.}(x)+\beta H_-(x)\big\} =\inf_{x\in\ell_1(\R)}\big\{I(x)+\beta H^\HYL_{\mu,l.s.c.}(x)\big\} 
$$
\end{lemma}
\begin{proofsect}{Proof}
The infimum of any function on an open set is equal to the infimum of its lower semicontinuous regularisation over the same set. Recall that $ \ell_1(\R) $ is open and that $ I_0\equiv \infty $ in $ \ell_1(\R)\setminus\ell_1(\R_+) $.
We thus need to show that
\begin{equation}\label{reg}
\big(H^\PMF_{\mu,a,l.s.c.}+H_-\big)_{l.s.c.}(x)=H^\HYL_{\mu,l.s.c.}(x) \quad \mbox{ for all } x\in\ell_1(\R_+).
\end{equation}
Note that
$$
H^\HYL_\mu(x)=H^\PMF_{\mu,a}(x) +H_-(x) \ge H^\PMF_{\mu,a,l.s.c.}(x)+H_-(x) \ge \big(H^\PMF_{\mu,a,l.s.c.}+H_-\big)_{l.s.c.}(x).
$$
To show \eqref{reg} we use the proof of Lemma~\ref{L:regHYL} and choose the sequence according to \eqref{seqeps}. We thus obtain 
$$
H^\HYL_{\mu,l.s.c.}(x)\ge \big(H^\PMF_{\mu,a,l.s.c.}+H_-\big)_{l.s.c.}(x).
$$
\qed
\end{proofsect}

We finally combine Proposition~\ref{prop-uboundHYL} and Proposition~\ref{prop-lboundHYL}  to finish the proof for Theorem~\ref{THM-PMF}.\qed

\subsection{Proof of Theorem~\ref{THM-densityLDP}}

The prove the first part of Theorem~\ref{THM-densityLDP} requires first the limiting logarithmic moment generating function in Proposition~\ref{logmomentdensity} and  an application of the G\"artner-Ellis theorem. In fact, the reference measure is an independent superposition of Poisson point processes but not identical distributed.

\begin{proofsect}{Proof of Proposition~\ref{logmomentdensity}}  For $ t\in \R $ we get, using the independence of the Poisson point processes with expectation $ {\tt E}^{\ssup{\bc}}_k $, 
$$
\begin{aligned}
{\tt E}^{\ssup{\bc}}\Big[\ex^{t\sum_{k+1}^\infty k\Ncal_k}\Big]=\prod_{k=1}^\infty {\tt E}^{\ssup{\bc}}_k\Big[\ex^{t\sum_{k=1}^\infty k\Ncal_k}\Big]=\prod_{k=1}^\infty \sum_{m=0}^\infty \ex^{tkm}\P(\Ncal_k=m)=\prod_{k=1}^\infty \ex^{\abs{\L_N}q_k^{\ssup{\alpha}}(\ex^{tk}-1)},
\end{aligned}
$$
and thus
$$
\L(t)=\lim_{N\to\infty}\frac{1}{\abs{\L_N}}\log {\tt E}^{\ssup{\bc}}\Big[\ex^{t\sum_{k+1}^\infty k\Ncal_k}\Big]=\sum_{k=1}^\infty q_k^{\ssup{\alpha}}\big(\ex^{tk}-1\big).
$$
\qed
\end{proofsect}

\begin{proofsect}{Proof of Theorem~\ref{THM-densityLDP}}
(a) This is a straightforward application of the G\"artner-Ellis theorem, see \cite{DZ09}. To ensure that $ 0 $ is in the domain of the logarithmic moment generating function we need to have $ \alpha<0 $. Then the rate function is giving as the Legendre-Fenchel transform
$$
\begin{aligned}
J_\alpha(x)&=\sup_{t\in\R}\Big\{t x-\sum_{k=1}^\infty q_k^{\ssup{\alpha}}\big(\ex^{tk}-1\big)\Big\}=\beta\Big(p(\beta,\alpha)+\sup_{t\in\R}\big\{x\frac{t}{\beta}-p(\beta,\alpha+\frac{t}{\beta})\big\} \Big)\\
&=\beta\Big(p(\beta,\alpha)+\sup_{t\in\R}\big\{(\frac{t}{\beta}+\alpha)x-p(\beta,\alpha+\frac{t}{\beta})\big\} -x\alpha\Big)\\
&=\beta\Big(p(\beta,\alpha)+f(\beta,x)-\alpha x\Big),
\end{aligned}
$$
where we used that for $ x\le \varrho_{\rm c} $,
$$
f(\beta,x)=\sup_{\alpha\in\R}\{\alpha x-p(\beta,\alpha)\}.
$$
Clearly, $ J_\alpha(x)=\infty $ when $ x<0 $ as the empirical density only takes positive values. Suppose that $ x>\varrho_{\rm c} $ fo dimensions $ d\ge 3 $ ($ \varrho_{\rm c}=\infty $ for $ d=1,2 $). Then
$$
\sup_{t\in\R}\big\{  (\frac{t}{\beta}+\alpha)\varrho_{\rm c} +(\frac{t}{\beta}+\alpha)(x-\varrho_{\rm c})-p(\beta,\alpha+\frac{t}{\beta})  \big\}\ge f(\beta,\varrho_{\rm c})+\sup_{t\in\R}\big\{(\frac{t}{\beta}+\alpha)(x-\varrho_{\rm c})\big\}=+\infty,
$$
and thus $ J_\alpha(x) $ for $ x\notin [0,\varrho_{\rm c}] $.

\medskip

\noindent (b) This is straightforward application of Varadhan's Lemma in \cite[Theorem~4.3.1]{DZ09} using that $ h(x):=x\mapsto -\mu x+\frac{a}{2}x^2 $ is continuous and the fact that the tail-condition is satisfied
$$
\lim_{M\to\infty}\limsup_{N\to\infty}\frac{1}{\abs{\L_N}}\log {\tt E}\Big[\ex^{-\beta \abs{\L_N}h(\brho_n)}\1\{h(\brho_n)\ge M\}\Big]=-\infty.
$$
\qed
\end{proofsect}

\section{Proofs:  variational analysis and Bose-Einstein condensation (BEC)}

\subsection{Zeroes of the rate functions}\label{sec-zeroLDP}
We collect all proofs for the zeroes o four rate functions.

\begin{proofsect}{Proof of Proposition~\ref{zeroIdeal}}  To find the zeroes of the ideal gas rate function, first let us find the critical points by setting the G{\^a}teaux derivative of the function to zero. That is, we find the set of points $\tilde{x}\in\ell_1\left(\R_+\right)$ such that
\begin{equation*}
    \d I_\mu\left(\tilde{x};y\right)=0\quad\forall y\in\ell_1\left(\R\right).
\end{equation*}
This yields a single equation for each element of the sequence $\tilde{x}$. This set of equations has the unique solution $ \tilde{x}=\xi $ given by
\begin{equation}
\xi_k = q^{\ssup{\mu}}_k, \quad k\in\N.
\end{equation}
Since the rate function $ I_\mu $ is strictly convex where it is finite, this critical point is the unique global minimiser.
\qed
\end{proofsect}

\begin{proofsect}{Proof of Proposition~\ref{zeroCMF}}
	For the existence of a minimiser, recall that $I^{\CMF}_\mu$ is lower semicontinuous and has compact level-sets. Also note that $I_\mu$ is strictly convex where it is finite and $H^{\CMF}$ is convex. Therefore $I^\CMF_\mu$ is strictly convex where it is finite (a non-empty set) and uniqueness of the minimiser follows.
	
	To calculate the minimiser, we search for stationary points. Since $I^\CMF$ is strictly convex where it is finite, if we find a stationary point then it is the global minimiser. By considering the coordinate derivatives, we know that the minimiser must satisfy all the following equations:
	\begin{equation*}
	\log \frac{x_k}{q^{\ssup{\mu}}_k} + a\beta \sum^\infty_{k=1}x_k  = 0, \quad k\in\N.
	\end{equation*}
	To make this more manageable, we introduce the dummy variable $\Gamma \in \R_+$ and corresponding equation $\Gamma = \sum^\infty_{k=1}x_k$. Our problem is then to solve
	\begin{align}
	\label{eqn:CMFstationary1}
	\log \frac{x_k}{q^{\ssup{\mu}}_k} + a\beta \Gamma &= 0, \quad k\in\N,\\
	\label{eqn:CMFstationary2}
	\Gamma - \sum^\infty_{k=1}x_k &= 0.
	\end{align}
	Given $\Gamma$, \eqref{eqn:CMFstationary1} is uniquely solved by $ x_k = q^{\ssup{\mu}}_k \exp\left( - a \beta \Gamma\right), k\in\N $, and therefore \eqref{eqn:CMFstationary2} becomes
	\begin{equation*}
	\Gamma = \exp\left( -a\beta \Gamma\right)\bar{q}^{\ssup{\mu}}.
	\end{equation*}
	This has the unique solution $ \Gamma = \frac{1}{a\beta}W_0\left(a\beta \bar{q}^{\ssup{\mu}}\right) = \frac{1}{a\beta}W_0\left(K\right) $,
	and so \eqref{eqn:CMFstationary1} and \eqref{eqn:CMFstationary2}  are uniquely jointly solved by $ x=\xi $ given by
	\begin{equation*}
	\xi_k = q^{\ssup{\mu}}_k \exp\left(-W_0\left(K\right)\right) = \frac{W_0\left(K\right)}{K}q^{\ssup{\mu}}_k, \quad k\in\N.
	\end{equation*}
	\qed
\end{proofsect}

\begin{proofsect}{Proof of Proposition~\ref{zeroPMF}} Without loss of generality we put $ \alpha\equiv 0 $, and write $ q_k^{\ssup{0}}=q_k, k\in\N $ and $ I_0=I $. To obtain the unique zero of the rate function we shall find the unique minimiser of the un-normalised rate function $ F(x):=I(x)+\beta H^\PMF_{\mu,l.s.c.}(x) $. For the existence of a minimiser, recall that $F$ is lower semicontinuous and has compact level-sets. Also note that $ I $ is strictly convex where it is finite, and $ H^\PMF_{\mu,l.s.c.} $ is also convex in the linear function $ D(x) $.  Therefore $F$ is strictly convex where it is finite ( a non-empty set) and uniqueness of the minimiser follows. To calculate the minimiser, we search for stationary points. Since $F$ is strictly convex where it is finite, if we find a stationary point then it is the global minimiser.  By considering again as in the proof of Proposition~\ref{zeroIdeal} the coordinate derivatives, we know that the minimiser must satisfy all the following equations
$$
\log\frac{x_k}{q_k}+\beta k\big(aD(x)-\mu\big)_+=0, \quad k\in\N.
$$
To make this more manageable, we introduce the dummy variable $ \delta\in\R_+ $ and corresponding equation $ \delta=D(x) $.
\begin{align}
\log\frac{x_k}{q_k}+\beta k\big(a \delta-\mu\big)_+&=0, \quad k\in\N,\label{1stcond}\\
\delta-D(x)&=0.\label{seccond}
\end{align}
Given the value $ \delta $, \eqref{1stcond} is uniquely solved by $x_k=q_k\exp\big(\beta k\big(\mu-a\delta\big)_-\big), k\in\N,
$ and therefore \eqref{seccond} becomes
\begin{equation}\label{solutionDelta}
\delta=\sum_{k=1}^\infty k q_k \exp\big(\beta k\big(\mu-a\delta\big)_-\big).
\end{equation}
Denote the right hand side in \eqref{solutionDelta} by $h(\delta) $,  and note that $ h(\delta)\to 0 $ as $ \delta\to\infty $. Furthermore,
$$
\lim_{\delta\to 0} h(\delta)=\begin{cases} \sum_{k\in\N} kq_k^{\ssup{\mu}} \in (0,\infty) & ,\mu<0,\\ \infty &,d=1,2 \mbox{ and } \mu\ge 0,\\ \varrho_{\rm c}\in (0,\infty) &,d\ge 3 \mbox{ and } \mu\ge 0,\end{cases}
$$
see Figure~\ref{fig:PMFconsistency}. In all cases there exists a unique solution  which we denote $ \delta^* $.

\begin{figure}
	\centering
	\begin{subfigure}[b]{0.45\textwidth}
		\begin{tikzpicture}[scale = 1.25]
		\draw[->] (0,0) node[below left]{$0$} -- (0,3) node[left]{$h\left(\delta\right)$};
		\draw[->] (0,0) -- (4,0) node[right]{$\delta$};
		\draw[dashed] (1,0) node[below]{$\frac{\mu+\alpha}{a}$} -- (1,3);
		\draw[thick] (1,3) to [out=270,in=175] (4,0.5);
		\end{tikzpicture}
		\caption{$d=1,2$}
		\label{fig:PMFconsistency12}
	\end{subfigure}
	\begin{subfigure}[b]{0.45\textwidth}
		\begin{tikzpicture}[scale =1.25]
		\draw[->] (0,0) node[below left]{$0$} -- (0,3)node[left]{$h\left(\delta\right)$};
		\draw[->] (0,0) -- (4,0) node[right]{$\delta$};
		\draw[dashed] (1,0) node[below]{$\frac{\mu+\alpha}{a}$} -- (1,3);
		\draw[thick] (0,2) node[left]{$\varrho_\mathrm{c}$} -- (1,2) to [out=270,in=175] (4,0.5);
		\end{tikzpicture}
		\caption{$d\geq 3$}
		\label{fig:PMFconsistency>=3}
	\end{subfigure}
	\caption{Sketch of $h\left(\delta\right)$. This shows $\mu+\alpha>0$, but the sketch translates with $\mu$.}
	\label{fig:PMFconsistency}
\end{figure}
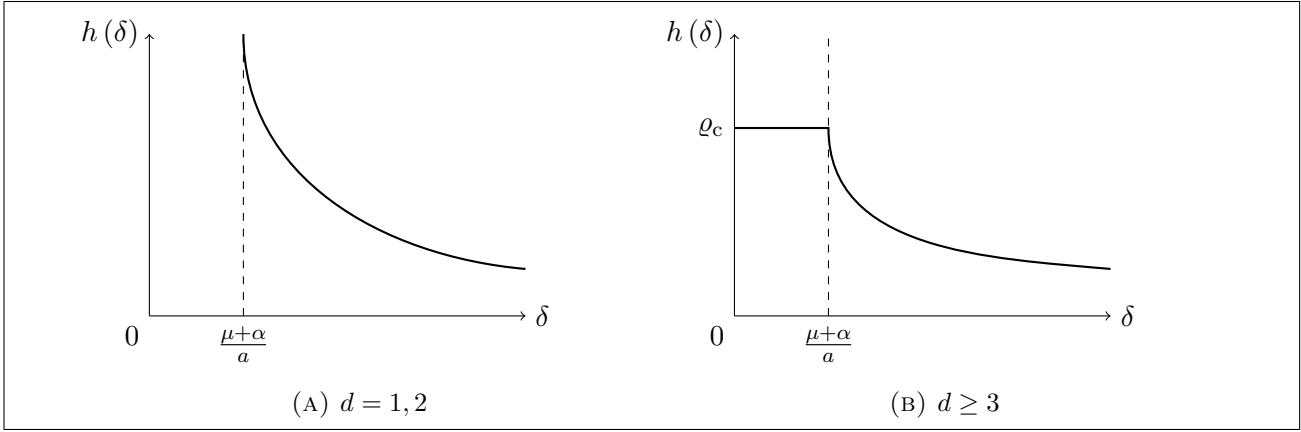

\qed
\end{proofsect}

\begin{proofsect}{Proof of Proposition~\ref{zeroHYL}}
	Without loss of generality we put $ \alpha\equiv 0 $, and write $ q_k^{\ssup{0}}=q_k, k\in\N $,  and write $ I^\HYL_{\mu,0}=I^\HYL_\mu $.	For the existence of such a minimiser $\xi$, recall that $I^\HYL_{\mu,\alpha}$ is lower semicontinuous and has compact level-sets.
	
	Suppose that the partial derivative $\frac{\partial I^\HYL_\mu}{\partial x_k}$ is defined and non-zero at $x\in \mathrm{Int}\left\{\ell_1\left(\R_+\right)\right\}$. Then $I^\HYL_\mu$ does not achieve its infimum at $x$. Since the boundary $\partial\ell_1\left(\R_+\right) = \left\{x\in \ell_1\left(\R_+\right) : \exists\, k \text{ s.t. }x_k=0\right\}$, and $\frac{\partial I^\HYL_\mu}{\partial x_k} = -\infty$ here, the infimum is not achieved here.
	
	For $x\in\mathrm{Int}\left\{\ell_1\left(\R_+\right)\right\}$ we have
	\begin{equation*}
	\frac{\partial I^\HYL_\mu}{\partial x_k} (x)= \log \frac{x_k}{q_k} - b\beta k^2 x_k - \beta k \left(\mu - a D\left(x\right)\right)
	\begin{Bmatrix}
	1&,aD\left(x\right) \geq \mu\\
	-\frac{b}{a-b}&,aD\left(x\right) \leq \mu
	\end{Bmatrix}, \qquad k\in\N,
	\end{equation*}
	which is defined everywhere in $\mathrm{Int}\left\{\ell_1\left(\R_+\right)\right\}$. Hence a solution $\xi$ must solve $\frac{\partial I^\HYL_\mu}{\partial x_k} (\xi)= 0$ for all $k\in\N $. To make this more manageable, we introduce the dummy variable $\delta\in\R_+$ and corresponding equation $\delta = D\left(x\right)$.  Our problem is then to solve
	\begin{align}
	\label{eqn:HYLstationary1}
	\log \frac{x_k}{q_k} - b\beta k^2 x_k - \beta k \left(\mu - a \delta\right)
	\begin{Bmatrix}
	1&,a\delta \geq \mu\\
	-\frac{b}{a-b}&,a\delta \leq \mu
	\end{Bmatrix} &= 0, \qquad k\in\N,\\[1.5ex]
	\label{eqn:HYLstationary2}
	\delta - D\left(x\right) &= 0.
	\end{align}
	Unfortunately - unlike in the corresponding PMF case - even when we are given $\delta$ we are not guaranteed to have a solution for \eqref{eqn:HYLstationary1}, or that such a solution would be unique. If we fix $\delta$, then the $k^{th}$ equation of \eqref{eqn:HYLstationary1} either has no solution or is solved by 
	\begin{equation*}
	x_k = -\frac{1}{b\beta k^2}W_{\chi_k}\Big(-b\beta k^2 q_k \exp\Big[\beta k \Big(\mu - a \delta\Big)
	\begin{Bmatrix}
	1&,a\delta \geq \mu\\
	-\frac{b}{a-b}&,a\delta \leq \mu
	\end{Bmatrix}
	\Big]\Big)
	\end{equation*}
	for all $\chi_k\in\left\{0,-1\right\}$, where $W_0$ and $W_{-1}$ are the two real branches of the Lambert W function. The `no solution' case corresponds precisely to $W_0$ and $W_{-1}$ not being defined for this input.
	
	Substituting these $x_k$ back into \eqref{eqn:HYLstationary2} gives the condition \eqref{eqn:HYLconsistency} as required.
\qed
\end{proofsect}

\begin{proofsect}{Proof of Proposition~\ref{HYLboringcase}}
	The essence of this proof is to find a set of parameters for which Proposition~\ref{zeroHYL} produces only one candidate zero. First note that where they are finite, $g^\chi$ are continuous. Also, for $\chi\ne0$ we have $g^\chi\left(\delta\right) \sim C_\chi \delta$ as $\delta\to+\infty$ where $C\chi \geq \frac{a}{b} > 1$. Because $g^\chi$ has a translational symmetry with $\mu + \alpha$, this means that for sufficiently negative $\mu+\alpha$ these branches will have no solution. In contrast, $g^0$ is continuous and decreasing for $\delta > \frac{\mu+\alpha}{a}$. This means that as we decrease $\mu+\alpha$ we are eventually guaranteed to have a unique solution for $\delta$.
	\qed 
\end{proofsect}

\subsection{Proofs: pressure representations, thermodynamic behaviour and  BEC}\label{sec-pressureBEC}
Theorem~\ref{THM-pressure} allows to analyse the thermodynamic limit.  Note that the partial derivative of the pressure with respect to the chemical potential is the expected physical particle density. This particle density depends on the thermodynamic systems, and its value can equal the density of the minimisers or it can exceed the density of the minimisers. In the latter case one experiences an excess density which is the particle mass condensing in the so-called BEC state.
The following result will be useful in calculating the derivatives of these pressures.

\begin{lemma}
	\label{diffPressureTech}
	Let $I\subset\R$ be an open interval, $F:\ell_1 \times I \to \R$, and $\xi\in C^1\left(I;\ell_1\right)$. Also define
$$
	\mathcal{G}:I\to\R;  s\mapsto F\left(\xi\left(s\right),s\right).
	$$
	Then if $F\left(x,s\right)$ is G{\^a}teaux differentiable in its first argument at $\xi\left(s\right)$ with $\left.\frac{\partial F}{\partial x_k}\right|_{\xi}=0$ $\forall k\in\N$, then
	\begin{equation*}
	\frac{\d \mathcal{G}}{\d s} = \left.\frac{\partial F}{\partial s}\right|_{\xi}.
	\end{equation*}
\end{lemma}

\begin{proofsect}{Proof}
An application of the chain rule gives
\begin{equation*}
\frac{\d \mathcal{G}}{\d s} = \left.\frac{\partial F}{\partial s}\right|_{\xi} + \sum^\infty_{j=1}\frac{\d \xi_k}{\d s}\left.\frac{\partial F}{\partial x_k}\right|_{x=\xi}.
\end{equation*}
Since the partial derivatives of $F$ with respect to $x_k$ vanish at $\xi$, we only keep the first term.
\qed
\end{proofsect}
\smallskip

\noindent \textbf{Ideal Bose gas.}

\medskip

\begin{proofsect}{Proof of Proposition~\ref{P:ideal1}}
(1)  Clearly, $ p(\beta,0)=\frac{1}{\beta(4\pi\beta)^{\frac{d}{2}}}\sum_{k=1}^\infty\frac{1}{k^{1+d/2}}<\infty $ for all $ d\ge 1 $. Convexity follows from properties of the Bose functions, $g(1+\frac{d}{2},-\beta\mu) $, see \eqref{defBosefunctionapa} in Appendix~\ref{app-Bose}. (2) This follows from the Bose functions, $g(n,x)$, being differentiable for $x>0$, and
\begin{equation*}
		\frac{\d}{\d x}g(n,x) = -g(n-1,x), \qquad\forall x>0.
	\end{equation*}
	Then the first derivative follows from directly differentiating the representation \eqref{p}.
	\qed
\end{proofsect}

\begin{proofsect}{Proof of Proposition~\ref{densityIdeal}}
(1) This follows by direct computation. The exponential term ensures that the derivative of the finite-volume pressure is increasing in $ \mu $. As long as the box $ \L_N $ has finite volume one can give the average particle density any pre-assigned value by choosing a chemical potential.\\
(2) The limit in \eqref{critical} is obtained by direct calculation in conjunction with basic properties of the Bose function summarised in Appendix~\ref{app-Bose}. The convergence of the unique root is ensured as long as the expected particle density stays below the critical density which is finite only in dimensions $ d\ge 3 $.
\qed
\end{proofsect}

\begin{proofsect}{Proof of Proposition~\ref{IdealFreeEnergy}}
	Since $p\left(\beta,\mu\right)=+\infty$ for $\mu>0$, we only need to search $\mu\leq0$. On the interior of this region $p$ is differentiable, and we look for stationary points. If $\varrho\geq\varrho_\mathrm{c}$, then there are no stationary points for $\mu<0$ and $\mu\varrho-p\left(\beta,\mu\right)$ is increasing in $\mu$. Hence the supremum is achieved at $\mu = 0$. If $\varrho<\varrho_\mathrm{c}$, then there is a unique stationary point. This is also a local maximum and is given at $\mu = \alpha$ as required. This has the required limit as $\varrho\uparrow\varrho_\mathrm{c}$ implying the  continuity for $f$.    \qed
\end{proofsect}

\begin{proofsect}{Proof of Theorem~\ref{thm:IDEALcondensate}}
		Let us begin with the $\mu<0$ case, because a similar approach will be important for our approaches to the PMF and HYL cases. For $s\leq-\mu$, fixed $N$, and fixed $K$, define
		\begin{equation*}
		g^\ssup{K}_N\left(s\right) := \frac{1}{\beta\left|\Lambda_N\right|}\log\E_{\nu_{N,\mu}}\left[\exp\left(\left|\Lambda_N\right|s\beta\left(D-D_K\right)\right)\right].
		\end{equation*}
		Then
		\begin{equation*}
		\left.\frac{\d g^\ssup{K}_N}{\d s}\right|_{s=0} = \mathbb{E}_{\nu_{N,\mu}}\left[D-D_K\right] \quad \text{and}\quad \Delta\left(\beta,\mu\right) = \lim_{K\to\infty}\lim_{N\to\infty}\left.\frac{\d g^\ssup{K}_N}{\d s}\right|_{s=0}.
		\end{equation*}
		Since $g^\ssup{K}_N$ are all convex in $s$, we will use Griffith's Lemma to get the point-wise limit of the derivative from the derivative of the point-wise limit:
		\begin{equation*}
		\Delta\left(\beta,\mu\right) = \lim_{K\to\infty}\frac{\d}{\d s}\Big(\lim_{N\to\infty}g^\ssup{K}_N\left(s\right)\Big)\Big|_{s=0}.
		\end{equation*}
		
		To calculate the point-wise limit of $g^\ssup{K}_N$, we first rewrite $g^\ssup{K}_N$ as
		\begin{equation*}
		g^\ssup{K}_N\left(s\right) = \frac{1}{\beta\left|\Lambda_N\right|}\log\E_{\nu_{N,\mu+s}}\left[\exp\left(-\left|\Lambda_N\right|s\beta D_K\right)\right] + \frac{1}{\beta\left|\Lambda_N\right|}\log \frac{Z_{N}\left(\beta,\mu+s\right)}{Z_{N}\left(\beta,\mu\right)}.
		\end{equation*}
		Then we want to use Varadhan's Lemma with our LDP for the ideal Bose gas measure and the tilt $\phi = -s\beta D_K$. This $\phi$ is continuous, but we need to pay attention to the boundedness conditions. We will show
		\begin{equation}\label{IDEALVaradhanConstraint}
		\lim_{M\to\infty}\limsup_{N\to\infty}\frac{1}{\left|\L_N\right|}\E_{\nu_{N,\mu+s}}\left[\ex^{\left|\L_N\right|\phi}\mathds{1}\left\{\phi\geq M\right\}\right] = -\infty.
		\end{equation}
		For $0\leq s \leq -\mu$, we have $\phi\leq 0$ almost surely, so \eqref{IDEALVaradhanConstraint} holds trivially. For $s<0$ we have to work a little harder. Since $\phi$ is continuous, the set $\left\{\phi=m\right\}$ is closed (and measurable). Hence our LDP for the ideal Bose gas model gives us a bound on the probability of this set:
		\begin{equation*}
		\limsup_{N\rightarrow\infty}\frac{1}{\left|\Lambda_N\right|}\log \nu_{N,\mu+s}\left(\phi=m\right) \leq -\inf_{\phi=m}I_{\mu+s} \leq \bar{q}^\ssup{0} + \beta\left(\mu+s\right)\Big(\frac{m}{\left|s\right|\beta}\Big).
		\end{equation*}
		This means that for sufficiently large $N$ there exists a $m$ and $N$ independent constant $C>\bar{q}^\ssup{0}$, such that
		\begin{equation*}
		\ex^{m\left|\Lambda_N\right|}\nu_{N,\mu+s}\left(\phi=m\right) \leq \exp\Big(\abs{\Lambda_N}\big[C + \frac{\mu}{\abs{s}}m\big]\Big).
		\end{equation*}
		Since $\mu<0$, we have sufficiently fast decay in $m$ to prove that \eqref{IDEALVaradhanConstraint} holds even for $s<0$, and Varadhan gives us
		\begin{equation*}
		\lim_{N\rightarrow\infty}g^\ssup{K}_N\left(s\right) = -\inf_{\ell_1}\big\{\frac{1}{\beta}I_{\mu+s} + sD_J\big\} + p\left(\beta,\mu+s\right) -p\left(\beta,\mu\right),\qquad \forall s\leq-\mu.
		\end{equation*}
		
		In the style of Lemma \ref{zeroIdeal}, we can find that this infimum is achieved at $\xi\left(s\right)\in\ell_1\left(\R_+\right)$, where
		\begin{equation*}
		\xi_k=\begin{cases}
		q_k^{\ssup{\mu}} &,k\leq K,\\
		q_k^{\ssup{\mu+s}} &, k>K.
		\end{cases}
		\end{equation*}
		Hence
		\begin{equation*}
			\lim_{N\rightarrow\infty}g^\ssup{K}_N\left(s\right) = \frac{1}{\beta}\sum^\infty_{j=K+1}\big(q^\ssup{\mu+s}_j-q^\ssup{\mu}_j\big),\qquad \frac{\d}{\d s}\Big(\lim_{N\to\infty}g^\ssup{K}_N\left(s\right)\Big)\Big|_{s=0} = \sum^\infty_{j=K+1}jq^\ssup{\mu}_j.
		\end{equation*}
		Finally the sum vanishes as $K\to\infty$.
		
		For $\mu=0$ with $d=1,2$, we take a more direct approach. It is clear from our construction of the ideal Bose gas model that the point-wise limit $\lim_{N\to\infty}\E_{\nu_{N,\mu}}\left[jx_j\right] = jq^\ssup{\mu}_j$. Then for all $M\in\N$,
		\begin{align*}
		\liminf_{N\to\infty}\E_{\nu_{N,\alpha}}\Big[\sum^\infty_{j=K+1}jx_j\Big] \geq \lim_{N\to\infty}\E_{\nu_{N,\mu}}\Big[\sum^M_{j=K+1}jx_j\Big] = \sum^M_{j=K+1}jq^\ssup{\mu}_j.
		\end{align*}
		Since this lower bound diverges as $M\to\infty$ if $\mu=0$ and $d=1,2$, we have our result for this case.
		
		For $\mu=0$ and $d\geq3$, the behaviour changes depending upon the boundary condition. Nevertheless, by applying direct methods similar to the $\mu>0$ case we get the required results.
		\qed
	\end{proofsect}
\medskip

\noindent\textbf{The CMF model}

\medskip

\begin{proofsect}{Proof of Proposition~\ref{P:densityCMF}}
(1)  The continuity of $p^\CMF$ for $\mu\leq 0$ follows from \eqref{p-cmf} and the continuity of $K=a\beta \bar{q}^\ssup{\mu}$ and $W_0$. Convexity follows from considering the derivatives of $W_0\left(K\right)$ with respect to $\mu$ for $\mu<0$. Appendix~\ref{app-Lambert} is useful for this calculation.\\
(2) Smoothness follows from $W_0$ and the Bose functions being differentiable on the appropriate regions. The form of the first derivative can be found by either directly differentiating \eqref{p-cmf}, or by using Lemma~\ref{diffPressureTech} with the zero found in Proposition~\ref{zeroCMF}.\\
(3) We obtain \eqref{criticalCMF} from \eqref{critical} and the continuity of $ W_0 $ and $ \overline{q}^{\ssup{\mu}} $. 
\qed
\end{proofsect}

\begin{proofsect}{Proof of Proposition~\ref{P:LFT-CMF}}
	This is proven in the same way as Proposition~\ref{IdealFreeEnergy}.
	\qed
\end{proofsect}

\begin{proofsect}{Proof of Theorem~\ref{thm:CMFcondensate}}
	This proof begins identically to Theorem~\ref{thm:IDEALcondensate}. For $s\leq-\mu$, fixed $N$, and fixed $K$, define
	\begin{equation*}
	g^\ssup{K}_N\left(s\right) := \frac{1}{\beta\left|\Lambda_N\right|}\log\E_{\nu_{N,\mu}^\CMF}\left[\exp\left(\left|\Lambda_N\right|s\beta\left(D-D_K\right)\right)\right].
	\end{equation*}
	Then by rearranging terms and applying Varadhan, we find that for $s\leq -\mu$,
	\begin{equation*}
	\lim_{N\rightarrow\infty}g^\ssup{K}_N\left(s\right) = -\inf_{\ell_1}\big\{\frac{1}{\beta}I^\CMF_{\mu+s} + sD_J\big\} + p^\CMF\left(\beta,\mu+s\right) -p^\CMF\left(\beta,\mu\right),\qquad \forall s\leq-\mu.
	\end{equation*}
	
	In the style of Lemma \ref{zeroCMF}, we can find that this infimum is achieved at $\xi\left(s\right)\in\ell_1\left(\R_+\right)$, where
	\begin{equation*}
	\xi_k=\begin{cases}
	\frac{W_0\left(a\beta \bar{q}^\ssup{\mu}\right)}{a\beta \bar{q}^\ssup{\mu}}q_k^{\ssup{\mu}} &,  k\leq K,\\
	\frac{W_0\left(a\beta \bar{q}^\ssup{\mu+s}\right)}{a\beta \bar{q}^\ssup{\mu+s}}q_k^{\ssup{\mu+s}} &,  k>K.
	\end{cases}
	\end{equation*}
	Substituting this into the infimum, and then taking the derivative gives us
	\begin{equation*}
	\frac{\d}{\d s}\Big(\lim_{N\to\infty}g^\ssup{K}_N\left(s\right)\Big)\Big|_{s=0} = \frac{W_0\left(a\beta \bar{q}^\ssup{\mu}\right)}{a\beta \bar{q}^\ssup{\mu}}\sum^\infty_{j=K+1}jq^\ssup{\mu}_j.
	\end{equation*}
	Finally the sum vanishes as $K\to\infty$, and Griffith's Lemma gives us the result.
	\qed
\end{proofsect}

\bigskip

\noindent\textbf{The PMF model}

\medskip

\begin{proofsect}{Proof of Proposition~\ref{P:densPMF}}
	This follows from either directly differentiating \eqref{p-pmf}, or by using Lemma~\ref{diffPressureTech} with the zero found in Proposition~\ref{zeroPMF}. Convexity also follows from Proposition~\ref{zeroPMF}, noting that $D\left(\xi^\PMF\right)$ is continuous and increasing in $\mu$.
	\qed
\end{proofsect}

\begin{proofsect}{Proof Proposition~\ref{P:fPMF}}
The proof is obtained directly from our analysis above, or by application of \eqref{PMFpressure} in Remark~\ref{RemdensityLDP}.
\qed
\end{proofsect}

\begin{proofsect}{Proof of Theorem~\ref{thm:PMFcondensate}}
	Let us set $\alpha=0$ and omit it from the notation. Our proof begins similarly to Theorems~\ref{thm:IDEALcondensate} and \ref{thm:CMFcondensate}. For $\mu,s\in\R$, fixed $N$, and fixed $K$, define
	\begin{equation*}
		g^\ssup{K}_N\left(s\right) = \frac{1}{\beta\left|\Lambda_N\right|}\log\E_{\nu_{N,\mu}^\PMF}\left[\exp\left(\left|\Lambda_N\right|s\beta\left(D-D_K\right)\right)\right].
	\end{equation*}
	We once again rearrange terms to get
	\begin{equation*}
	g^\ssup{K}_N\left(s\right) = \frac{1}{\beta\left|\Lambda_N\right|}\log\E_{\nu_{N,\mu+s}^\PMF}\left[\exp\left(-\left|\Lambda_N\right|s\beta D_K\right)\right] + \frac{1}{\beta\left|\Lambda_N\right|}\log \frac{Z^\PMF_{N}\left(\beta,\mu+s\right)}{Z^\PMF_{N}\left(\beta,\mu\right)}.
	\end{equation*}
	Then we want to use Varadhan's Lemma with our LDP for the PMF measure and the tilt $\phi = -s\beta D_K$. This $\phi$ is continuous, but we need to pay attention to the boundedness conditions. We will show
	\begin{equation}\label{VaradhanConstraint}
		\lim_{M\to\infty}\limsup_{N\to\infty}\frac{1}{\left|\L_N\right|}\E_{\nu_{N,\mu+s}^\PMF}\left[\ex^{\left|\L_N\right|\phi}\mathds{1}\left\{\phi\geq M\right\}\right] = -\infty.
	\end{equation}
	For $s\geq0$, we have $\phi\leq 0$ almost surely, so \eqref{VaradhanConstraint} holds trivially. For $s<0$ we have to work a little harder. Our LDP for the PMF model gives us a bound on the probability of this set:
	\begin{align*}
		\limsup_{N\rightarrow\infty}\frac{1}{\left|\Lambda_N\right|}\log \nu^\PMF_{N,\mu+s}\left(\phi=m\right) &\leq -\inf_{\phi=m}I^\PMF_{\mu+s}\\
		\frac{m}{\left|s\right|\beta}\geq \frac{\mu+s}{a}\implies \qquad& \leq \inf_{\ell_1}\left\{I + \beta H^\PMF_{\mu+s,l.s.c.}\right\} + \beta(\mu+s)\left(\frac{m}{\abs{s}\beta}\right) - \frac{a\beta}{2}\left(\frac{m}{\abs{s}\beta}\right)^2.
	\end{align*}
	This means that given $m \geq \left|s\right|\beta \frac{\mu+s}{a}$, then for sufficiently large $N$ there exists a $m$ and $N$ independent constant $C>\inf_{\ell_1}\left\{I_0 + \beta H^\PMF_{\mu+s,l.s.c.}\right\}$, such that
	\begin{equation*}
		\ex^{m\abs{\Lambda_N}}\nu^\PMF_{N,\mu+s}\left(\phi=m\right) \leq \exp\Big(\abs{\Lambda_N}\big[C + \frac{\mu}{\abs{s}m} - \frac{a}{2\beta\abs{s}^2}m^2\big]\Big).
	\end{equation*}
	The very fast decay with $m$ proves that \eqref{VaradhanConstraint} holds even for $s<0$, and Varadhan gives us
	\begin{equation*}
		\lim_{N\rightarrow\infty}g^\ssup{K}_N(s) = -\inf_{\ell_1}\big\{\frac{1}{\beta}I^\PMF_{\mu+s} + sD_J\big\} + p^\PMF\left(\beta,\mu+s\right) -p^\PMF\left(\beta,\mu\right),\qquad \forall s\in\R.
	\end{equation*}
	
	In the style of Lemma \ref{zeroPMF}, we can find that this infimum is achieved at $\xi\left(s\right)\in\ell_1\left(\R_+\right)$, where
	\begin{equation*}
		\xi_k=q_k^{\ssup{0}} \exp\left(\beta k\left[\left(\mu + s-a\delta^*\right)_--s\mathds{1}\left\{k\leq J\right\}\right]\right),\quad k\in\N,
	\end{equation*}
	and $ \delta^*\left(s\right)$ is given implicitly as the unique solution to the equation
	\begin{equation}
		\delta^*=\begin{cases}
			\sum^K_{j=1}jq_j^{\ssup{0}} \exp\left(\beta j\left(\mu-a\delta^*\right)\right) + \sum^\infty_{j=K+1}jq_j^{\ssup{0}} \exp\left(\beta j\left(\mu + s-a\delta^*\right)\right) &, \delta^* \geq \frac{\mu+s}{a},\\
			\sum^K_{j=1}jq_j^{\ssup{0}} \exp\left(-s\beta j\right) + \sum^\infty_{j=K+1}jq_j^{\ssup{0}} &, \delta^* \leq \frac{\mu+s}{a}.
		\end{cases}
	\end{equation}
	If we denote
	\begin{equation}
		\delta^*_K=\begin{cases}
			\sum^K_{j=1}jq_j^{\ssup{0}} \exp\left(\beta j\left(\mu-a\delta^*\right)\right) &, \delta^* \geq \frac{\mu+s}{a},\\
			\sum^K_{j=1}jq_j^{\ssup{0}} \exp\left(-s\beta j\right)  &, \delta^* \leq \frac{\mu+s}{a},
		\end{cases}
	\end{equation}
	then Lemma \ref{diffPressureTech} tells us that
	\begin{align*}
		\frac{\d}{\d s}\Big(\lim_{N\rightarrow\infty}g^{\ssup{K}}_N\left(s\right)\Big) &= \left(\delta^* - \delta^*_K\right)\left(s\right) + \Big(\frac{\mu+s}{a}-\delta^*\left(s\right)\Big)_+\\
		\frac{\d}{\d s}\Big(\lim_{N\rightarrow\infty}g^\ssup{K}_N\left(s\right)\Big)\Big|_{s=0} &= \begin{cases}
			\left(\delta^* - \delta^*_K\right)\left(0\right) &, \mu\leq a\varrho_{\mathrm{c}},\\
			\frac{\mu}{a} - \delta^*_K\left(0\right) &, \mu \geq a\varrho_{\mathrm{c}},
		\end{cases}\\
		\lim_{K\to\infty}\frac{\d}{\d s}\Big(\lim_{N\rightarrow\infty}g^\ssup{K}_N\left(s\right)\Big)\Big|_{s=0} &= \left(\frac{\mu}{a}-\varrho_{\mathrm{c}}\right)_+.
	\end{align*}
	\qed
\end{proofsect}

\bigskip

\noindent \textbf{The HYL model}

\medskip

\begin{proofsect}{Proof of Proposition~\ref{P:HYLsub}}
	This is proven by using Lemma~\ref{diffPressureTech} with the zero found in Proposition~\ref{HYLboringcase}. Convexity also follows from Proposition~\ref{HYLboringcase}, noting that $D\left(\xi^\HYL\right)$ is continuous and increasing in $\mu$.
	\qed
\end{proofsect}

\begin{proofsect}{Proof of Proposition~\ref{HYLpressure}}
	Let us set $\alpha\equiv0$. Recall that $g^\chi\left(\delta\right)$ denotes the right hand side of \eqref{eqn:HYLconsistency} for a given $\chi$. Now $\beta\geq\beta^*$ ensures that $g^\chi$ can be defined for all $\delta\in\R$. For this proof we will only require $\mu\leq \bar{\mu}:= ag^0\left(\frac{\mu}{a}\right) = \frac{-1}{b\beta}\sum^\infty_{j=1}\frac{1}{j}W_{0}\left(-b\beta j^2q^{\ssup{0}}_j\right)$.
	
	Note that for $d\geq 3$, the arguments for the Lambert W functions are strictly increasing in the summation index $k$, approaching $0$. This means that since the difference $W_0\left(x\right) - W_{-1}\left(x\right)\geq 0$ is strictly increasing in $x$ and equals $0$ if and only if $x=-\ex^{-1}$, we only need to consider finitely many $\chi$ for a given $\mu$ (all of which are eventually $0$). Now since any non-convexity in $g^\chi$ can only arrive via the finitely many $\chi_k=-1$ terms, solutions to $\delta = g^\chi\left(\delta\right)$ are locally finite in $\R$. To complement this, note that $\lim_{\delta\to+\infty}g^0\left(\delta\right)=0$ whilst for $\chi\ne 0$ we have $g^\chi\left(\delta\right) \gg \delta$. Hence we only need to consider a finite range of $\delta$, and therefore for a given $\mu$ there are only finitely many solutions for $\delta$.
	
	Because $g^\chi$ is continuous for each $\chi$, we can collect solutions uniquely and maximally into continuous paths $\xi^j\left(\mu\right)$ defined on closed (possibly infinite) intervals $I^j$ with non-empty interior. We allow families to overlap at endpoints of these intervals. Because we are only considering $\mu\leq \bar{\mu}$ and there are only finitely many solutions for each $\mu$, we will only have finitely many families being relevant to our discussion.
	
	For each of these families, we will denote 
	\begin{equation*}
	D^j\left(\mu\right) := D\left(\xi^j\left(\mu\right)\right), \qquad
	P^j\left(\mu\right) := p\left(\beta,0\right)-\frac{1}{\beta}\left(I + \beta H^\HYL_{\mu,l.s.c.}\right)\left(\xi^j\left(\mu\right)\right),
	\end{equation*}
	defined on the interval $I^j$. From Proposition~\ref{zeroHYL} and Theorem~\ref{THM-pressure} we know that
	\begin{equation*}
	p^\HYL\left(\beta,\mu\right) = p\left(\beta,0\right) -\frac{1}{\beta}\inf_{\ell_1}\left\{I+\beta H^\HYL_{\mu,l.s.c.}\right\} = \max_{j}P^j\left(\mu\right).
	\end{equation*}
	Therefore for each $\mu$, there exists a $J$ such that $p^\HYL\left(\beta,\mu\right)=P^J\left(\mu\right)$.
	
	From the continuity of $g^\chi$ we know that all $D^j$ are continuous on their $I^j$. Then Lemma~\ref{diffPressureTech} tells us that each $P^j$ is differentiable on $\mathrm{Int}I^j$, with derivative
	\begin{equation*}
	\frac{\d P^j}{\d \mu} = D^j + \frac{\left(\mu - a D^j\right)_+}{a-b}.
	\end{equation*}
	Continuity of this derivative follows from the continuity of $D^j$.
	
	Let us now consider the $\chi=0$ solutions. Since $g^0$ is convex when restricted to $\delta\leq \frac{\mu}{a}$, $\frac{\d g^\chi}{\d \delta}\to+\infty$ as $\delta\uparrow\frac{\mu}{a}$, and $g^0$ is decreasing for $\delta \geq \frac{\mu}{a}$, there exists $\underline{\mu}<\bar{\mu}$ such that this branch has multiple solutions if and only if $\mu\in\left[\underline{\mu},\bar{\mu}\right]$. Let us label these $\xi^0$, $\xi^1$, and $\xi^2$ such that $D^0\geq D^1\geq D^2$. Note that $I^0=\left(-\infty,\bar{\mu}\right]$, $I^1=\left[\underline{\mu},\bar{\mu}\right]$, and $I^2=\left[\underline{\mu},+\infty\right)$. For a visualisation of these solutions, see Figure~\ref{fig:HYL3sols}.

	Since $\xi^0\left(\bar{\mu}\right) = \xi^1\left(\bar{\mu}\right)$ and $\xi^1\left(\underline{\mu}\right) = \xi^2\left(\underline{\mu}\right)$, we have $P^0\left(\bar{\mu}\right) = P^1\left(\bar{\mu}\right)$ and $P^1\left(\underline{\mu}\right) = P^2\left(\underline{\mu}\right)$. Because $D^0\geq \frac{\mu}{a}$ and $D^{1,2}\leq \frac{\mu}{a}$, we have
	\begin{equation*}
		\frac{\d P^0}{\d \mu} = D^0, \qquad \frac{\d P^{1}}{\d \mu} = \frac{b}{a-b}\left(\frac{\mu}{b} - D^{1}\right) < \frac{b}{a-b}\left(\frac{\mu}{b} - D^{2}\right) = \frac{\d P^{2}}{\d \mu},
	\end{equation*}
	on $\left(\underline{\mu},\bar{\mu}\right)$. Together these mean that $P^2\left(\bar{\mu}\right) > P^1\left(\bar{\mu}\right) = P^0\left(\bar{\mu}\right)$.
	
	Now extending our attention to all $\xi^j$ defined on some part of $\left(-\infty,\bar{\mu}\right]$, we define
	\begin{equation*}
		M:= \left\{\mu\leq\bar{\mu}: \exists j \text{ such that } P^j\left(\mu\right) > P^0\left(\mu\right)\right\}.
	\end{equation*}
	We have just shown that $M\ne\emptyset$, so $\hat{\mu}:=\inf M$ is finite. Since $P^2$ and $P^0$ are continuous, $\hat{\mu}<\bar{\mu}$.
	
	If $\hat{\mu} \in M$, then $\max_{j}P^j\left(\mu\right)$ is discontinuous at $\mu=\hat{\mu}$. In this case we are done.
	
	If $\hat{\mu} \notin M$, then since the $P^j$ are each continuous, $\exists J$ and $\epsilon>0$ such that $P^J\left(\mu\right) > P^0\left(\mu\right)$ for $\mu\in\left(\hat{\mu},\hat{\mu} + \epsilon\right)$.
	
	Now we have to show that the derivatives of $P^0$ and $P^J$ necessarily have different limits as we take $\mu\to\hat{\mu}$ from their respective sides. First note that
	\begin{equation*}
	\lim_{\mu\uparrow\hat{\mu}}\frac{\d P^0}{\d\mu} = D^0\left(\hat{\mu}\right).
	\end{equation*}
	Note that $g^\chi\left(\delta\right)\geq g^0\left(\delta\right)$ with equality only if $\chi=0$ or if we have both $\beta = \beta^*$ and $\delta = \frac{\mu}{a}$. Therefore $D^J\ne\frac{\mu}{a}$ and
	\begin{align*}
	D^J < \frac{\mu}{a} &\implies D^J\in\left[D^2,D^1\right]\\
	D^J > \frac{\mu}{a} &\implies D^J > D^0.
	\end{align*}
	This last inequality is strict because equality could only occur at $\mu = \bar{\mu}$, but $\hat{\mu}<\bar{\mu}$.
	
	If $D^J>\frac{\mu}{a}$, then
	\begin{equation*}
	\lim_{\mu\downarrow\hat{\mu}}\frac{\d P^J}{\d\mu} = D^J\left(\hat{\mu}\right) > D^0\left(\hat{\mu}\right) = \lim_{\mu\uparrow\hat{\mu}}\frac{\d P^0}{\d\mu}
	\end{equation*}
	and we are done.
	
	From the symmetry of $g^\chi$ about $\delta = \frac{\mu}{a}$ and from $g^0$ being decreasing for $\delta \geq \frac{\mu}{a}$, we have
	\begin{equation}
	\label{RelDensity}
	D^0 - \frac{\mu}{a} < \frac{b}{a-b}\left(\frac{\mu}{a}-D^1\right)
	\end{equation}
	for $\mu \in\left[\underline{\mu},\bar{\mu}\right)$. See Figure~\ref{fig:HYL-ineq}. This implies that if $D^J<\frac{\mu}{a}$,
	\begin{equation*}
	\frac{\d P^J}{\d \mu}  = \frac{b}{a-b}\left(\frac{\mu}{b} - D^J\right) \geq \frac{b}{a-b}\left(\frac{\mu}{b} - D^1\right) > D^0 = \frac{\d P^0}{\d \mu}.
	\end{equation*}
	Taking the limit to $\hat{\mu}$ gives our result.

	\begin{figure}
		\centering
		\begin{tikzpicture}
			\draw[dashed] (2,-0.5) -- (2,4);
			\draw[thick] (2,4) to [out=270,in=40] (0,0);
			\draw[thick] (2,4) to [out=270,in=170] (8,0);
			\draw[thick] (0,0.1) -- (8,4);
			\draw[dashed] (0,0.25) -- (8,0.25);
			\draw[<->] (0.3,-0.5) -- (2,-0.5);
			\draw[<->] (2,-0.5) -- (6.75,-0.5);
			\draw[<->] (2,4) -- (3.15,4);
			\node[draw] at (1,-1) {$\frac{\mu}{a}-D^1$};
			\node[draw] at (2.5,4.5) {$D^0-\frac{\mu}{a}$};
			\node[draw] at (3.5,-1) {$\frac{b}{a-b}\left(\frac{\mu}{a}-D^1\right)$};
			\fill (0.3,0.25) circle (0.1);
			\fill (6.75,0.25) circle (0.1);
			\fill (3.15,1.65) circle (0.1);
			\draw[dashed] (0.3,-0.5) -- (0.3,0.25);
			\draw[dashed] (6.75,-0.5) -- (6.75,0.25);
			\draw[dashed] (3.15,1.65) -- (3.15,4);
		\end{tikzpicture}
		\caption{\label{fig:HYL-ineq} Sketch of part of \eqref{eqn:HYLconsistency} to show the inequality \eqref{RelDensity}. Note $\alpha\equiv0$.}
	\end{figure}
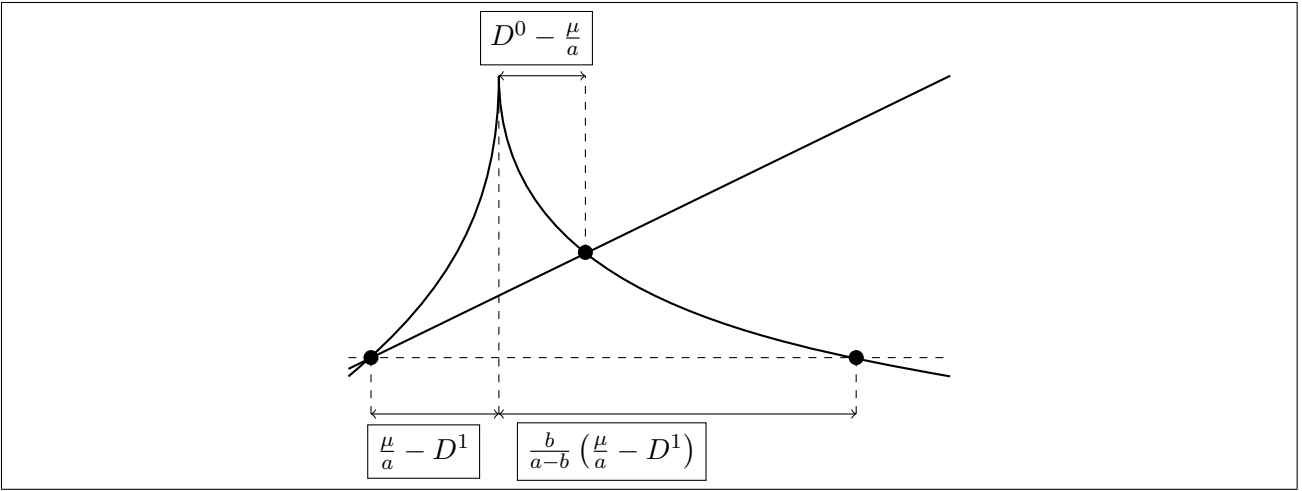
	
	\qed
\end{proofsect}

\begin{proofsect}{Proof of Theorem~\ref{THM:HYL-Cond}}
	The proof of \eqref{HYLcondenstate1} follows very similarly to the corresponding stage of the proof of Theorem \ref{thm:PMFcondensate}. Note that the HYL version of \eqref{VaradhanConstraint} follows because $H^\HYL_{\mu,l.s.c.}\geq H^\PMF_{\mu,a-b,l.s.c.}$ almost surely. The proof of \eqref{HYLcondenstate2} uses Lemma \ref{diffPressureTech}.
	\qed
\end{proofsect}

\section*{Appendices}
\begin{appendices}
\section{Processes for Dirichlet and periodic boundary conditions}\label{app-A}

\noindent For Dirichlet boundary condition, one restricts the Brownian bridges to not leaving the set $\L$. Consider the measure
\begin{equation}\label{nnBBMDir}
\mu^{\ssup{{\rm Dir},\beta}}_{x,y}(A)=\frac{\P_x(B\in A;B_\beta\in\d y)}{\d y},\qquad A\subset\Ccal^{\ssup{\rm Dir}}_{1,\L}\mbox{ measurable},
\end{equation}
which has total mass
\begin{equation}
g_\beta^{\ssup{\rm Dir}}(x,y)=\mu_{x,y}^{\ssup{\rm Dir,\beta}}(\Ccal^{\ssup{\rm Dir}}_{1,\L})=\frac{\P_x(B_{[0,\beta]}\subset\L; B_\beta\in\d y)}{\d y}.
\end{equation}

For periodic boundary condition, the marks are Brownian bridges on the  torus $\L=(\R \slash L \Z)^d$. The corresponding path measure is denoted by $\mu_{x,y}^{\ssup{\rm per,\beta}}$; its total mass is equal to 
\begin{equation}
g_\beta^{\ssup{\rm per}}(x,y)=\mu_{x,y}^{\ssup{\rm per,\beta}}(\Ccal^{\ssup{\rm per}}_\L)=\sum_{z\in\Z^d}g_{\beta}(x,y+zL)=(4\pi\beta)^{-d/2}\sum_{z\in\Z^d}{\rm e}^{-\frac{|x-y-zL|^2}{4\beta}}.
\end{equation}
For periodic and Dirichlet boundary conditions \eqref{q*def} is replaced by
\begin{equation}\label{qdefbc}
\overline{q}^{\ssup{\rm bc}}=\sum_{k=1}^N q_k^{\ssup{\rm bc}},\qquad\mbox{ with } q_k^{\ssup{\rm bc}}=\frac{1}{k|\L|}\int_\L\d x\,g_{k\beta}^{\ssup{\rm bc}}(x,x).
\end{equation}
Note that this weight depends on $ \L $ and on $N$. We introduce the Poisson point process $\omega_{\rm P} = \sum_{x \in \xi_{ \rm P}} \delta_{(x,B_x)}$ on $\L\times E^{\ssup{\rm bc}}$ with intensity measure $ \nu^{\ssup{\rm bc}} $ whose projections on $ \L\times \Ccal_{k,\L}^{\ssup{\rm bc}} $ with $k\leq N$ are equal to $ \nu_k^{\ssup{\rm bc}}(\d x,\d f)=\frac{1}{k}\Leb_\L(\d x)\otimes\mu_{x,x}^{\ssup{{\rm bc},k\beta}}(\d f)$ and are zero on this set for $k>N$. We do not label $\omega_{\rm P}$ nor $\xi_{ \rm P}$ with the boundary condition nor with $N$; $\xi_{ \rm P}$ is a Poisson process on $\L$ with intensity measure $\overline q^{\ssup{\rm bc}}$ times the restriction $\Leb_\L$ of the Lebesgue measure to $\L$. By ${\tt Q}^{\ssup{\rm bc}}$ and ${\tt E}^{\ssup{\rm bc}}$ we denote probability and expectation with respect to this process. Conditionally on $\xi_{ \rm P}$, the lengths of the cycles $B_x$ with $x\in\xi_{ \rm P}$ are independent and have distribution $(q_k^{\ssup{\rm bc}}/\overline q^{\ssup{\rm bc}})_{k\in\{1,\dots,N\}}$; this process has only marks with lengths $\leq N$. A cycle $B_x$ of length $k$ is distributed according to
\begin{equation}\label{pathprobmeasure}
\P_{x,x}^{\ssup{{\rm bc},k\beta}}(\d f)=\frac{\mu_{x,x}^{\ssup{{\rm bc},k\beta}}(\d f)}{g^{\ssup{\rm bc}}_{k\beta}(x,x)}.
\end{equation}

The above representations allows us to prove the large deviation principles for different boundary conditions. For details we refer to \cite{ACK} where these arguments are presented in detail for higher level large deviation principles. The independence of the thermodynamic limit of the average finite-volume pressure in Theorem~\ref{THM-pressure} follows using either the arguments in \cite{ACK} or in \cite{R71,AN73,BR97}.
\section{Bose function}\label{app-Bose}
The Bose functions are poly-logarithmic functions defined by 
\begin{equation}\label{defBosefunctionapa}
g(n,\alpha):=\Li_n(\ex^{-\alpha})=\frac{1}{\Gamma(n)}\int_0^\infty\frac{t^{n-1}}{{\rm e}^{t+\alpha}-1}\,\d t=\sum_{k=1}^\infty k^{-n}{\rm e}^{-\alpha k}\quad\mbox{ for all } n \mbox{ and } \alpha>0,
\end{equation}
and also for $ \alpha=0 $ and $ n>1 $. In the latter case,
\begin{equation}\label{zeta}
g(n,0)=\sum_{k=1}^\infty k^{-n}=\zeta(n),
\end{equation}
which is the zeta function of Riemann. The behaviour of the Bose functions about $ \alpha=0 $ is given by
\begin{equation}\label{boseexp}
g(n,\alpha)=\begin{cases}
\Gamma(1-n)\alpha^{n-1}+\sum_{k=0}^\infty\zeta(n-k)\frac{(-\alpha)^k}{k!} &, n\not= 1,2,3,\ldots, \\[1.5ex]
\frac{(-\alpha)^{n-1}}{(n-1)!}\Big[-\log\alpha +\sum_{m=1}^{n-1}\frac{1}{m}\Big]+\sum_{\heap{k=0}{k\not= n-1}}\zeta(n-k)\frac{(-\alpha)^k}{k!} &, n\in\N.                   
\end{cases}
\end{equation}
At $ \alpha=0 $, $ g(n,\alpha) $ diverges for $ n\le 1 $, indeed for all $n$ there is some kind of singularity at $ \alpha=0 $, such as a branch point. For further details see \cite{Gram25}. The expansions \eqref{boseexp} are in terms of $ \zeta(n) $, which for for $ n\le 1 $ must be found by analytic ally continuing \eqref{zeta}. With the asymptotic properties of the zeta function it can be shown that the $k$ series in \eqref{boseexp} are convergent for $ |\alpha|<2\pi $. Consequently \eqref{boseexp} also represents an analytic continuation of $ g(n,\alpha) $ for $ \alpha<0 $. When $ \alpha\gg 1 $ the series \eqref{defBosefunctionapa} itself is rapidly convergent, and as $ \alpha\to\infty $, $ g(n,\alpha)\sim\ex^{-\alpha} $ for all $n$. Some plots for the Bose functions are given in Figure~\ref{BoseFCT}.

\begin{figure}[H]
\includegraphics[scale=1]{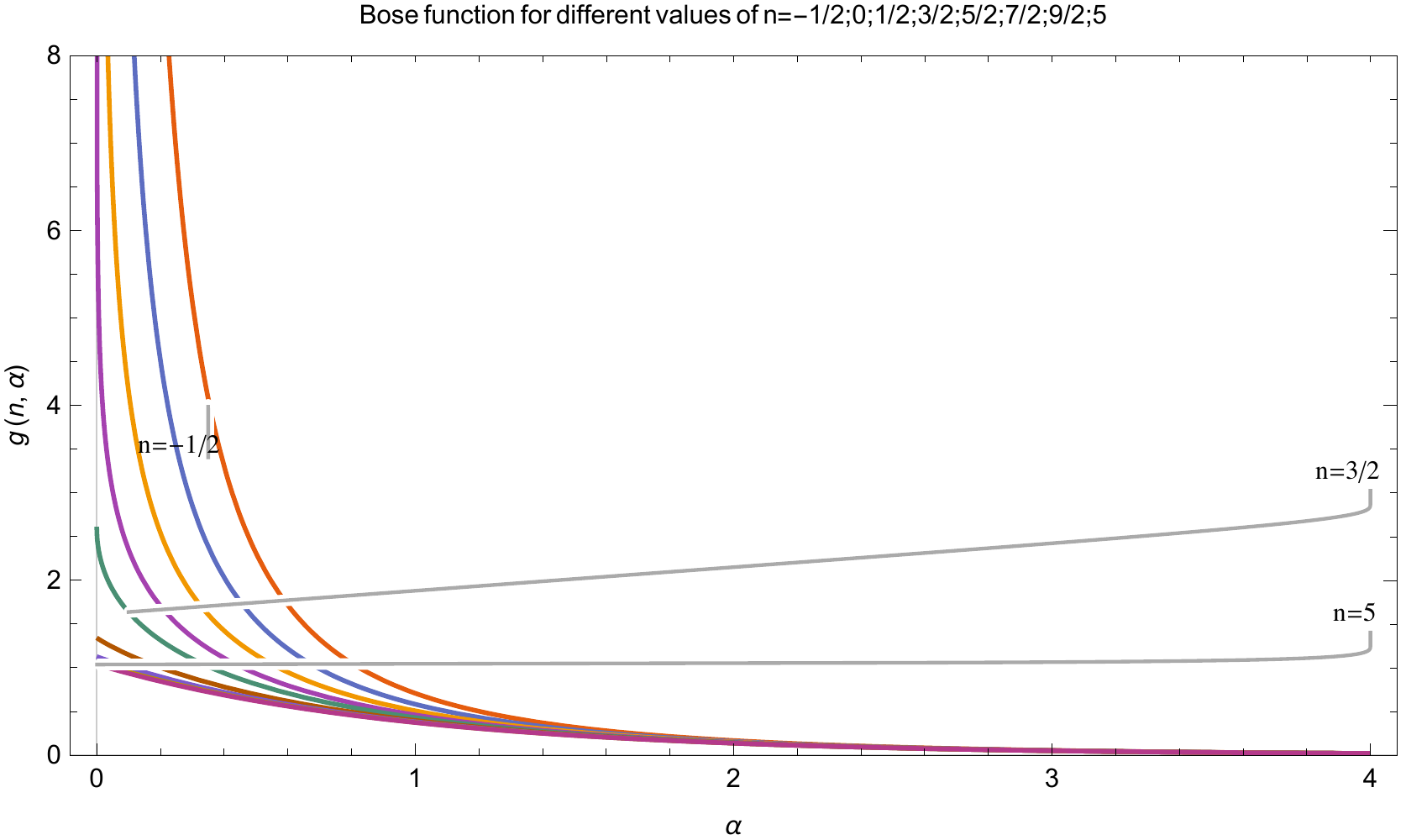}
\caption{}\label{BoseFCT}
\end{figure}

\section{Lambert W function}\label{app-Lambert}

The Lambert W function (sometimes called elsewhere the \emph{Omega function}) is defined as the multi-valued inverse of the $\mathbb{C}\rightarrow\mathbb{C}$ function $w\mapsto w\ex^{w}$. We shall only be concerned with the two branches on $\R$. Figure \ref{fig:lambertW} shows these two real branches, denoted $W_0$ and $W_{-1}$. The $W_0$ branch is defined on $\left[-\ex^{-1},\infty\right)$, whereas the $W_{-1}$ branch is only defined on $\left[-\ex^{-1},0\right)$. Given a branch $W_l$ with $l\in\left\{0,-1\right\}$, we can find its (real) derivative $W'_l$ by differentiating the equation $W_l\left(x\right)\ex^{W_l\left(x\right)} = x$. This gives us
\begin{equation*}
W'_l\left(x\right) = \frac{1}{x}\frac{W_l\left(x\right)}{1+ W_l\left(x\right)}.
\end{equation*}
Taking further derivatives and applying induction shows that the branches are smooth on the interior of their respective domains, and gives expressions for each order of the derivative. We make use of some asymptotic expansions of $W_0$ and $W_1$:
\begin{equation*}
\begin{alignedat}{2}
W_0\left(x\right) &= x - x^2 + o\left(x^2\right)  & & \quad\text{as }x\to 0,\\
W_0\left(x\right) &= \log x - \log \left(\log x\right) + o\left(1\right) & &\quad\text{as }x\to+\infty,\\
W_{-1}\left(x\right) &= \log\left(-x\right) - \log\left(-\log\left(-x\right)\right) + o\left(1\right) & & \quad\text{as }x\uparrow 0.\\
\end{alignedat}
\end{equation*}
For more details, see \cite{CGHJK96}.

\begin{figure}
	\centering
	\begin{tikzpicture}[xscale=1.5]
	\draw[->] (-1.1,0) -- (3.1,0) node[below]{$x$};
	\draw[<-] (0,1.6) node[left]{$W\left(x\right)$} -- (0,-4.1);
	\draw[very thick] (-0.37,-1) to [out=90,in=225] (0,0) to [out=45,in=190] (3,1.5) node[right]{$W_0$};
	\draw[very thick, dashed] (-0.37,-1) to [out=270,in=90] (-0.05,-4) node[left]{$W_{-1}$};
	\draw[dashed] (-0.37,0) node[above]{$-1/e$} -- (-0.37,-1) -- (0,-1) node[right]{$-1$};
	\end{tikzpicture}
	\caption{\label{fig:lambertW} The two real branches of $W$: $W_0$ and $W_{-1}$.}
\end{figure}
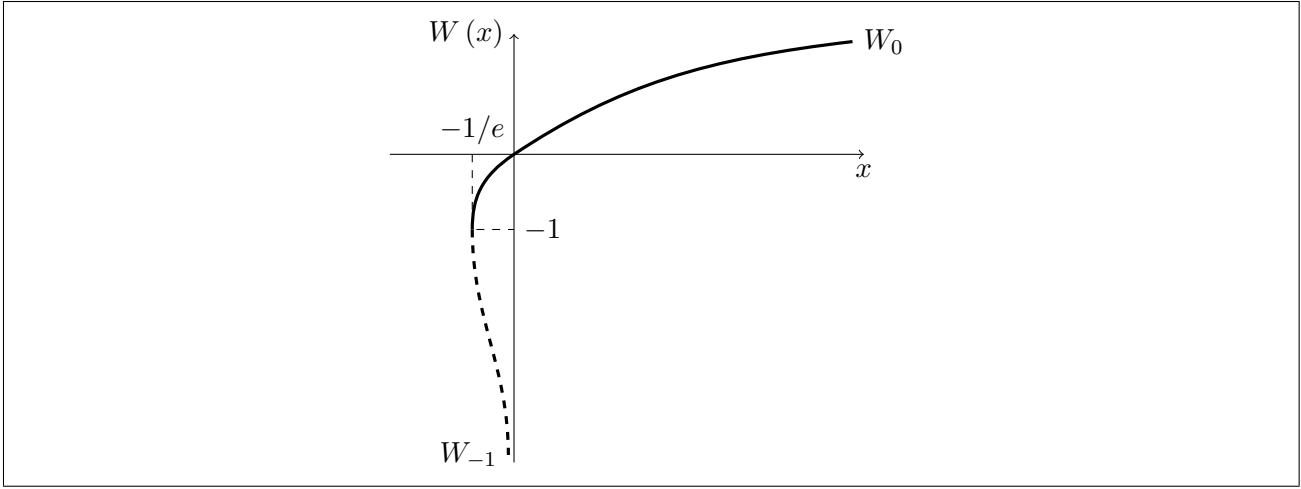
\end{appendices}

\end{document}